\documentclass[11pt]{article}

\usepackage[a4paper,margin=1in]{geometry}

\usepackage{amssymb}
\usepackage{mathtools}
\usepackage{amsthm}
\mathtoolsset{showonlyrefs}     

\usepackage{graphicx}
\usepackage{subcaption}

\usepackage{enumitem}

\usepackage[colorlinks=true,
            linkcolor=blue,
            citecolor=blue,
            urlcolor=blue]{hyperref}

\numberwithin{equation}{section}

\theoremstyle{plain}
\newtheorem{theorem}{Theorem}[section]
\newtheorem{lemma}[theorem]{Lemma}
\newtheorem{proposition}[theorem]{Proposition}
\newtheorem{corollary}[theorem]{Corollary}

\newtheorem{assumption}{Assumption}

\theoremstyle{remark}
\newtheorem{remark}{Remark}

\title{Convergence of Markovian Iteration for $Z$-Coupled FBSDEs via Gaussian Smoothing}
\author{
Zhipeng Huang 
\thanks{Corresponding author. Mathematical Institute, Utrecht University, Utrecht, The Netherlands (\texttt{z.huang1@uu.nl}).}
\and
Cornelis W.~Oosterlee
\thanks{Mathematical Institute, Utrecht University, Utrecht, The Netherlands (\texttt{c.w.oosterlee@uu.nl}).}}

\date{}

\begin{document}

\maketitle

\begin{abstract}
In this paper, we investigate the Markovian iteration method for coupled forward-backward stochastic differential equations (FBSDEs) with drift $b(t,X_t,Y_t,Z_t)$ and deterministic, time-dependent diffusion. Bender and Zhang (2008) established convergence results for Markovian iteration in the $Y$-coupled setting. Extending these results to equations with $Z$-coupling presents additional challenges, particularly in obtaining uniform control of the spatial Lipschitz constants of the discrete decoupling fields across iterations and time steps.

We address this difficulty by representing the discrete $Z$-field through the spatial derivative of the Gaussian-smoothed $Y$-field. This representation, combined with weak differentiation, yields uniform spatial Lipschitz bounds for both discrete decoupling fields without requiring their classical second derivatives. Under suitable contraction conditions, we prove convergence of the Markovian iteration and well-posedness of the discrete scheme. With additional regularity assumptions, we derive an overall error bound accounting for both iteration and time discretization. Numerical experiments compare three practical implementations and demonstrate the effectiveness of the derivative-based and merged approaches.
\end{abstract}


\medskip


\tableofcontents 

\section{Introduction}\label{p2-sec1}
 
In this paper, we study the numerical solution of a system of coupled forward-backward stochastic differential equations (FBSDEs) on a complete probability space $(\Omega,\mathcal F,\mathbb P)$ equipped with the augmented natural filtration $(\mathcal F_t)_{0\leq t\leq T}$ of a $d_3$-dimensional Brownian motion $W \coloneqq \{W_t\}_{0\leq t \leq T}$,
\begin{equation}\label{p2-eq:continuousFBSDE} 
\left\{
\begin{aligned}
X_t & = x_0 + \int_0^t b(s, X_s, Y_s, Z_s) \mathrm{d}s + \int_0^t \sigma(s) \mathrm{d}W_s,\\
Y_t & = g(X_T) + \int_t^T f(s, X_s, Y_s, Z_s) \mathrm{d}s - \int_t^T Z_s \mathrm{d} W_s,
\end{aligned}
\right.
\end{equation}
where $(X, Y, Z) \coloneqq \{(X_t, Y_t, Z_t)\}_{0\leq t \leq T}$ is a triplet of $(\mathbb{R}^{d_1} \times \mathbb{R}^{d_2} \times \mathbb{R}^{{d_2} \times {d_3}})$-valued and adapted stochastic processes. The functions $b: [0,T]\times \mathbb{R}^{d_1} \times \mathbb{R}^{d_2} \times \mathbb{R}^{d_2 \times d_3} \to \mathbb{R}^{d_1}$, $\sigma: [0,T] \to \mathbb{R}^{d_1 \times d_3}$, $f:[0,T]\times \mathbb{R}^{d_1} \times \mathbb{R}^{d_2} \times \mathbb{R}^{d_2 \times d_3} \to \mathbb{R}^{d_2}$, and $g: \mathbb{R}^{d_1} \to \mathbb{R}^{d_2}$ are deterministic measurable mappings. The triplet $(X,Y,Z)$ is a solution if it satisfies the required measurability and integrability conditions and \eqref{p2-eq:continuousFBSDE} holds $\mathbb P$-almost surely for all $t\in[0,T]$; see \cite{zhang2017backward,ma1999forward}. Furthermore, the solution $(X, Y, Z)$ is linked to a system of parabolic PDEs through the nonlinear Feynman--Kac formula under suitable conditions,
\begin{equation}\label{p2-eq:quasi-linear-pde}
\left\{
\begin{aligned}
&\partial_t u^i
+\frac12\partial_{xx}u^i:
\bigl(\sigma(t)\sigma(t)^\top\bigr)
+(\nabla_xu^i)^\top
b\bigl(t,x,u,(\nabla_xu)^\top\sigma(t)\bigr)
\\
&\qquad
+f^i\bigl(t,x,u,(\nabla_xu)^\top\sigma(t)\bigr)=0,
\qquad i=1,\ldots,d_2,
\\
&u(T,x)=g(x).
\end{aligned}
\right.
\end{equation}
with
\begin{equation}\label{p2-eq:decouple}
Y_t=u(t,X_t),
\qquad
Z_t=(\nabla_xu(t,X_t))^\top\sigma(t)
\coloneqq v(t,X_t),
\end{equation}
and the mappings $u$ and $v$ are referred to as decoupling fields in the FBSDEs literature.

The study of FBSDEs, or BSDEs in the decoupled case, dates back to the seminal work in \cite{bismut1973conjugate} and later \cite{pardoux1990adapted}, who first investigated general nonlinear BSDEs. Due to their connection with quasi-linear PDEs through the nonlinear Feynman-Kac formula, FBSDEs have broad applications in mathematical finance, physics, and stochastic control. Significant progress has been made in their solution theory following \cite{pardoux1990adapted}, including the method of contraction mapping \cite{antonelli_backward-forward_1993, pardoux1999forward}, which establishes well-posedness under standard assumptions when $T$ is small; the four-step scheme of \cite{ma_solving_1994}, which removes the restriction on $T$ under certain regularity conditions; and the method of continuation \cite{hu1995solution, peng1999fully}, which extends to non-Markovian FBSDEs under a different set of assumptions. For a comprehensive discussion of these methods, see \cite{ma1999forward}.

Despite the rich theoretical background, numerical methods for FBSDEs remain a highly relevant area of research since these equations are often analytically intractable. Various methods have been developed for the decoupled case, including the Malliavin calculus method \cite{bouchard2004malliavin, bouchard2004discrete, crisan2010monte}, the quantization method \cite{beck2020overview, pages2018improved}, and regression-based approaches \cite{lemor2006rate, gobet2005regression}. For the coupled case, numerical approximation is more challenging due to the intertwined nature of the forward and backward equations. Notable methods include the four-step scheme-based approach \cite{milstein_numerical_2006}, the Markovian iteration scheme \cite{bender2008time}, and Fourier expansion techniques \cite{huijskens_efficient_2016, negyesi2025numerical}. More recently, neural network-based algorithms, such as the Deep BSDE method \cite{e_deep_2017, han2018, hanlong2020}, have gained popularity due to their high accuracy and effectiveness in handling high-dimensional problems. Establishing convergence of the deep BSDE method requires well-posedness and time-discretization error estimates for the underlying discrete FBSDE. The analysis in \cite{hanlong2020} builds on the results and techniques of \cite{bender2008time}, which treat forward coefficients independent of the $Z$ process. Consequently, their convergence analysis is restricted to the $Y$-coupled setting. More recently, \cite{reisinger2023} used the method of continuation under monotonicity conditions to establish well-posedness and stability of fully coupled discrete FBSDEs, including McKean--Vlasov systems, and derive a posteriori error estimates for given approximations. Instead, our work uses different assumptions and proof techniques to establish convergence of the Markovian iteration under suitable contraction conditions, as detailed below.

In this article, we study a Markovian iteration scheme for FBSDEs with fully coupled drift $b(t,X_t,Y_t,Z_t)$ and time-dependent diffusion, using a backward step whose driver is implicit in $Y$. The additional $Z$-dependence in the forward equation is not considered in \cite{bender2008time}. The main difficulty is to control the spatial Lipschitz constants of the discrete $Z$-decoupling fields uniformly across time steps and iterations. Our main analytical tool is the Gaussian smoothing operator, also known as Gaussian convolution operator, which represents the discrete $Z$-field through the spatial derivative of the Gaussian-smoothed $Y$-field. Combining its smoothing properties with weak differentiation, the backward recursion, and terminal regularity, we establish an invariant Lipschitz class for both discrete decoupling fields, without requiring classical second derivatives of the iterated fields. These uniform bounds allow us to prove geometric convergence of the Markovian iteration with respect to the iteration index $m$, together with existence and uniqueness of a square-integrable solution to the discrete scheme under suitable contraction conditions. Under additional regularity assumptions, we also derive a time-discretization error estimate.

Our backward scheme evaluates the driver implicitly at $Y_i$, whereas \cite{bender2008time} evaluates it at $Y_{i+1}$. This implicit formulation is also used in the reference scheme underlying the convergence analysis of the deep BSDE method in \cite{hanlong2020}. Although we do not develop a deep BSDE method here, our results provide a theoretical building block for such analyses in the presence of $Z$-dependent forward drifts and deterministic, time-dependent diffusion. The implicit formulation also facilitates the numerical implementation of the Markovian iteration in Section~\ref{p2-sec4}.

The remainder of this paper is organized as follows. Section~\ref{p2-sec2} introduces the Markovian iteration scheme considered here, and its Gaussian smoothing representation of the discrete $Z$-field. Section~\ref{p2-sec3} presents the required assumptions and preliminary estimates, including the properties of the Gaussian smoothing operator. It establishes uniform Lipschitz and growth bounds, proves geometric convergence of the Markovian iteration and well-posedness of the discrete FBSDE, and derives time-discretization and overall error estimates. Section~\ref{p2-sec4} presents three numerical implementations and compares their performance through experiments with varying time-step sizes and numbers of Markovian iterations. Finally, Section~\ref{p2-sec:conclusion} concludes the paper.

\section{The Markovian iteration for FBSDEs with $Z$-coupling}\label{p2-sec2}

We now introduce the Markovian iteration scheme for the FBSDE \eqref{p2-eq:continuousFBSDE}. For simplicity, we focus on the one-dimensional setting for the discussion and the convergence analysis.

Let $m$ denote the iteration index and $N$ the number of time steps. We consider the uniform grid $\pi=\{t_i=ih:i=0,1,\ldots,N\}$ with step size $h\coloneqq T/N$, and write $\Delta W_i\coloneqq W_{t_{i+1}}-W_{t_i}$, $\mathbb{E}_{t_i}[\cdot]\coloneqq\mathbb{E}[\cdot\mid\mathcal F_{t_i}]$. Following the approach in \cite{bender2008time}, we start from the initialization $u_i^{\pi,0}=v_i^{\pi,0}=0$ for $i=0,\ldots,N-1$, $u_N^{\pi,0}=g$, and define, for $m\geq1$,
\begin{equation}\label{p2-eq:implicit_full_scheme}
\left\{
\begin{aligned}
X_0^{\pi,m}
&=x_0,
\\
X_{i+1}^{\pi,m}
&=
X_i^{\pi,m}
+h\,b\bigl(
t_i,
X_i^{\pi,m},
u_i^{\pi,m-1}(X_i^{\pi,m}),
v_i^{\pi,m-1}(X_i^{\pi,m})
\bigr)
+\sigma(t_i)\Delta W_i,
\\
Y_N^{\pi,m}
&=g(X_N^{\pi,m}),
\qquad u_N^{\pi,m}=g,
\\
\widehat Z_i^{\pi,m}
&=
\frac1h\,\mathbb{E}_{t_i}
\left[
Y_{i+1}^{\pi,m}\Delta W_i
\right],
\\
Y_i^{\pi,m}
&=
\mathbb{E}_{t_i}\left[Y_{i+1}^{\pi,m}\right]
+h\,f\bigl(
t_i,
X_i^{\pi,m},
Y_i^{\pi,m},
\widehat Z_i^{\pi,m}
\bigr),
\\
u_i^{\pi,m}(X_i^{\pi,m})
&=Y_i^{\pi,m},
\\
v_i^{\pi,m}(X_i^{\pi,m})
&=\widehat Z_i^{\pi,m}.
\end{aligned}
\right.
\end{equation}
The forward recursion is performed for $i=0,\ldots,N-1$, followed by the backward recursion for $i=N-1,\ldots,0$. Throughout, the discrete decoupling fields $u_i^{\pi,m}$ are defined for $0\leq i\leq N$, with $u_N^{\pi,m}=g$ for every $m\geq0$, whereas $v_i^{\pi,m}$ is defined only for $0\leq i<N$. The fields $u_i^{\pi,m}$ and $v_i^{\pi,m}$ approximate $u(t_i,\cdot)$ and $v(t_i,\cdot)$, respectively.

The identities in \eqref{p2-eq:implicit_full_scheme} specify the decoupling fields along the forward process. For the convergence analysis in the later section, we shall define $u_i^{\pi,m}$ and $v_i^{\pi,m}$ in a pointwise manner. To this end, fix $i\in\{0,\ldots,N-1\}$ and $x\in\mathbb R$, and restart the scheme at time $t_i$ from $X_i^{\pi,m}=x$, keeping the previous-iteration input fields fixed. We define $u_i^{\pi,m}(x)$ and $v_i^{\pi,m}(x)$ as the resulting initial backward values $Y_i^{\pi,m}$ and $\widehat Z_i^{\pi,m}$, respectively. These values are deterministic and agree with the identities in \eqref{p2-eq:implicit_full_scheme} when evaluated along the original forward process. Together with $u_N^{\pi,m}=g$, this construction defines the fields pointwise and facilitates the subsequent regularity analysis.

In addition to the $Z$-coupling, our scheme differs from \cite{bender2008time} in that the backward equation in \eqref{p2-eq:implicit_full_scheme} is implicit in $Y_i^{\pi,m}$. At each backward step, once the forward process, $Y_{i+1}^{\pi,m}$, and $\widehat Z_i^{\pi,m}$ are known, $Y_i^{\pi,m}$ is determined by solving
\begin{equation}
Y_i^{\pi,m}
=
\mathbb{E}_{t_i}\left[Y_{i+1}^{\pi,m}\right]
+h\,f\left(
t_i,
X_i^{\pi,m},
Y_i^{\pi,m},
\widehat Z_i^{\pi,m}
\right).
\end{equation}
This is the usual implicit Euler step for a decoupled FBSDE. If $f$ is globally Lipschitz continuous in $y$, uniformly in $(t,x,z)$, and $h$ is sufficiently small that $h\operatorname{Lip}(f(t,x,\cdot,z)) <1$ for all $(t,x,z)$, the right-hand side defines a contraction in $Y_i^{\pi,m}$, and thus the equation admits a unique solution. Here and below, we denote by $\operatorname{Lip}(\phi)$ the global Lipschitz constant of a function $\phi$. We impose this condition $h\operatorname{Lip}(f(t,x,\cdot,z))<1$ throughout the article without further notice. The advantages of this implicit formulation for numerical implementation will be used in Section~\ref{p2-sec4}.

The contraction-based strategy in \cite{bender2008time} relies on uniform bounds for the spatial Lipschitz constants of the discrete $Y$-decoupling fields. Since successive iterates are evaluated along different forward processes, these bounds allow us to separate the effect of changing the forward state from the difference between the fields, with constants independent of the iteration index. Together with growth estimates and a suitable contraction condition, this yields summable bounds on successive differences, showing that $\{u_i^{\pi,m}(x)\}_{m\geq1}$ is a Cauchy sequence for each fixed partition $\pi$, time index $i$, and spatial point $x$. When the forward coefficients do not depend on $Z$ as in \cite{bender2008time}, differences between the $Z$-processes can be controlled through backward stability estimates, without requiring uniform spatial Lipschitz bounds for the fields $v_i^{\pi,m}$. In the present setting, the additional $Z$-dependence of the forward drift requires us to extend this strategy by establishing such bounds for both $u_i^{\pi,m}$ and $v_i^{\pi,m}$.

Our approach to controlling the spatial Lipschitz constants of $v_i^{\pi,m}$ exploits the Gaussian structure of the one-step forward transition. To describe this connection, we introduce the Gaussian smoothing operator
\begin{equation}
(\mathcal G_h^\varsigma \phi)(a)
\coloneqq
\mathbb E\left[\phi(a+\varsigma\sqrt h\,G)\right],
\qquad G\sim\mathcal N(0,1),
\end{equation}
for a fixed parameter $\varsigma\neq0$ and functions $\phi$ for which the expectation is finite.

For a fixed time index $i$ and state $x\in\mathbb R$, the conditional distribution of $X_{i+1}^{\pi,m}$ given $X_i^{\pi,m}=x$ is Gaussian with variance $\sigma(t_i)^2h$, and mean
\begin{equation}
\label{p2-eq:implicit_forward_conditional_mean}
A_i^{\pi,m}(x)
\coloneqq
\mathbb E\left[
X_{i+1}^{\pi,m}\mid X_i^{\pi,m}=x
\right]
=
x+h\,b\bigl(
t_i,x,u_i^{\pi,m-1}(x),v_i^{\pi,m-1}(x)
\bigr).
\end{equation}
Thus, the operator with $\varsigma=\sigma_i$, with the convention that $\sigma_i \coloneqq \sigma(t_i)$, describes Gaussian averaging around this conditional mean. If $u_{i+1}^{\pi,m}$ is globally Lipschitz continuous, the Gaussian differentiation formula yields
\begin{equation}\label{p2-eq:v_gaussian_general}
\begin{aligned}
v_i^{\pi,m}(x)
&=
\frac1h\mathbb E\left[
u_{i+1}^{\pi,m}(X_{i+1}^{\pi,m})\Delta W_i
\,\middle|\,X_i^{\pi,m}=x
\right]
\\
&=
\left.
\sigma_i\,\partial_a
\mathbb E\left[
u_{i+1}^{\pi,m}(a+\sigma_i\Delta W_i)
\right]
\right|_{a=A_i^{\pi,m}(x)}
\\
&=
\sigma_i\,
\bigl(\partial_a \mathcal G_h^{\sigma_i}u_{i+1}^{\pi,m}\bigr)
\bigl(A_i^{\pi,m}(x)\bigr).
\end{aligned}
\end{equation}
This representation expresses the discrete $Z$-field through the classical derivative of a Gaussian-smoothed $Y$-field, even when the original $Y$-field is not classically differentiable. It provides the starting point for our spatial regularity estimates. This Gaussian smoothing operator acts by convolution with a Gaussian density and is therefore also referred to as a Gaussian convolution operator. Its relevant properties will be stated in Proposition~\ref{p2-prop:gaussian_semigroup}.

With additional regularity, for example if $u_{i+1}^{\pi,m}\in C^1(\mathbb R)$ with bounded derivative, we may differentiate under the expectation to obtain the classical Gaussian integration-by-parts representation,
\begin{equation}\label{p2-eq:v_gaussian_classcial}
v_i^{\pi,m}(x)
=
\sigma_i\,\mathbb E\left[
\partial_xu_{i+1}^{\pi,m}(X_{i+1}^{\pi,m})
\,\middle|\,X_i^{\pi,m}=x
\right].
\end{equation}
Using this representation to estimate $|v_i^{\pi,m}(x)-v_i^{\pi,m}(y)|$ directly involves controlling differences of $\partial_xu_{i+1}^{\pi,m}$ at two spatial points. In particular, if both $A_i^{\pi,m}$ and $\partial_xu_{i+1}^{\pi,m}$ are globally Lipschitz, a direct estimate gives $\operatorname{Lip}(v_i^{\pi,m}) \leq |\sigma_i|\operatorname{Lip}(A_i^{\pi,m}) \operatorname{Lip} (\partial_xu_{i+1}^{\pi,m} )$. With the first two factors uniformly controlled, obtaining a uniform bound through this estimate would require uniform control of $\operatorname{Lip}(\partial_xu_{i+1}^{\pi,m})$. This is stronger than the existence and boundedness of the first derivative, and establishing such regularity uniformly in the time-step size, time index, and iteration index is not straightforward. Nevertheless, \eqref{p2-eq:v_gaussian_classcial} remains useful for the numerical implementation in Section~\ref{p2-sec4}.

For the convergence analysis in Section \ref{p2-sec3}, we instead use the Gaussian smoothing representation \eqref{p2-eq:v_gaussian_general} to control the spatial Lipschitz constant of $v_i^{\pi,m}$ through the second derivative of $\mathcal G_h^{\sigma_i}u_{i+1}^{\pi,m}$. By Gaussian smoothing, this derivative exists classically even when $u_{i+1}^{\pi,m}$ is only globally Lipschitz. Using weak differentiation, the additional terminal regularity, together with a transition-operator estimate for bounded measurable functions, we establish $\operatorname{Lip}(v_i^{\pi,m})\leq C$ uniformly in the time-step size, the time index and the iteration index. Together with the corresponding spatial Lipschitz and growth bounds for the $Y$-fields, these estimates allow us to prove convergence of the Markovian iteration under a suitable contraction condition. The argument follows the same core strategy as \cite{bender2008time}, establishing convergence to a fixed point, but uses different estimates that lead to a comparatively simple contraction condition in the present setting.

\section{Convergence analysis}\label{p2-sec3}

In this section, we analyze the convergence of the scheme using the Gaussian smoothing operator. Section~\ref{p2-sec31} presents the assumptions and preliminary estimates, including the properties of this operator. Section~\ref{p2-sec32} establishes uniform spatial Lipschitz bounds for the decoupling fields and a linear-growth estimate for the first iteration. Section~\ref{p2-sec33} proves geometric convergence of the Markovian iteration under a suitable contraction condition. Finally, Section~\ref{p2-sec34} establishes the time-discretization error and combines it with the iteration error to obtain an overall error estimate.

\subsection{Preliminaries}\label{p2-sec31}

We impose the following standard assumptions, together with additional regularity conditions on the diffusion coefficient $\sigma$ and terminal function $g$ needed for our Gaussian smoothing approach.

\begin{assumption}\label{p2-ass:implicit_lipschitz}
The coefficient functions $b,f$, and $g$ satisfy the following conditions.

\begin{enumerate}[label=(\roman*)]

\item
There exist nonnegative constants $b_x,b_y,b_z,f_x,f_y,f_z,g_x$ such that, for all $t\in[0,T]$ and $x_1,x_2,y_1,y_2,z_1,z_2\in\mathbb R$,
\begin{equation}\label{p2-eq:implicit_lipschitz_coefficients}
\begin{aligned}
|b(t,x_1,y_1,z_1)-b(t,x_2,y_2,z_2)|^2
&\leq
b_x|x_1-x_2|^2
+
b_y|y_1-y_2|^2
+
b_z|z_1-z_2|^2,
\\
|f(t,x_1,y_1,z_1)-f(t,x_2,y_2,z_2)|^2
&\leq
f_x|x_1-x_2|^2
+
f_y|y_1-y_2|^2
+
f_z|z_1-z_2|^2,
\\
|g(x_1)-g(x_2)|^2
&\leq
g_x|x_1-x_2|^2.
\end{aligned}
\end{equation}

\item
There exists $k_f\in\mathbb R$ such that, for all $t\in[0,T]$ and $x,y_1,y_2,z\in\mathbb R$,
\begin{equation}
\label{p2-eq:f_monotonicity}
\bigl[
f(t,x,y_1,z)-f(t,x,y_2,z)
\bigr](y_1-y_2)
\leq
k_f|y_1-y_2|^2.
\end{equation}

\item
There exist nonnegative constants $b_0,f_0,g_0$ such that, for all $t\in[0,T]$ and $x,y,z\in\mathbb R$,
\begin{equation}
\label{p2-eq:implicit_linear_growth_coefficients}
\begin{aligned}
|b(t,x,y,z)|^2
&\leq
b_0+b_x|x|^2+b_y|y|^2+b_z|z|^2,
\\
|f(t,x,y,z)|^2
&\leq
f_0+f_x|x|^2+f_y|y|^2+f_z|z|^2,
\\
|g(x)|^2
&\leq
g_0+g_x|x|^2.
\end{aligned}
\end{equation}

\end{enumerate}
\end{assumption}

\begin{assumption}\label{p2-ass:constant_sigma}
The diffusion coefficient depends only on time, and there exist positive constants $\underline\sigma$ and $\overline\sigma$  such that $\underline\sigma\leq|\sigma(t)|\leq\overline\sigma$ for $t\in[0,T].$
\end{assumption}

\begin{assumption}\label{p2-ass:terminal_regular}
The terminal condition $g\in C^{1,1}(\mathbb R)$.
\end{assumption}

To analyze the behavior of the Markovian iteration, we introduce the operator associated with one full iteration. Given two families of deterministic input fields $\varphi=\{\varphi_i\}_{i=0}^{N-1}$ and $\xi=\{\xi_i\}_{i=0}^{N-1}$, let $F^\pi(\varphi,\xi)=(\Phi,\psi)$ denote the pair of discrete decoupling fields obtained by performing one iteration of the scheme \eqref{p2-eq:implicit_full_scheme}, with $(u^{\pi,m-1},v^{\pi,m-1})$ replaced by $(\varphi,\xi)$ and the newly generated fields $(u^{\pi,m},v^{\pi,m})$ denoted by $(\Phi,\psi)$. The terminal output is prescribed by $\Phi_N=g$; only the components with $i<N$ are used as inputs to $F^\pi$. With this convention, the Markovian iteration \eqref{p2-eq:implicit_full_scheme} can be written compactly as
\begin{equation}\label{p2-eq:markovian_operator_iteration}
(u^{\pi,m},v^{\pi,m})
=
F^\pi\left(u^{\pi,m-1},v^{\pi,m-1}\right),
\qquad m\geq1.
\end{equation}

Throughout the paper, for a globally Lipschitz function $\phi$, we write $ L(\phi)\coloneqq\operatorname{Lip}(\phi)^2$. For a finite family of such functions $\phi=\{\phi_i\}_{i\in I}$, we use the same notation to denote the largest squared Lipschitz constant: $ L(\phi)\coloneqq\max_{i\in I}L(\phi_i), $ where $I$ is the relevant set of time indices.

We next state the implicit analogue of the one-step backward stability estimate in \cite{bender2008time}, with the proof deferred to the appendix. To state the following lemma fully explicitly, choose auxiliary parameters $\lambda_x>0$, $\lambda_z>f_z$, and a sufficiently small $h_0>0$, and define
\begin{equation}
\label{p2-eq:implicit_backward_constants}
\theta_0\coloneqq 2k_f+\lambda_x+\lambda_z,
\qquad
\theta_1\coloneqq
\frac{1-f_z/\lambda_z}{1+|\theta_0|h_0},
\qquad
\theta_2\coloneqq\frac{\theta_0^+}{1-\theta_0^+h_0},
\qquad
\theta_3\coloneqq
\frac{f_x}{\lambda_x(1-\theta_0^+h_0)},
\end{equation}
where the notion $a^+\coloneqq\max\{a,0\}$ and $h_0$ is chosen so that $1-\theta_0^+h_0>0$.

\begin{lemma}[Implicit backward stability]
\label{p2-lem:implicit_backward_stability}
Let Assumption~\ref{p2-ass:implicit_lipschitz} hold. Choose auxiliary parameters $\lambda_x>0$ and $\lambda_z>f_z$, and $h_0>0$ such that $1 - \theta_0^+ h_0 > 0$.

For $\ell=1,2$, let $X_i^{(\ell)},Y_i^{(\ell)}\in L^2(\mathcal F_{t_i})$ and $Y_{i+1}^{(\ell)}\in L^2(\mathcal F_{t_{i+1}})$ satisfy
\begin{equation}\label{p2-eq:two_implicit_backward}
\begin{aligned}
\widehat Z_i^{(\ell)}
&=
\frac1h\mathbb E_{t_i}
\left[Y_{i+1}^{(\ell)}\Delta W_i\right],
\\
Y_i^{(\ell)}
&=
\mathbb E_{t_i}\left[Y_{i+1}^{(\ell)}\right]
+
h f\left(
t_i,X_i^{(\ell)},Y_i^{(\ell)},\widehat Z_i^{(\ell)}
\right),
\end{aligned}
\end{equation}
using the same Brownian increment $\Delta W_i$. Then there exist constants $\theta_2,\theta_3\geq0$ and $\theta_1>0$, independent of $h$, such that, for every $0<h\leq h_0$, the following holds,
\begin{equation}\label{p2-eq:implicit_backward_one_step}
|\Delta Y_i|^2
+
\theta_1 h|\Delta\widehat Z_i|^2
\leq
(1+\theta_2h)\mathbb E_{t_i}\left[|\Delta Y_{i+1}|^2\right]
+
\theta_3h|\Delta X_i|^2,
\end{equation}
where $\Delta X_i\coloneqq X_i^{(1)}-X_i^{(2)}$, $\Delta Y_j\coloneqq Y_j^{(1)}-Y_j^{(2)}$ for $j=i,i+1$, and $\Delta\widehat Z_i \coloneqq\widehat Z_i^{(1)}-\widehat Z_i^{(2)}$.

\end{lemma}

Throughout the remainder of the paper, whenever Assumption~\ref{p2-ass:implicit_lipschitz} holds, we take $h_0>0$ sufficiently small that Lemma~\ref{p2-lem:implicit_backward_stability} applies for every $0<h\leq h_0$, reducing $h_0$ further when necessary without additional mention.

We next recall the Gaussian smoothing operators $\mathcal G_h^\varsigma$ introduced in Section~\ref{p2-sec2}, which form a semigroup as $h$ varies with $\varsigma$ fixed. The following proposition collects several standard properties that are particularly useful in the present paper, for handling globally Lipschitz functions that need not be differentiable everywhere. A proof is provided in the appendix for completeness.

\begin{proposition}[Properties of the Gaussian convolution semigroup]
\label{p2-prop:gaussian_semigroup}
Fix $\varsigma\in\mathbb R\setminus\{0\}$. For $h>0$ and a function $\phi:\mathbb R\to\mathbb R$ that is either globally Lipschitz or bounded and measurable, define
\begin{equation}
\label{p2-eq:def_gaussian_semigroup}
(\mathcal G_h^\varsigma \phi)(a)
\coloneqq
\mathbb E\left[\phi\left(a+\varsigma\sqrt h\,G\right)\right],
\qquad G\sim\mathcal N(0,1),
\end{equation}
and let $\mathcal G_0^\varsigma$ be the identity operator.

The family $\{\mathcal G_h^\varsigma\}_{h\geq0}$ satisfies the semigroup property $\mathcal G_h^\varsigma \mathcal G_s^\varsigma \phi = \mathcal G_{h+s}^\varsigma \phi$ for $h,s\geq0$. Moreover, the following properties hold for every $h>0$:

\vspace{-0.1cm}
\begin{enumerate}[label=(\roman*),ref=\roman*,itemsep=0pt,parsep=0pt]

\item \label{p2-item:gaussian_smoothing}
\textnormal{Gaussian smoothing.} If $\phi$ is globally Lipschitz or bounded and measurable, then
\begin{equation}
\mathcal G_h^\varsigma \phi\in C^\infty(\mathbb R).
\end{equation}
For $\phi\in L^\infty(\mathbb R)$, the operator is independent of the chosen measurable representative and satisfies
\begin{equation}
\|\mathcal G_h^\varsigma \phi\|_{L^\infty}
\leq \|\phi\|_{L^\infty}.
\end{equation}

\item \label{p2-item:gaussian_differentiation}
\textnormal{Gaussian differentiation formula.} If $\phi$ is globally Lipschitz or bounded and measurable, then, for every $a\in\mathbb R$,
\begin{equation}
\label{p2-eq:gaussian_first_derivative}
\partial_a(\mathcal G_h^\varsigma \phi)(a)
=
\frac{1}{\varsigma\sqrt h}
\mathbb E\left[
\phi(a+\varsigma\sqrt h\,G)G
\right].
\end{equation}
Equivalently, if $\Delta W\sim\mathcal N(0,h)$, then
\begin{equation}
\label{p2-eq:gaussian_first_derivative_dw}
\partial_a(\mathcal G_h^\varsigma \phi)(a)
=
\frac{1}{\varsigma h}
\mathbb E\left[
\phi(a+\varsigma\Delta W)\Delta W
\right].
\end{equation}

\item \label{p2-item:gaussian_weak_commutation}
\textnormal{Commutation with weak differentiation.} If $\phi$ is globally Lipschitz, then it is locally absolutely continuous and differentiable almost everywhere. Its almost-everywhere derivative coincides with its weak derivative, denoted by $\phi'$, and satisfies
\begin{equation}
\label{p2-eq:weak_derivative_bound}
\phi'\in L^\infty(\mathbb R),
\qquad
\|\phi'\|_{L^\infty}\leq\operatorname{Lip}(\phi).
\end{equation}
Furthermore,
\begin{equation}
\label{p2-eq:gaussian_weak_derivative}
\partial_a(\mathcal G_h^\varsigma \phi)(a)
=
(\mathcal G_h^\varsigma \phi')(a),
\qquad a\in\mathbb R.
\end{equation}
Since $\phi'\in L^\infty(\mathbb R)$, properties~(\ref{p2-item:gaussian_smoothing}) and (\ref{p2-item:gaussian_differentiation}) also apply to $\phi'$. In particular,
\begin{equation}
\label{p2-eq:gaussian_second_derivative_weak}
\partial_{aa}(\mathcal G_h^\varsigma \phi)(a)
=
\partial_a(\mathcal G_h^\varsigma \phi')(a)
=
\frac{1}{\varsigma\sqrt h}
\mathbb E\left[
\phi'(a+\varsigma\sqrt h\,G)G
\right],
\qquad a\in\mathbb R,
\end{equation}
without requiring a classical second derivative of $\phi$.

\end{enumerate}

\end{proposition}

To accommodate our time-dependent diffusion setting as in \eqref{p2-eq:implicit_full_scheme},  we apply the operator with $\varsigma=\sigma_i$ and write
\begin{equation}\label{p2-def:operator_Gi}
\mathcal G_i\coloneqq \mathcal G_h^{\sigma_i},
\qquad 0\leq i<N.
\end{equation}
The properties in Proposition~\ref{p2-prop:gaussian_semigroup} therefore apply to each $\mathcal G_i$.

Next, we establish general estimates showing that the transition operator of a discrete-time process with additive Gaussian noise maps bounded measurable functions to Lipschitz continuous functions and propagates the Lipschitz continuity of Lipschitz input functions. These estimates will be used to establish uniform spatial Lipschitz bounds for the discrete $Z$-fields.

\begin{lemma}[Smoothing for discrete-time processes with additive Gaussian noise] \label{p2-lem:additive_euler_smoothing}

Let $k<\ell$. For each $x\in\mathbb R$, consider the discrete-time process $\{X_r^{k,x}\}_{r=k}^{\ell}$ defined on a filtered probability space $(\Omega,\mathcal F,\{\mathcal F_r\}_{r=k}^{\ell},\mathbb P)$ by
\begin{equation}\label{p2-eq:additive_euler_chain}
X_k^{k,x}=x,
\qquad
X_{r+1}^{k,x}
=
X_r^{k,x}+hD_r(X_r^{k,x})
+\sigma_r\sqrt h\,G_r,
\qquad r=k,\ldots,\ell-1,
\end{equation}
where $\{G_r\}_{r=k}^{\ell-1}$ are independent standard normal random variables, and $\mathcal F_r$ is generated by $G_k,\ldots,G_{r-1}$ together with all $\mathbb P$-null sets. The drift functions $D_r:\mathbb R\to\mathbb R$ are deterministic and satisfy $|D_r(x)-D_r(y)|\leq K_D|x-y|$ for all $x,y\in\mathbb R$ and some constant $K_D\geq0$ independent of $r$, and the deterministic diffusion coefficients satisfy $|\sigma_r|\geq\underline\sigma>0$.

For a measurable function $\phi:\mathbb R\to\mathbb R$, define the transition operator $\mathcal P_{k,\ell}$ of the process \eqref{p2-eq:additive_euler_chain} by
\begin{equation}
\label{p2-eq:def_euler_semigroup}
(\mathcal P_{k,\ell}\phi)(x)
\coloneqq
\mathbb E\left[\phi(X_\ell^{k,x})\right],
\qquad x\in\mathbb R,
\end{equation}
whenever the expectation is well defined. Thus, $\mathcal P_{k,\ell}\phi$ is the deterministic function of the initial state $x$ obtained by propagating $\phi$ through the discrete process from time $t_k$ to time $t_\ell$.

Then, for every $\phi\in L^\infty(\mathbb R)$, with $\|\phi\|_{L^\infty}$ denoting the essential supremum norm,
\begin{equation}\label{p2-eq:additive_smoothing_bound}
\operatorname{Lip}(\mathcal P_{k,\ell}\phi)
\leq
\frac{1+K_D(t_\ell-t_k)}
{\underline\sigma\sqrt{t_\ell-t_k}}
\|\phi\|_{L^\infty},
\end{equation}
Moreover, if $\phi$ is globally Lipschitz, then
\begin{equation}\label{p2-eq:additive_lipschitz_propagation}
\operatorname{Lip}(\mathcal P_{k,\ell}\phi)
\leq
e^{K_D(t_\ell-t_k)}
\operatorname{Lip}(\phi).
\end{equation}

\end{lemma}

\begin{proof}

Fix $x,y\in\mathbb R$, and define the number of steps $n\coloneqq\ell-k$ and the initial state difference $\delta\coloneqq y-x$.

We first prove \eqref{p2-eq:additive_smoothing_bound}. Alongside $X^{k,x}$, we construct a process $Y^{k,y}$ starting from $Y_k^{k,y}=y$ such that
\begin{equation}
\label{p2-eq:coupling_difference}
Y_r^{k,y}-X_r^{k,x}
=
\frac{\ell-r}{n}\delta,
\qquad r=k,\ldots,\ell.
\end{equation}
In particular, we have $Y_\ell^{k,y}=X_\ell^{k,x}$ almost surely.

Let us write discrete processes as 
\begin{align}
X_{r+1}^{k,x}
=
X_r^{k,x}+hD_r(X_r^{k,x})
+\sigma_r\sqrt h\,G_r
,\qquad
Y_{r+1}^{k,y}
=
Y_r^{k,y}+hD_r(Y_r^{k,y})
+\sigma_r\sqrt h\,(G_r+\Gamma_r).
\end{align}
To enforce condition \eqref{p2-eq:coupling_difference}, we can choose $\Gamma_r$ such that
\begin{equation} 
\sigma_r\sqrt h\,\Gamma_r
=
-\frac{\delta}{n}
-h\left[
D_r(Y_r^{k,y})-D_r(X_r^{k,x})
\right].
\end{equation}
Then a direct calculation gives
\begin{equation}
Y_{r+1}^{k,y}-X_{r+1}^{k,x}
=
Y_r^{k,y}-X_r^{k,x}-\frac{\delta}{n}
=
\frac{\ell-r-1}{n}\delta,
\end{equation}
and the condition \eqref{p2-eq:coupling_difference} follows by induction.

With this choice of $\Gamma_r$, we use the Lipschitz continuity of $D_r$ and \eqref{p2-eq:coupling_difference} and obtain
\begin{align*}
|\Gamma_r|
\leq
\frac{1}{\underline\sigma\sqrt h}
\left[
\frac{|\delta|}{n}
+hK_D\frac{\ell-r}{n}|\delta|
\right]
\leq
\frac{|\delta|}{\underline\sigma\sqrt h\,n}
\left(1+K_Dnh\right).
\end{align*}
Consequently, since $nh=t_\ell-t_k$,
\begin{equation}
\label{p2-eq:Gamma_energy_bound}
\sum_{r=k}^{\ell-1}|\Gamma_r|^2
\leq
\frac{|\delta|^2}
{\underline\sigma^2(t_\ell-t_k)}
\left(1+K_D(t_\ell-t_k)\right)^2.
\end{equation}
Let us define
\begin{equation}\label{p2-eq:discrete_girsanov_density}
\Lambda
\coloneqq
\exp\left(
-\sum_{r=k}^{\ell-1}\Gamma_rG_r
-\frac12\sum_{r=k}^{\ell-1}|\Gamma_r|^2
\right).
\end{equation}
Since $\Gamma_r$ is $\mathcal F_r$-measurable and $G_r$ is independent of $\mathcal F_r$, the moment-generating function of a standard normal random variable gives
\begin{equation}
\mathbb E\left[
\left.
\exp\left(
-\Gamma_rG_r-\frac12|\Gamma_r|^2
\right)
\right|
\mathcal F_r
\right]
=1.
\end{equation}
Together with the deterministic bound \eqref{p2-eq:Gamma_energy_bound}, successive conditioning therefore gives $\mathbb E[\Lambda]=1$. 

Hence, $\mathrm d\mathbb Q=\Lambda\,\mathrm d\mathbb P$ defines a probability measure. The discrete Gaussian change-of-measure formula implies that, under $\mathbb Q$, the shifted increments $G_r+\Gamma_r$ are independent standard normal random variables. Consequently, under $\mathbb Q$, $Y^{k,y}$ has the same law as the discrete-time process \eqref{p2-eq:additive_euler_chain} starting from $y$.

For $\phi\in L^\infty(\mathbb R)$, using the definition of the transition operator, the identity $\mathrm d\mathbb Q=\Lambda\,\mathrm d\mathbb P$, and the fact that $Y_\ell^{k,y}=X_\ell^{k,x}$, we obtain
\begin{align*}
(\mathcal P_{k,\ell}\phi)(y)
&=
\mathbb E_{\mathbb Q}\left[\phi(Y_\ell^{k,y})\right]
=
\mathbb E_{\mathbb P}\left[\Lambda\phi(Y_\ell^{k,y})\right]
=
\mathbb E_{\mathbb P}\left[\Lambda\phi(X_\ell^{k,x})\right],
\\
(\mathcal P_{k,\ell}\phi)(x)
&=
\mathbb E_{\mathbb P}\left[\phi(X_\ell^{k,x})\right],
\end{align*}
Since $X_\ell^{k,x}$ has a density with respect to Lebesgue measure, $|\phi(X_\ell^{k,x})|\leq\|\phi\|_{L^\infty}$, $\mathbb P$-almost surely. Therefore,
\begin{equation}
\label{p2-eq:semigroup_difference_density}
\begin{aligned}
\left|
(\mathcal P_{k,\ell}\phi)(y)
-
(\mathcal P_{k,\ell}\phi)(x)
\right|
=
\left|
\mathbb E_{\mathbb P}
\left[(\Lambda-1)\phi(X_\ell^{k,x})\right]
\right|
\leq
\|\phi\|_{L^\infty}
\mathbb E_{\mathbb P}|\Lambda-1|.
\end{aligned}
\end{equation}

We next derive a bound for $\mathbb E_{\mathbb P}|\Lambda-1|$. Let $H(\mathbb Q\mid\mathbb P)$ denote the relative entropy of $\mathbb Q$ with respect to $\mathbb P$, and let $\|\mathbb Q-\mathbb P\|_{\mathrm{TV}}$ denote their total variation distance. Since $G_r+\Gamma_r$ is conditionally centered under $\mathbb Q$ given $\mathcal F_r$, we obtain
\begin{equation}
H(\mathbb Q\mid\mathbb P)
=
\mathbb E_{\mathbb Q}[\log\Lambda]
=
\frac12\mathbb E_{\mathbb Q}
\left[
\sum_{r=k}^{\ell-1}|\Gamma_r|^2
\right]
\leq
\frac{|\delta|^2}
{2\underline\sigma^2(t_\ell-t_k)}
\left(1+K_D(t_\ell-t_k)\right)^2.
\end{equation}
The $L^1$ representation of total variation distance and Pinsker's inequality give
\begin{equation}
\mathbb E_{\mathbb P}|\Lambda-1|
=
2\|\mathbb Q-\mathbb P\|_{\mathrm{TV}}
\leq
\sqrt{2H(\mathbb Q\mid\mathbb P)}
\leq
\frac{1+K_D(t_\ell-t_k)}
{\underline\sigma\sqrt{t_\ell-t_k}}
|x-y|.
\end{equation}
Combining this estimate with \eqref{p2-eq:semigroup_difference_density} proves \eqref{p2-eq:additive_smoothing_bound}.

Finally, we prove \eqref{p2-eq:additive_lipschitz_propagation}. Let $X^{k,x}$ and $X^{k,y}$ be two processes satisfying \eqref{p2-eq:additive_euler_chain}, driven by the same Gaussian increments, but start from different states $x$ and $y$, respectively. Then using the Lipschitz continuity of $D_r$ gives
\begin{equation}
|X_{r+1}^{k,x}-X_{r+1}^{k,y}|
\leq
(1+hK_D)|X_r^{k,x}-X_r^{k,y}|.
\end{equation}
and iterating this estimate gives
\begin{equation}
|X_\ell^{k,x}-X_\ell^{k,y}|
\leq
(1+hK_D)^{\ell-k}|x-y|
\leq
e^{K_D(t_\ell-t_k)}|x-y|.
\end{equation}
Therefore, if $\phi$ is globally Lipschitz, we have
\begin{align*}
\left|
(\mathcal P_{k,\ell}\phi)(x)
-
(\mathcal P_{k,\ell}\phi)(y)
\right|
\leq
\mathbb E\left[
|\phi(X_\ell^{k,x})-\phi(X_\ell^{k,y})|
\right]
\leq
\operatorname{Lip}(\phi)
e^{K_D(t_\ell-t_k)}
|x-y|.
\end{align*}

\end{proof}

\subsection{Lipschitz continuity and linear growth bounds}\label{p2-sec32}

We first establish an invariant Lipschitz class, in the sense that, if the input discrete fields satisfy prescribed spatial Lipschitz bounds, the output discrete fields obtained after one Markovian iteration satisfy the same bounds.

Let $\bar L,\bar V>0$ be candidate bounds for the squared Lipschitz constants of the $Y$- and $Z$-fields, respectively, and define, for later use,
\begin{equation}\label{p2-eq:def_KB_constant_sigma}
K_B\coloneqq\sqrt{b_x+b_y\bar L+b_z\bar V}.
\end{equation}
Recall also that $\theta_2,\theta_3$ are the constants introduced in Lemma~\ref{p2-lem:implicit_backward_stability}.

\begin{theorem}[Invariant Lipschitz class] \label{p2-thm:constant_sigma_invariant_class}

Let Assumptions~\ref{p2-ass:implicit_lipschitz} -- \ref{p2-ass:terminal_regular} hold. Let $h_0>0$ be chosen such that the implicit backward step is well defined and the conditions of Lemma~\ref{p2-lem:implicit_backward_stability} hold for every $0<h\leq h_0$.

Suppose that $\bar L,\bar V>0$ satisfy the conditions 
\begin{equation}
\label{p2-eq:constant_sigma_invariant_L}
\left(
 g_x+\theta_3T
\right)
 e^{(\theta_2+2K_B)T}
\leq
\bar L,
\end{equation}
\begin{equation}
\label{p2-eq:constant_sigma_invariant_V}
\begin{aligned}
\overline\sigma^2(1+h_0K_B)^2
\Bigg[
 e^{K_BT}\operatorname{Lip}(g')
+
\frac{2(1+K_BT)\sqrt T}{\underline\sigma}
\left(
 K_B\sqrt{\bar L}
+
\sqrt{f_x+f_y\bar L+f_z\bar V}
\right)
\Bigg]^2
\leq
\bar V.
\end{aligned}
\end{equation}
Then, for every pair of globally Lipschitz input fields $\varphi=\{\varphi_i\}_{i=0}^{N-1}$ and $\xi=\{\xi_i\}_{i=0}^{N-1}$ satisfying
\begin{equation}
\label{p2-eq:constant_sigma_input_class}
L(\varphi)\leq\bar L,
\qquad
L(\xi)\leq\bar V,
\end{equation}
the output fields $(\Phi,\psi)=F^\pi(\varphi,\xi)$ satisfy
\begin{equation}
\label{p2-eq:constant_sigma_output_class}
L(\Phi)\leq\bar L,
\qquad
L(\psi)\leq\bar V.
\end{equation}
Moreover,
\begin{equation}
\label{p2-eq:constant_sigma_smoothed_Hessian}
\begin{aligned}
\sup_{0\leq i\leq N-1}
\left\|
\partial_{aa}\mathcal G_i\Phi_{i+1}
\right\|_\infty
\leq{}&
 e^{K_BT}\operatorname{Lip}(g')
\\
&+
\frac{2(1+K_BT)\sqrt T}{\underline\sigma}
\left(
 K_B\sqrt{\bar L}
+
\sqrt{f_x+f_y\bar L+f_z\bar V}
\right).
\end{aligned}
\end{equation}

\end{theorem}

\begin{proof}

Let us fix a uniform partition $\pi$ with step size $0<h\leq h_0$ such that Lemma~\ref{p2-lem:implicit_backward_stability} holds. Throughout the proof, we suppress the dependence of the fields and processes on $\pi$ and omit the iteration index, since only this single application of $F^\pi$ is considered.

For a state $x$ at time $t_i$, the one-step forward transition is
\begin{equation}
\label{p2-eq:one_step_transition_constant_sigma}
X_{i+1}^{i,x}
=
x
+
h\,b\left(
t_i,x,\varphi_i(x),\xi_i(x)
\right)
+
\sigma_i\Delta W_i
=
A_i(x)+\sigma_i\Delta W_i,
\qquad
\Delta W_i\sim\mathcal N(0,h).
\end{equation}
where $A_i(x)\coloneqq x+h\,b(t_i,x,\varphi_i(x),\xi_i(x))$.

We divide the proof into two parts, treating the $Y$-field and the $Z$-field, respectively.

\medskip
\noindent
\textit{Part 1: Uniform Lipschitz bound for the $Y$-field.}

Starting from $\Phi_N=g$, we proceed by backward induction. Fix $i\in\{0,\ldots,N-1\}$ and suppose that $\Phi_{i+1}$ is globally Lipschitz. By Assumption~\ref{p2-ass:implicit_lipschitz}, \eqref{p2-eq:constant_sigma_input_class}, and the definition of $K_B$, we have
\begin{equation}
\label{p2-eq:Ai_lipschitz}
\begin{aligned}
|A_i(x)-A_i(y)|
&=
\left|
x-y
+
h\left[
b\left(t_i,x,\varphi_i(x),\xi_i(x)\right)
-
b\left(t_i,y,\varphi_i(y),\xi_i(y)\right)
\right]
\right|
\\
&\leq
|x-y|
+
h\left|
b\left(t_i,x,\varphi_i(x),\xi_i(x)\right)
-
b\left(t_i,y,\varphi_i(y),\xi_i(y)\right)
\right|
\\
&\leq
(1+hK_B)|x-y|.
\end{aligned}
\end{equation}
For $x,y\in\mathbb R$, let $X_{i+1}^{i,x}$ and $X_{i+1}^{i,y}$ be the one-step forward transitions in \eqref{p2-eq:one_step_transition_constant_sigma}, starting from $X_i^{i,x}=x$ and $X_i^{i,y}=y$, respectively, and driven by the same Brownian increment. Then
\begin{equation}
|X_{i+1}^{i,x}-X_{i+1}^{i,y}|^2
=
|A_i(x)-A_i(y)|^2
\leq
(1+hK_B)^2|x-y|^2. 
\end{equation}
Applying Lemma~\ref{p2-lem:implicit_backward_stability} and dropping the nonnegative $Z$-term,
\begin{align*}
|\Phi_i(x)-\Phi_i(y)|^2
&\leq
(1+\theta_2h)
\mathbb E\left[
|\Phi_{i+1}(X_{i+1}^{i,x})
-\Phi_{i+1}(X_{i+1}^{i,y})|^2
\right]
+
\theta_3h|x-y|^2
\\
&\leq
\left[
(1+\theta_2h)(1+hK_B)^2L(\Phi_{i+1})
+
\theta_3h
\right]
|x-y|^2.
\end{align*}
Thus,
\begin{equation}\label{p2-eq:Phi_recursion_constant_sigma}
L(\Phi_i)
\leq
(1+\theta_2h)(1+hK_B)^2L(\Phi_{i+1})
+
\theta_3h.
\end{equation}
Since $\Phi_N=g$ is globally Lipschitz, this recursion establishes global Lipschitz continuity of every $\Phi_i$ by backward induction.

Using $L(\Phi_N)=L(g)\leq g_x$ and the elementary inequality $1+r\leq e^r$ for any $r\geq0$, we have the following by iterating \eqref{p2-eq:Phi_recursion_constant_sigma},
\begin{equation}
\label{p2-eq:Phi_bound_constant_sigma}
L(\Phi)
\leq
\left(g_x+\theta_3T\right)e^{(\theta_2+2K_B)T}.
\end{equation}
By condition \eqref{p2-eq:constant_sigma_invariant_L}, we prove the first result, 
\begin{equation}
\label{p2-eq:Phi_barL_constant_sigma}
L(\Phi)\leq\bar L.
\end{equation}

\medskip
\noindent
\textit{Part 2: Gaussian representation and Lipschitz continuity of the $Z$-field.}

By Part 1, each $\Phi_i$ is globally Lipschitz. Rademacher's theorem implies that it is differentiable almost everywhere. Moreover, Lipschitz functions on $\mathbb R$ are locally absolutely continuous, so integration by parts shows that this almost-everywhere derivative coincides with the weak derivative. We denote it by $w_i\coloneqq\partial_x\Phi_i$.

Moreover,
\begin{equation}
\label{p2-eq:wi_bound}
\|w_i\|_{L^\infty}
\leq \operatorname{Lip}(\Phi_i)
\leq \sqrt{\bar L},
\qquad i=0,\ldots,N.
\end{equation}
Here and below, $L^\infty$-norms are essential suprema with respect to Lebesgue measure. The weak derivatives are defined only almost everywhere, but their values on the remaining Lebesgue-null sets may be chosen arbitrarily. Since $\sigma_i\neq0$, the Gaussian random variables used below have densities and hit these sets with probability zero. These choices therefore do not change the corresponding expectations.

Let us recall the definition of the operator $\mathcal G_i$ as in \eqref{p2-def:operator_Gi} and its properties stated in Proposition \ref{p2-prop:gaussian_semigroup}. Then we can write one-step backward equation $\Phi_i(x)$ as follows 
\begin{equation}\label{p2-eq:Phi_gaussian_form}
\begin{aligned}
\Phi_i(x)
& =
\mathbb E
\left[
\Phi_{i+1}\left(
X_{i+1}^{i,x}
\right)
\right]
+
h f\left(
t_i,x,\Phi_i(x),\psi_i(x)
\right)
\\
& =
\mathbb E
\left[
\Phi_{i+1}\left(
A_i(x)+\sigma_i\Delta W_i
\right)
\right]
+
h f\left(
t_i,x,\Phi_i(x),\psi_i(x)
\right)  \\
&
=
(\mathcal G_i\Phi_{i+1})(A_i(x))
+
hf_i(x)
\end{aligned}
\end{equation}
where $ f_i(x) \coloneqq f\left( t_i,x,\Phi_i(x),\psi_i(x) \right) $.

Similarly, and additionally using \eqref{p2-eq:gaussian_first_derivative_dw}, the $Z$-field is given by
\begin{equation}
\begin{aligned}
\psi_i(x)
&=
\frac{1}{h}
\mathbb E
\left[
\Phi_{i+1}\left(
X_{i+1}^{i,x}
\right)
\Delta W_i
\right]
\\
&=
\frac{1}{h}
\mathbb E
\left[
\Phi_{i+1}\left(
A_i(x)+\sigma_i\Delta W_i
\right)
\Delta W_i
\right]
\\
& = \sigma_i
\left(
\partial_a\mathcal G_i\Phi_{i+1}
\right)(A_i(x)),
\end{aligned}
\end{equation}
and by Proposition~\ref{p2-prop:gaussian_semigroup}, $\partial_a\mathcal G_i\Phi_{i+1} = \mathcal G_i w_{i+1}$, and therefore we can further write 
\begin{equation}
\label{p2-eq:psi_gaussian_form}
\psi_i(x)
=
\sigma_i
\left(
\mathcal G_i w_{i+1}
\right)(A_i(x)).
\end{equation}

We now prove $L(\psi_i)\leq\bar V$ and the smoothed second-derivative estimate by backward induction.

Let $i=N-1$. Since $\Phi_N=g$, we have $w_N=g'$, and therefore
\begin{equation}
\operatorname{Lip}
\left(
\mathcal G_{N-1} w_N
\right)
=
\operatorname{Lip}
\left(
\mathcal G_{N-1} g'
\right)
\leq
\operatorname{Lip}(g').
\end{equation}
Using \eqref{p2-eq:psi_gaussian_form} and \eqref{p2-eq:Ai_lipschitz},
\begin{align*}
|\psi_{N-1}(x)-\psi_{N-1}(y)|
&\leq
|\sigma_{N-1}|\,
\operatorname{Lip}(\mathcal G_{N-1} g')
|A_{N-1}(x)-A_{N-1}(y)|
\\
&\leq
|\sigma_{N-1}|(1+hK_B)
\operatorname{Lip}(g')
|x-y|.
\end{align*}
Consequently, by condition \eqref{p2-eq:constant_sigma_invariant_V},
\begin{equation}
L(\psi_{N-1})
\leq
\overline\sigma^2(1+h_0K_B)^2
\operatorname{Lip}(g')^2
\leq
\bar V.
\end{equation}
Moreover, using Proposition~\ref{p2-prop:gaussian_semigroup} we can write $\partial_{aa}\mathcal G_{N-1}\Phi_N=\partial_a\mathcal G_{N-1} g'$, and $\partial_a\mathcal G_{N-1} g'$ is smooth. Therefore $\partial_{aa}\mathcal G_{N-1}\Phi_N$ exists classically and we have
\begin{equation}
\left\|
\partial_{aa}\mathcal G_{N-1}\Phi_N
\right\|_\infty
\leq
\operatorname{Lip}(\mathcal G_{N-1} g')
\leq
\operatorname{Lip}(g').
\end{equation}
Thus the required smoothed second-derivative estimate also holds at $i=N-1$.

Now we consider $i\leq N-2$, and suppose that
\begin{equation}
\label{p2-eq:psi_backward_induction}
L(\psi_j)\leq\bar V,
\qquad
j=i+1,\ldots,N-1.
\end{equation}
We show that $L(\psi_i)\leq\bar V$ at the current index $i$ and establish the corresponding smoothed second-derivative estimate.

For $j=i+1,\ldots,N-1$, define $ f_j(x) \coloneqq f\left( t_j,x,\Phi_j(x),\psi_j(x) \right) $. The induction hypothesis and \eqref{p2-eq:Phi_barL_constant_sigma} imply
\begin{align*}
|f_j(x)-f_j(y)|^2
&\leq
f_x|x-y|^2
+
f_y|\Phi_j(x)-\Phi_j(y)|^2
+
f_z|\psi_j(x)-\psi_j(y)|^2
\\
&\leq
\left(
f_x+f_y\bar L+f_z\bar V
\right)|x-y|^2.
\end{align*}
Hence $f_j$ is globally Lipschitz and its weak derivative satisfies
\begin{equation}
\label{p2-eq:Fi_weak_derivative}
\|f_j'\|_{L^\infty}
\leq
\sqrt{
f_x+f_y\bar L+f_z\bar V
}.
\end{equation}
Likewise, define $ b_j(x) \coloneqq b\left( t_j,x,\varphi_j(x),\xi_j(x) \right) $, and let $b_j'$ denote its almost-everywhere, equivalently weak derivative. Then we have
\begin{equation}
\label{p2-eq:beta_bound}
\|b_j'\|_{L^\infty}
\leq 
\operatorname{Lip}(b_j) 
\leq K_B,
\end{equation}
and since $A_j(x)=x+hb_j(x)$, we have $A_j'(x)=1+hb_j'(x)$ for almost every $x$.

Since $\mathcal G_j\Phi_{j+1}$ is smooth, the ordinary chain rule applies at every point where $A_j$ and $f_j$ are differentiable. Combining this with \eqref{p2-eq:gaussian_weak_derivative} in Proposition~\ref{p2-prop:gaussian_semigroup}, we obtain, for almost every $x$, $\partial_x\Phi_j(x) = w_j(x)$, and 
\begin{equation}
\begin{aligned}
w_j(x) 
& =
\left(
\partial_a \mathcal G_j\Phi_{j+1}
\right)(A_j(x))
A_j'(x)
+
hf_j'(x) \\
&=
\left(
\partial_a \mathcal G_j\Phi_{j+1}
\right)(A_j(x))
\left(
1+hb_j'(x)
\right)
+
hf_j'(x)  \\
&=
\left(
\mathcal G_j w_{j+1}
\right)(A_j(x))
\left(
1+hb_j'(x)
\right)
+
hf_j'(x)  \\
&=
\left(
\mathcal G_j w_{j+1}
\right)(A_j(x))
+
hb_j'(x)
\left(
\mathcal G_j w_{j+1}
\right)(A_j(x))
+
hf_j'(x)  \\
& =
(Q_jw_{j+1})(x)
+
hR_j(x),
\label{p2-eq:weak_derivative_recursion}
\end{aligned}
\end{equation}
where we define the operator $Q_j$ and the function $R_j$ by
\begin{equation}
(Q_j\phi)(x)
=
(\mathcal G_j\phi)(A_j(x))
,
\qquad
R_j(x)
\coloneqq
b_j'(x)(Q_jw_{j+1})(x)
+
f_j'(x),
\end{equation}
where $R_j$ is defined for almost every $x\in\mathbb R$.

Iterating \eqref{p2-eq:weak_derivative_recursion} from $i+1$ to $N-1$ and using the terminal condition $w_N=g'$ gives, almost everywhere,
\begin{equation}
w_{i+1}
=
Q_{i+1}\cdots Q_{N-1}g'
+
h\sum_{j=i+1}^{N-1}
Q_{i+1}\cdots Q_{j-1}R_j,
\end{equation}
where an empty product is interpreted as the identity, and in particular, for $j=i+1$, $Q_{i+1}\cdots Q_{j-1}=I$.

Applying the operator $\mathcal G_i$ to both sides and defining $\mathcal R_{i,j} \coloneqq \mathcal G_i Q_{i+1}\cdots Q_{j-1}$ for $j \geq i+1$, with the same convention for empty products, we obtain
\begin{equation}\label{p2-eq:wi_Duhamel}
\mathcal G_i w_{i+1}
=
\mathcal R_{i,N}g'
+
h\sum_{j=i+1}^{N-1}
\mathcal R_{i,j}R_j.
\end{equation}
To identify the operator $\mathcal R_{i,j}$, let us first recall the definition that
\begin{equation}
(Q_r\phi)(x)
=
(\mathcal G_r\phi)(A_r(x))
=
\mathbb E
\left[
\phi\left(
x+hb_r(x)+\sigma_r\sqrt h\,G_r
\right)
\right].
\end{equation}
where $b_r(x)= b\left(t_r,x,\varphi_r(x),\xi_r(x)\right)$, and the variables $G_r$ are independent standard normal random variables. Thus $Q_r$ is the one-step transition operator of an Euler scheme with drift $b_r$, whereas $\mathcal G_i$ is the one-step transition operator with zero drift and diffusion coefficient $\sigma_i$.

Accordingly, $\mathcal R_{i,j}$ can be identified with the transition operator over $[t_i,t_j]$ of an auxiliary discrete-time process with additive Gaussian noise whose first step has zero drift and whose subsequent drift functions are given by $b_r$, with diffusion coefficients $\sigma_i,\ldots,\sigma_{j-1}$. These coefficients are deterministic and satisfy $|\sigma_r|\geq\underline\sigma$. Since $\operatorname{Lip}(b_r)\leq K_B$, all drift functions of this auxiliary chain have Lipschitz constant bounded by $K_B$.

To use Lemma \ref{p2-lem:additive_euler_smoothing}, we need an estimate for $\|R_j\|_{L^\infty}$. Since $w_{j+1}\in L^\infty(\mathbb R)$, Proposition \ref{p2-prop:gaussian_semigroup} gives $|(\mathcal G_jw_{j+1})(a)|\leq\|w_{j+1}\|_{L^\infty}$ for every $a\in\mathbb R$. Evaluating at $a=A_j(x)$ yields
\begin{equation}
\|Q_jw_{j+1}\|_{L^\infty}
\leq
\|w_{j+1}\|_{L^\infty}
\leq
\sqrt{\bar L}.
\end{equation}
Together with \eqref{p2-eq:Fi_weak_derivative} and \eqref{p2-eq:beta_bound}, this yields
\begin{equation}
\label{p2-eq:Rj_bound}
\|R_j\|_{L^\infty}
\leq
K_B\sqrt{\bar L}
+
\sqrt{
f_x+f_y\bar L+f_z\bar V
}.
\end{equation}
Therefore, using Lemma~\ref{p2-lem:additive_euler_smoothing} and \eqref{p2-eq:Rj_bound} yields
\begin{equation}
\begin{aligned}
\operatorname{Lip}
\left(
\mathcal R_{i,N}g'
\right)
\leq
e^{K_BT}\operatorname{Lip}(g') 
,\qquad 
\operatorname{Lip}
\left(
\mathcal R_{i,j}R_j
\right)
\leq
\frac{1+K_BT}
{\underline\sigma\sqrt{t_j-t_i}}
\left(
K_B\sqrt{\bar L}
+
\sqrt{
f_x+f_y\bar L+f_z\bar V
}
\right) .
\end{aligned}
\end{equation}
Consequently, using \eqref{p2-eq:wi_Duhamel} we have
\begin{align}
\operatorname{Lip}
\left(
\mathcal G_i w_{i+1}
\right)
\leq
e^{K_BT}\operatorname{Lip}(g')
+
\frac{1+K_BT}{\underline\sigma}
\left(
K_B\sqrt{\bar L}
+
\sqrt{
f_x+f_y\bar L+f_z\bar V
}
\right)
h
\sum_{j=i+1}^{N-1}
\frac1{\sqrt{t_j-t_i}}.
\label{p2-eq:smoothed_gradient_before_sum}
\end{align}
Since
\begin{equation}
h\sum_{j=i+1}^{N-1}
\frac1{\sqrt{t_j-t_i}}
=
\sqrt h
\sum_{r=1}^{N-i-1}\frac1{\sqrt r}
\leq
2\sqrt T,
\end{equation}
we obtain
\begin{equation}
\label{p2-eq:smoothed_gradient_bound}
\begin{aligned}
\operatorname{Lip}
\left(
\mathcal G_i w_{i+1}
\right)
\leq
e^{K_BT}\operatorname{Lip}(g')
+
\frac{2(1+K_BT)\sqrt T}{\underline\sigma}
\left(
K_B\sqrt{\bar L}
+
\sqrt{
f_x+f_y\bar L+f_z\bar V
}
\right).
\end{aligned}
\end{equation}
Note that by Proposition~\ref{p2-prop:gaussian_semigroup}, we can write $\partial_a\mathcal G_i\Phi_{i+1}=\mathcal G_i w_{i+1}$, and $\mathcal G_i w_{i+1}$ is smooth. Therefore $\partial_{aa}\mathcal G_i\Phi_{i+1}$ exists in the classical sense, and we have
\begin{equation}
\left\|
\partial_{aa}\mathcal G_i\Phi_{i+1}
\right\|_\infty
\leq
\operatorname{Lip}
\left(
\mathcal G_i w_{i+1}
\right).
\end{equation}
Using \eqref{p2-eq:smoothed_gradient_bound} yields the estimate \eqref{p2-eq:constant_sigma_smoothed_Hessian} at the current index $i$.

Finally, using \eqref{p2-eq:psi_gaussian_form}, \eqref{p2-eq:Ai_lipschitz}, and \eqref{p2-eq:smoothed_gradient_bound},
\begin{align*}
|\psi_i(x)-\psi_i(y)|
&\leq
|\sigma_i|\,
\operatorname{Lip}
\left(
\mathcal G_i w_{i+1}
\right)
|A_i(x)-A_i(y)|
\\
&\leq
|\sigma_i|(1+hK_B)
\Bigg[
e^{K_BT}\operatorname{Lip}(g')
\\
&\qquad\qquad+
\frac{2(1+K_BT)\sqrt T}{\underline\sigma}
\left(
K_B\sqrt{\bar L}
+
\sqrt{
f_x+f_y\bar L+f_z\bar V
}
\right)
\Bigg]
|x-y|.
\end{align*}
Hence, by condition \eqref{p2-eq:constant_sigma_invariant_V}, 
\begin{align*}
L(\psi_i)
\leq 
\overline\sigma^2(1+h_0K_B)^2
\Bigg[
e^{K_BT}\operatorname{Lip}(g')
+
\frac{2(1+K_BT)\sqrt T}{\underline\sigma}
\left(
K_B\sqrt{\bar L}
+
\sqrt{
f_x+f_y\bar L+f_z\bar V
}
\right)
\Bigg]^2
\leq
\bar V
\end{align*}
This completes the induction, and therefore we can conclude that $L(\psi_i)\leq\bar V$, $i=0,\ldots,N-1$, and obtain the desired estimate \eqref{p2-eq:constant_sigma_smoothed_Hessian}.

\end{proof}

Having established that one Markovian iteration preserves the Lipschitz bounds, uniform bounds across all iterations follow immediately by induction.

\begin{corollary}[Uniform Lipschitz bounds for the Markovian iteration]
\label{p2-cor:constant_sigma_uniform_iteration}
Suppose the assumptions of Theorem~\ref{p2-thm:constant_sigma_invariant_class} hold. Initialize $u_i^{\pi,0}=v_i^{\pi,0}=0$ for $i=0,\ldots,N-1$, and set $u_N^{\pi,0}=g$. Then, for every $m\geq0$,
\begin{equation}
\label{p2-eq:constant_sigma_uniform_iteration_LV}
L(u^{\pi,m})\leq\bar L,
\qquad
L(v^{\pi,m})\leq\bar V.
\end{equation}
Moreover,
\begin{equation}
\label{p2-eq:constant_sigma_uniform_smoothed_Hessian_iteration}
\begin{aligned}
&\sup_{m\geq1}
\sup_{0\leq i\leq N-1}
\left\|
\partial_{aa}\mathcal G_i u_{i+1}^{\pi,m}
\right\|_\infty
\\
&\qquad\leq
e^{K_BT}\operatorname{Lip}(g')
+
\frac{2(1+K_BT)\sqrt T}{\underline\sigma}
\left(
K_B\sqrt{\bar L}
+
\sqrt{f_x+f_y\bar L+f_z\bar V}
\right).
\end{aligned}
\end{equation}
\end{corollary}

\begin{proof}
The initial fields belong to the invariant class, since $L(g)\leq g_x\leq\bar L$. Applying Theorem~\ref{p2-thm:constant_sigma_invariant_class} successively proves the Lipschitz bounds by induction on $m$. The smoothed second-derivative bound follows from the same theorem at each iteration $m\geq1$.
\end{proof}

\begin{remark}[Uniform boundedness of the discrete $Z$-fields]
\label{p2-rem:constant_sigma_uniform_v_bound}
Under the assumptions of Theorem~\ref{p2-thm:constant_sigma_invariant_class}, we can write $\psi_i(x) = \sigma_i(\mathcal G_i w_{i+1})(A_i(x))$, and therefore
\begin{equation}
\label{p2-eq:constant_sigma_uniform_v_bound}
|\psi_i(x)|^2
\leq
\overline\sigma^2\|w_{i+1}\|_{L^\infty}^2
\leq
\overline\sigma^2\bar L.
\end{equation}
Hence the discrete $Z$-fields are uniformly bounded, independently of the partition and the iteration index.
\end{remark}

\begin{remark}[Feasibility of the invariant conditions]
\label{p2-rem:constant_sigma_invariant_feasibility}
The conditions \eqref{p2-eq:constant_sigma_invariant_L} and \eqref{p2-eq:constant_sigma_invariant_V} in Theorem \ref{p2-thm:constant_sigma_invariant_class} are coupled through $K_B=\sqrt{b_x+b_y\bar L+b_z\bar V}$ and the term $\sqrt{f_x+f_y\bar L+f_z\bar V}$. Nevertheless, they define a nonempty parameter regime.

One simple choice is to fix $\bar L$ and $\bar V$ such that 
\begin{equation}
\label{p2-eq:strict_candidate_LV}
\bar L>g_x,
\qquad
\bar V>\overline\sigma^2\operatorname{Lip}(g')^2.
\end{equation}
Then $K_B$ is fixed, and the left-hand side of \eqref{p2-eq:constant_sigma_invariant_L} satisfies
\begin{equation}
(g_x+\theta_3T)e^{(\theta_2+2K_B)T}
\longrightarrow g_x
\qquad\text{as }T\downarrow0.
\end{equation}
Similarly, as $T\downarrow0$ and $h_0\downarrow0$, the left-hand side of \eqref{p2-eq:constant_sigma_invariant_V} converges to $\overline\sigma^2\operatorname{Lip}(g')^2$.

Therefore, the strict inequalities in \eqref{p2-eq:strict_candidate_LV} ensure that both \eqref{p2-eq:constant_sigma_invariant_L} and \eqref{p2-eq:constant_sigma_invariant_V} hold for sufficiently small $T$ and admissible $h_0$. These conditions are sufficient and are not expected to be sharp.
\end{remark}

We next establish a linear-growth estimate for the first Markovian iteration, which provides the initial bound needed for the convergence analysis, as explained in Remark~\ref{p2-rem:implicit_initial_energy}. This bound is independent of the partition and is subsequently used to establish convergence and uniform linear growth across all iterations in Theorem~\ref{p2-thm:implicit_markovian_convergence}.

\begin{lemma}[Linear growth of the first iteration]
\label{p2-thm:implicit_first_growth}
Suppose that the assumptions of Theorem~\ref{p2-thm:constant_sigma_invariant_class} hold. For sufficiently small $h_0>0$, there exists a constant $C>0$, independent of $h,i$ and $x$, such that, for every $0< h \leq h_0$,
\begin{equation}
\label{p2-eq:implicit_first_growth}
|u_i^{\pi,1}(x)|^2\leq C(1+|x|^2),
\qquad 0\leq i\leq N,\quad x\in\mathbb R.
\end{equation}
\end{lemma}

\begin{proof}
Throughout this proof, we consider the first iteration, $m=1$, and suppress the superscripts $\pi,1$ on the processes, while retaining them on the decoupling fields.

Fix $i_0\in\{0,\ldots,N-1\}$ and $x\in\mathbb R$. Since the input fields vanish at all indices $i<N$, the forward process satisfies,  
\begin{equation}
X_{i_0}=x,
\qquad
X_{i+1}
=
X_i+h b(t_i,X_i,0,0)+\sigma_i\Delta W_i,
\qquad i=i_0,\ldots,N-1.
\end{equation}
The linear growth of the drift and the boundedness of $\sigma$ yield the standard Euler--Maruyama second-moment estimate
\begin{equation}
\label{p2-eq:first_forward_growth_implicit}
\sup_{i_0\leq i\leq N}
\mathbb E\left[|X_i|^2\right]
\leq C(1+|x|^2).
\end{equation}
For $i=i_0,\ldots,N-1$, define $\bar Y_i\coloneqq\mathbb E_{t_i}[Y_{i+1}]$. Then we can write
\begin{equation}
h\widehat Z_i
=
\mathbb{E}_{t_i}\left[Y_{i+1}\Delta W_i\right]
=
\mathbb{E}_{t_i}\left[
(Y_{i+1}-\bar Y_i)\Delta W_i
\right].
\end{equation}
Applying the conditional Cauchy--Schwarz inequality and using $\mathbb{E}_{t_i}[|\Delta W_i|^2]=h$ yields
\begin{equation}
\begin{aligned}
h^2|\widehat Z_i|^2
&\leq
\mathbb{E}_{t_i}\left[|Y_{i+1}-\bar Y_i|^2\right]
\mathbb{E}_{t_i}\left[|\Delta W_i|^2\right]
\\
&=
h\mathbb{E}_{t_i}\left[|Y_{i+1}-\bar Y_i|^2\right]
\\
&=
h\left(
\mathbb{E}_{t_i}\left[|Y_{i+1}|^2\right]
-
|\bar Y_i|^2
\right).
\end{aligned}
\end{equation}
Rearranging gives
\begin{equation}
\label{p2-eq:projection_growth_implicit}
|\bar Y_i|^2+h|\widehat Z_i|^2
\leq
\mathbb{E}_{t_i}\left[|Y_{i+1}|^2\right].
\end{equation}
Fix $\delta>f_z$, independently of $h$. Using the implicit backward equation in \eqref{p2-eq:implicit_full_scheme}, Young's inequality, and the growth bound on $f$, we obtain
\begin{equation}
\begin{aligned}
|Y_i|^2
=
|\bar Y_i+h f(t_i,X_i,Y_i,\widehat Z_i)|^2
\leq
(1+\delta h)|\bar Y_i|^2
+
(\delta^{-1}h+h^2)
\left[
f_0+f_x|X_i|^2+f_y|Y_i|^2
+f_z|\widehat Z_i|^2
\right].
\end{aligned}
\end{equation}
Using \eqref{p2-eq:projection_growth_implicit}, we obtain
\begin{equation}
\label{p2-eq:implicit_growth_onestep_raw}
\begin{aligned}
&\left[1-f_yh(\delta^{-1}+h)\right]|Y_i|^2
+
h\left[1+\delta h-f_z(\delta^{-1}+h)\right]
|\widehat Z_i|^2
\\
\leq &
(1+\delta h)\mathbb E_{t_i}\left[|Y_{i+1}|^2\right]
+
h(\delta^{-1}+h)(f_0+f_x|X_i|^2).
\end{aligned}
\end{equation}
Since $\delta>f_z$, the coefficient of $h|\widehat Z_i|^2$ is positive,
\begin{equation}
1+\delta h-f_z(\delta^{-1}+h)
=
1-\frac{f_z}{\delta}+(\delta-f_z)h
>0.
\end{equation}
Choose $h_0>0$ sufficiently small that
\begin{equation}
1-f_yh(\delta^{-1}+h)\geq\frac12,
\qquad 0<h\leq h_0.
\end{equation}
Discarding the nonnegative $Z$-term in \eqref{p2-eq:implicit_growth_onestep_raw} and dividing by $1-f_yh(\delta^{-1}+h)$, we obtain
\begin{equation}
\label{p2-eq:implicit_growth_onestep}
|Y_i|^2
\leq
(1+Ch)\mathbb E_{t_i}\left[|Y_{i+1}|^2\right]
+
Ch(1+|X_i|^2),
\end{equation}
where $C>0$ is independent of $h$.

Taking expectations in \eqref{p2-eq:implicit_growth_onestep}, using \eqref{p2-eq:first_forward_growth_implicit} and the terminal growth bound $|Y_N|^2=|g(X_N)|^2\leq g_0+g_x|X_N|^2$, and applying the discrete Gronwall inequality backward yields
\begin{equation}
\sup_{i_0\leq i\leq N}
\mathbb E|Y_i|^2
\leq C(1+|x|^2).
\end{equation}
Finally, since $Y_{i_0}=u_{i_0}^{\pi,1}(x)$ and the choice of $(i_0,x)$ was arbitrary, the estimate holds for $i_0=0,\ldots,N-1$ and every $x\in\mathbb R$. At $i_0=N$, it follows directly from $u_N^{\pi,1}=g$ and the growth bound on $g$. This proves \eqref{p2-eq:implicit_first_growth}.

\end{proof}

\begin{remark}\label{p2-rem:implicit_initial_energy}
Let $X^{\pi,1,i_0,x}$ denote the first-iteration forward process started from $x$ at time $t_{i_0}$. By Lemma~\ref{p2-thm:implicit_first_growth} and the forward moment estimate \eqref{p2-eq:first_forward_growth_implicit} established in its proof,
\begin{equation}
\max_{i_0\leq i\leq N}
\mathbb E\bigl|u_i^{\pi,1}(X_i^{\pi,1,i_0,x})\bigr|^2
\leq C(1+|x|^2).
\end{equation}
Moreover, Remark~\ref{p2-rem:constant_sigma_uniform_v_bound} gives $|v_i^{\pi,1}(y)|^2\leq\overline\sigma^2\bar L$ for every $i<N$ and $y\in\mathbb R$. Since $(N-i_0)h\leq T$, it follows that, for every fixed $\eta>0$,
\begin{equation}
\label{p2-eq:implicit_initial_energy_finite}
\begin{aligned}
\sup_{0\leq i_0\leq N}\sup_{x\in\mathbb R}
\frac{1}{1+|x|^2}
\Bigg[
\max_{i_0\leq i\leq N}
\mathbb E\bigl|u_i^{\pi,1}(X_i^{\pi,1,i_0,x})\bigr|^2
+\eta h\sum_{i=i_0}^{N-1}
\mathbb E\bigl|v_i^{\pi,1}(X_i^{\pi,1,i_0,x})\bigr|^2
\Bigg]
\leq C,
\end{aligned}
\end{equation}
where $C$ may depend on $\eta$ but is independent of the partition and the initial pair $(i_0,x)$. 
\end{remark}

\subsection{Convergence of the Markovian iteration}\label{p2-sec33}

We now study the convergence of the Markovian iteration, and introduce some necessary notations used in the analysis.

For an initial pair $(i_0,x)$, let $(X^{\pi,m,i_0,x},Y^{\pi,m,i_0,x}, \widehat Z^{\pi,m,i_0,x})$ denote the processes generated by iteration $m$, restarted at $X_{i_0}^{\pi,m,i_0,x}=x$. For $m\geq1$, define the successive field differences by
\begin{equation}
\delta u_i^m
\coloneqq u_i^{\pi,m}-u_i^{\pi,m-1},
\quad
0\leq i\leq N,
\qquad
\delta v_i^m
\coloneqq v_i^{\pi,m}-v_i^{\pi,m-1},
\quad
0\leq i<N.
\end{equation}
and introduce the weighted quantities
\begin{align}
\mathcal U_m
&\coloneqq
\sup_{0\leq i_0\leq N}\sup_{x\in\mathbb R}
\frac{1}{1+|x|^2}
\max_{i_0\leq i\leq N}
\mathbb E\bigl|
\delta u_i^m(X_i^{\pi,m,i_0,x})
\bigr|^2,
\label{p2-eq:def_implicit_Um}
\\
\mathcal V_m
&\coloneqq
\sup_{0\leq i_0\leq N}\sup_{x\in\mathbb R}
\frac{1}{1+|x|^2}
h
\sum_{i=i_0}^{N-1}
\mathbb E\bigl|
\delta v_i^m(X_i^{\pi,m,i_0,x})
\bigr|^2,
\label{p2-eq:def_implicit_Vm}
\\
\mathcal E_m
&\coloneqq
\mathcal U_m+\eta\mathcal V_m,
\qquad \eta>0.
\label{p2-eq:def_implicit_energy}
\end{align}
Empty sums are understood as zero. For each fixed iteration and partition, these quantities are finite, since the fields are globally Lipschitz and the restarted forward processes have second moments with at most quadratic dependence on the starting state. Remark~\ref{p2-rem:implicit_initial_energy} implies $\mathcal E_1\leq C$, where $C$ may depend on $\eta$ but is independent of the partition.

For later use, let us define the following nonnegative constants,
\begin{equation}\label{p2-eq:def_collect_constants}
\begin{alignedat}{2}
B_X&\coloneqq b_x+2b_y\bar L+2b_z\bar V,
\qquad & C_X&\coloneqq T^{-1}+2TB_X,
\\
\beta_u&\coloneqq 4b_yT^2e^{C_XT},
& \beta_v&\coloneqq 4b_zTe^{C_XT},
\\
C_Y&\coloneqq e^{\theta_2T}(g_x+\theta_3T),
& C_Z&\coloneqq \frac{g_x+\theta_2TC_Y+\theta_3T}{\theta_1},
\\
C_U&\coloneqq 2(C_Y+\bar L),
& C_V&\coloneqq 2(C_Z+\bar VT),
\end{alignedat}
\end{equation}
Here $\theta_2$, $\theta_3$, and $\theta_1$ are the constants from Lemma~\ref{p2-lem:implicit_backward_stability}. All these constants are independent of $h$ and the iteration index $m$, but may depend on the chosen invariant bounds $\bar L$ and $\bar V$.

\begin{theorem}[Geometric convergence of the Markovian iteration]
\label{p2-thm:implicit_markovian_convergence}
Suppose that the assumptions of Theorem~\ref{p2-thm:constant_sigma_invariant_class} hold. Assume that there exists $\eta>0$ such that
\begin{equation}
\label{p2-eq:implicit_rho_condition}
\rho_\eta:=(C_U+\eta C_V)
\max\left\{\beta_u,\frac{\beta_v}{\eta}\right\}<1.
\end{equation}
Then
\begin{equation}
\label{p2-eq:implicit_energy_contraction}
\mathcal E_{m+1}\leq\rho_\eta\mathcal E_m,
\qquad m\geq1,
\end{equation}
and consequently, with the convention $\rho_\eta^0=1$ also when $\rho_\eta=0$,
\begin{equation}
\label{p2-eq:implicit_geometric_energy}
\mathcal E_m\leq\rho_\eta^{m-1}\mathcal E_1,
\qquad m\geq1.
\end{equation}
For each fixed partition $\pi$, the fields $(u^{\pi,m},v^{\pi,m})$ converge pointwise in the spatial variable, as $m\to\infty$, to a fixed point $(u^\pi,v^\pi)$ of $F^\pi$, satisfying
\begin{equation}
L(u^\pi)\leq\bar L,\qquad L(v^\pi)\leq\bar V,
\qquad |v_i^\pi(x)|^2\leq\overline\sigma^2\bar L.
\end{equation}
Moreover, for some $C$ independent of $m$ and $h$,
\begin{equation} \label{p2-eq:implicit_uniform_u_growth}
|u_i^{\pi,m}(x)|^2+|u_i^\pi(x)|^2
\leq C(1+|x|^2).
\end{equation}
For every restarted initial pair $(i_0,x)$, the coupled discrete FBSDE associated with the limiting decoupling fields $(u^\pi,v^\pi)$ has a unique square-integrable solution. In particular, $(u^\pi,v^\pi)$ is the unique fixed point of $F^\pi$ in the invariant Lipschitz class.

\end{theorem}

\begin{proof}

Fix a partition $\pi$ and an initial pair $(i_0,x)$. To simplify the notation, we suppress the superscripts $\pi,i_0,x$ on the processes throughout this proof, retaining only the iteration index $m$, while keeping the full notation for the decoupling fields.

For $m\geq1$, let us write
\begin{equation}
\Delta X_i^m:=X_i^{m+1}-X_i^m,\quad
\Delta Y_i^m:=Y_i^{m+1}-Y_i^m,\quad
\Delta\widehat Z_i^m:=\widehat Z_i^{m+1}-\widehat Z_i^m.
\end{equation}
The uniform Lipschitz bounds give
\begin{align}
|u_i^{\pi,m}(X_i^{m+1})-u_i^{\pi,m-1}(X_i^m)|^2
&\leq2\bar L|\Delta X_i^m|^2
+2|\delta u_i^m(X_i^m)|^2,
\label{p2-eq:implicit_input_u_diff}
\\
|v_i^{\pi,m}(X_i^{m+1})-v_i^{\pi,m-1}(X_i^m)|^2
&\leq2\bar V|\Delta X_i^m|^2
+2|\delta v_i^m(X_i^m)|^2.
\label{p2-eq:implicit_input_v_diff}
\end{align}
Let $\Delta b_i^m$ denote the difference of the forward drifts in iterations $m+1$ and $m$, that is, 
\begin{equation}
\Delta b_i^m
\coloneqq 
b\bigl(
t_i,X_i^{m+1},
u_i^{\pi,m}(X_i^{m+1}),
v_i^{\pi,m}(X_i^{m+1})
\bigr)
-
b\bigl(
t_i,X_i^m,
u_i^{\pi,m-1}(X_i^m),
v_i^{\pi,m-1}(X_i^m)
\bigr).
\end{equation}
By Assumption~\ref{p2-ass:implicit_lipschitz},
\begin{equation}
\label{p2-eq:implicit_delta_b}
|\Delta b_i^m|^2\leq
B_X|\Delta X_i^m|^2
+2b_y|\delta u_i^m(X_i^m)|^2
+2b_z|\delta v_i^m(X_i^m)|^2.
\end{equation}
Since both iterations use the same deterministic coefficient $\sigma_i$ at time $t_i$, the diffusion increments cancel, 
\begin{equation}
\Delta X_{i+1}^m=\Delta X_i^m+h\Delta b_i^m,
\qquad\Delta X_{i_0}^m=0.
\end{equation}
Using Young's inequality with the form $2ab\leq T^{-1}a^2+Tb^2$, and $h\leq T$, we obtain
\begin{equation}
|\Delta X_{i+1}^m|^2
\leq(1+C_Xh)|\Delta X_i^m|^2
+4Tb_yh|\delta u_i^m(X_i^m)|^2
+4Tb_zh|\delta v_i^m(X_i^m)|^2.
\end{equation}
Taking expectations and applying discrete Gronwall inequality gives
\begin{equation}
\label{p2-eq:implicit_forward_difference}
\begin{aligned}
\max_{i_0\leq i\leq N}\mathbb E|\Delta X_i^m|^2
\leq 
\beta_u\max_{i_0\leq i\leq N}
\mathbb E|\delta u_i^m(X_i^m)|^2
+\beta_vh\sum_{i=i_0}^{N-1}
\mathbb E|\delta v_i^m(X_i^m)|^2.
\end{aligned}
\end{equation}
For the backward differences, Lemma~\ref{p2-lem:implicit_backward_stability} yields
\begin{equation}
\label{p2-eq:implicit_backward_difference_onestep}
\mathbb E|\Delta Y_i^m|^2
+\theta_1h\mathbb E|\Delta\widehat Z_i^m|^2
\leq(1+\theta_2h)\mathbb E|\Delta Y_{i+1}^m|^2
+\theta_3h\mathbb E|\Delta X_i^m|^2.
\end{equation}
At the terminal time, Assumption~\ref{p2-ass:implicit_lipschitz} gives
\begin{equation} \label{p2-eq:terminalY_g}
\mathbb E|\Delta Y_N^m|^2
=
\mathbb E
\left|
g(X_N^{m+1})-g(X_N^m)
\right|^2
\leq
g_x
\mathbb E|\Delta X_N^m|^2.
\end{equation}
Dropping the nonnegative $Z$-term in \eqref{p2-eq:implicit_backward_difference_onestep}, iterating backward from $N$ to $i$ gives
\begin{align*}
\mathbb E|\Delta Y_i^m|^2
\leq &
(1+\theta_2h)^{N-i}
\mathbb E|\Delta Y_N^m|^2
+
\theta_3h
\sum_{j=i}^{N-1}
(1+\theta_2h)^{j-i}
\mathbb E|\Delta X_j^m|^2  \\
\leq & 
e^{\theta_2T}
g_x
\max_{i_0\leq j\leq N}
\mathbb E|\Delta X_j^m|^2
+
\theta_3Te^{\theta_2T}
\max_{i_0\leq j\leq N}
\mathbb E|\Delta X_j^m|^2  \\
= & 
e^{\theta_2T}
\left(
g_x+\theta_3T
\right)
\max_{i_0\leq j\leq N}
\mathbb E|\Delta X_j^m|^2.
\end{align*}
and therefore we can write 
\begin{equation}
\label{p2-eq:implicit_backward_Y_difference}
\max_{i_0\leq i\leq N}\mathbb E|\Delta Y_i^m|^2
\leq C_Y \max_{i_0\leq i\leq N}\mathbb E|\Delta X_i^m|^2.
\end{equation}
For $Z$-difference, summing \eqref{p2-eq:implicit_backward_difference_onestep} gives
\begin{equation}
\begin{aligned}
\theta_1h\sum_{i=i_0}^{N-1}
\mathbb E|\Delta\widehat Z_i^m|^2
\leq{}&\mathbb E|\Delta Y_N^m|^2
-\mathbb E|\Delta Y_{i_0}^m|^2
\\
&+\theta_2h\sum_{i=i_0}^{N-1}\mathbb E|\Delta Y_{i+1}^m|^2
+\theta_3h\sum_{i=i_0}^{N-1}\mathbb E|\Delta X_i^m|^2.
\end{aligned}
\end{equation}
Discarding the nonpositive term $-\mathbb E|\Delta Y_{i_0}^m|^2$, using \eqref{p2-eq:implicit_backward_Y_difference}, \eqref{p2-eq:terminalY_g} and taking the maximum of $\mathbb E|\Delta X_i^m|^2$ yields
\begin{equation}
\label{p2-eq:implicit_backward_Z_difference}
h\sum_{i=i_0}^{N-1}\mathbb E|\Delta\widehat Z_i^m|^2
\leq C_Z \max_{i_0\leq i\leq N}\mathbb E|\Delta X_i^m|^2 .
\end{equation}
We now transfer these process estimates to the fields. Using the identities $Y_i^{m+1}=u_i^{\pi,m+1}(X_i^{m+1})$ and $Y_i^m=u_i^{\pi,m}(X_i^m)$, the uniform Lipschitz bound, and Young's inequality, we obtain
\begin{align*}
\mathbb E
\left|
\delta u_i^{m+1}(X_i^{m+1})
\right|^2
= &
\mathbb E \left|
\Delta Y_i^m
+
u_i^{\pi,m}(X_i^m)
-
u_i^{\pi,m}(X_i^{m+1})  
\right|^2 \\
\leq & 
2\mathbb E|\Delta Y_i^m|^2
+
2\bar L\mathbb E|\Delta X_i^m|^2.
\end{align*}
Therefore, by \eqref{p2-eq:implicit_backward_Y_difference}, we have
\begin{equation}
\label{p2-eq:implicit_U_next}
\max_{i_0\leq i\leq N}
\mathbb E
\left|
\delta u_i^{m+1}(X_i^{m+1})
\right|^2
\leq
C_U
\max_{i_0\leq i\leq N}
\mathbb E|\Delta X_i^m|^2.
\end{equation}
Similarly for the $Z$ process, we have
\begin{equation}
\delta v_i^{m+1}(X_i^{m+1})
=
\Delta\widehat Z_i^m
+
v_i^{\pi,m}(X_i^m)
-
v_i^{\pi,m}(X_i^{m+1}).
\end{equation} 
Using \eqref{p2-eq:implicit_backward_Z_difference}, and $h\sum_{i=i_0}^{N-1} \mathbb E|\Delta X_i^m|^2 \leq T \max_{i_0\leq i\leq N} \mathbb E|\Delta X_i^m|^2$ gives
\begin{equation}\label{p2-eq:implicit_V_next}
\begin{aligned}
h
\sum_{i=i_0}^{N-1}
\mathbb E
\left|
\delta v_i^{m+1}(X_i^{m+1})
\right|^2
&
\leq
2h
\sum_{i=i_0}^{N-1}
\mathbb E|\Delta\widehat Z_i^m|^2
+
2\bar Vh
\sum_{i=i_0}^{N-1}
\mathbb E|\Delta X_i^m|^2
\\
&
\leq
C_V
\max_{i_0\leq i\leq N}
\mathbb E|\Delta X_i^m|^2,
\end{aligned}
\end{equation}
Combining inequalities \eqref{p2-eq:implicit_U_next} and \eqref{p2-eq:implicit_V_next}  with \eqref{p2-eq:implicit_forward_difference}, dividing by $1+|x|^2$, and taking suprema over $(i_0,x)$ yields
\begin{equation}
\mathcal U_{m+1}\leq C_U(\beta_u\mathcal U_m+\beta_v\mathcal V_m),
\qquad
\mathcal V_{m+1}\leq C_V(\beta_u\mathcal U_m+\beta_v\mathcal V_m).
\end{equation}
Consequently,
\begin{equation}
\begin{aligned}
\mathcal E_{m+1}
\leq(C_U+\eta C_V)(\beta_u\mathcal U_m+\beta_v\mathcal V_m)
\leq(C_U+\eta C_V)
\max\{\beta_u,\beta_v/\eta\}\mathcal E_m
=\rho_\eta\mathcal E_m.
\end{aligned}
\end{equation}
This proves \eqref{p2-eq:implicit_energy_contraction}, and \eqref{p2-eq:implicit_geometric_energy} follows by iterating the estimate.

Next, we prove the convergence of $u_i^{\pi,m}$ and $v_i^{\pi,m}$. Since the preceding estimates hold for every initial pair $(i_0,x)$, we may choose $i_0=i$. The definitions of $\mathcal U_m$ and $\mathcal V_m$ then yield 
\begin{equation}
|\delta u_i^m(x)|^2
\leq
(1+|x|^2)\mathcal U_m
\leq
(1+|x|^2)
\rho_\eta^{m-1}\mathcal E_1
,\quad 
h|\delta v_i^m(x)|^2
\leq
(1+|x|^2)\mathcal V_m
\leq
\frac{1+|x|^2}{\eta}
\rho_\eta^{m-1}\mathcal E_1 ,
\end{equation}
which implies
\begin{equation}
\label{p2-eq:implicit_pointwise_iteration_difference}
\begin{aligned}
|\delta u_i^m(x)|
\leq
(1+|x|^2)^{1/2}
\mathcal E_1^{1/2}
\rho_\eta^{(m-1)/2}
,\qquad
|\delta v_i^m(x)|
\leq
\left(
\frac{1+|x|^2}{\eta h}
\right)^{1/2}
\mathcal E_1^{1/2}
\rho_\eta^{(m-1)/2}.
\end{aligned}
\end{equation}
Hence, for any integers $n>m$ we can write 
\begin{equation}
\begin{aligned}
& |u_i^{\pi,n}(x)-u_i^{\pi,m}(x)|
=
\left|
\sum_{k=m+1}^{n}
\delta u_i^k(x)
\right|
\leq
\sum_{k=m+1}^{n}
|\delta u_i^k(x)|
\leq
(1+|x|^2)^{1/2}
\mathcal E_1^{1/2}
\sum_{k=m+1}^{n}
\rho_\eta^{(k-1)/2}, \\
& |v_i^{\pi,n}(x)-v_i^{\pi,m}(x)|
\leq
\left(
\frac{1+|x|^2}{\eta h}
\right)^{1/2}
\mathcal E_1^{1/2}
\sum_{k=m+1}^{n}
\rho_\eta^{(k-1)/2}.
\end{aligned}
\end{equation}
Since $\rho_\eta<1$, the geometric series $ \sum_{k=1}^{\infty} \rho_\eta^{(k-1)/2} $ converges, and hence $ \sum_{k=m+1}^{n} \rho_\eta^{(k-1)/2} \longrightarrow 0 $ as $m\to\infty$. Therefore, $\{u_i^{\pi,m}(x)\}_{m\geq0}$ and $\{v_i^{\pi,m}(x)\}_{m\geq0}$ are Cauchy sequences for every fixed admissible $i$ and $x$, and we can therefore define
\begin{equation}
u_i^\pi(x)
\coloneqq
\lim_{m\to\infty}u_i^{\pi,m}(x),
\qquad
v_i^\pi(x)
\coloneqq
\lim_{m\to\infty}v_i^{\pi,m}(x).
\end{equation}
The uniform bounds pass to the limit, and hence $L(u^\pi)\leq\bar L$, $L(v^\pi)\leq\bar V$, and $|v_i^\pi(x)|^2\leq\overline\sigma^2\bar L$ for every $0\leq i<N$ and $x\in\mathbb R$.

Moreover, for $0\leq i<N$, summing the first inequality in \eqref{p2-eq:implicit_pointwise_iteration_difference} and using $u_i^{\pi,0}=0$ gives
\begin{equation}
|u_i^{\pi,m}(x)|
\leq
\frac{\mathcal E_1^{1/2}}{1-\sqrt{\rho_\eta}}
(1+|x|^2)^{1/2}.
\end{equation}
By Remark~\ref{p2-rem:implicit_initial_energy}, $\mathcal E_1$ is bounded independently of the partition. Taking the limit as $m\to\infty$ gives the same growth bound for $u_i^\pi$. At $i=N$, we have $u_N^{\pi,m}=u_N^\pi=g$, so the terminal growth bound gives the same conclusion. This proves \eqref{p2-eq:implicit_uniform_u_growth}.


Furthermore, \eqref{p2-eq:implicit_forward_difference}, \eqref{p2-eq:implicit_backward_Y_difference}, and \eqref{p2-eq:implicit_backward_Z_difference}, together with \eqref{p2-eq:implicit_geometric_energy}, imply that the corresponding forward and backward processes are Cauchy in their natural discrete $L^2$-norms. Using the uniform Lipschitz and growth bounds on the fields, the global Lipschitz continuity of the coefficients, and the $L^2$-continuity of the conditional projections for fixed $h$, we may pass to the limit in the discrete scheme \eqref{p2-eq:implicit_full_scheme}. Since the initial pair $(i_0,x)$ is arbitrary, the limiting fields satisfy
\begin{equation}
(u^\pi,v^\pi)=F^\pi(u^\pi,v^\pi)
\end{equation}
pointwise, and hence form a fixed point of $F^\pi$.

Finally, we prove the uniqueness. Let $(X^{(1)},Y^{(1)},\widehat Z^{(1)})$ and $(X^{(2)},Y^{(2)},\widehat Z^{(2)})$ be two square-integrable solutions of the discrete FBSDE with the same initial pair $(i_0,x)$ and driven by the same Brownian increments. We apply the preceding arguments to these two solutions.

Their forward drifts are $b(t_i,X_i^{(\ell)},Y_i^{(\ell)},\widehat Z_i^{(\ell)})$, $\ell=1,2$, and the identical diffusion terms $\sigma_i\Delta W_i$ cancel. The forward estimate therefore gives
\begin{equation}
\begin{aligned}
\max_{i_0\leq i\leq N}
\mathbb E|X_i^{(1)}-X_i^{(2)}|^2
\leq
e^{C_X^0T}\Bigg[
2b_yT^2
\max_{i_0\leq i\leq N}
\mathbb E|Y_i^{(1)}-Y_i^{(2)}|^2
+
2b_zTh\sum_{i=i_0}^{N-1}
\mathbb E|\widehat Z_i^{(1)}-\widehat Z_i^{(2)}|^2
\Bigg],
\end{aligned}
\end{equation}
where $C_X^0\coloneqq T^{-1}+2Tb_x$. The backward estimates apply directly to these solutions, yielding
\begin{equation}
\begin{aligned}
\max_{i_0\leq i\leq N}
\mathbb E|Y_i^{(1)}-Y_i^{(2)}|^2
&\leq
C_Y\max_{i_0\leq i\leq N}
\mathbb E|X_i^{(1)}-X_i^{(2)}|^2,
\\
h\sum_{i=i_0}^{N-1}
\mathbb E|\widehat Z_i^{(1)}-\widehat Z_i^{(2)}|^2
&\leq
C_Z\max_{i_0\leq i\leq N}
\mathbb E|X_i^{(1)}-X_i^{(2)}|^2.
\end{aligned}
\end{equation}
Consequently,
\begin{equation}
\begin{aligned}
\max_{i_0\leq i\leq N}
\mathbb E|X_i^{(1)}-X_i^{(2)}|^2
\leq
e^{C_X^0T}(2b_yT^2C_Y+2b_zTC_Z)
\max_{i_0\leq i\leq N}
\mathbb E|X_i^{(1)}-X_i^{(2)}|^2.
\end{aligned}
\end{equation}
Since $C_X^0\leq C_X$, $C_U\geq2C_Y$, and $C_V\geq2C_Z$, we have
\begin{equation}
e^{C_X^0T}(2b_yT^2C_Y+2b_zTC_Z)
\leq
\frac14(C_U\beta_u+C_V\beta_v)
\leq
\frac14\rho_\eta
<1.
\end{equation}
It follows that $\max_{i_0\leq i\leq N} \mathbb E|X_i^{(1)}-X_i^{(2)}|^2=0$, and the backward estimates imply equality of the two backward solutions. Since the initial pair $(i_0,x)$ is arbitrary, this also proves uniqueness of the fixed point in the invariant Lipschitz class of Theorem~\ref{p2-thm:constant_sigma_invariant_class}.

\end{proof}

\begin{remark}[Interpretation of the contraction condition]
\label{p2-rem:implicit_weak_coupling}
To understand when the contraction condition \eqref{p2-eq:implicit_rho_condition} holds, recall that
\begin{equation}
\beta_u=4b_yT^2e^{1+2B_XT^2},
\qquad
\beta_v=4b_zTe^{1+2B_XT^2}.
\end{equation}
With the coefficient bounds and invariant Lipschitz bounds $\bar L,\bar V$ fixed, both coefficients vanish as $T\downarrow0$.

Fix the backward stability constants $\theta_2,\theta_3,\theta_1$ corresponding to a sufficiently small admissible maximal step size $h_0$. Then, as $T\downarrow0$, 
\begin{equation}
\begin{aligned}
C_U =2(C_Y+\bar L)
\longrightarrow 2(g_x+\bar L),
\quad
C_V =2(C_Z+\bar VT)
\longrightarrow \frac{2g_x}{\theta_1}.
\end{aligned}
\end{equation}
Thus $C_U$ and $C_V$ remain bounded as $T\downarrow0$. Since $\beta_u,\beta_v\to0$, for every fixed $\eta>0$,
\begin{equation}
\rho_\eta
=
(C_U+\eta C_V)
\max\left\{\beta_u,\frac{\beta_v}{\eta}\right\}
\longrightarrow0.
\end{equation}
Together with the small-time feasibility of the invariant conditions as discussed in Remark \ref{p2-rem:constant_sigma_invariant_feasibility}, this yields a nonempty convergence regime.

For a fixed $T$, sufficiently weak drift coupling, as measured by $b_y$ and $b_z$, may also ensure \eqref{p2-eq:implicit_rho_condition}, provided the invariant conditions remain satisfied. Reducing the spatial Lipschitz bound $b_x$ can also improve this condition. Indeed, with the other constants fixed, $B_X$, $\beta_u$, $\beta_v$, and $\rho_\eta$ are nondecreasing functions of $b_x,b_y$, and $b_z$.

\end{remark}

\subsection{Convergence of the time discretization}\label{p2-sec34}

We now consider the time-discretization error. To this end, we require the following additional assumptions on the coefficients and the exact decoupling fields of the continuous FBSDE \eqref{p2-eq:continuousFBSDE}.

\begin{assumption}[Time regularity of the coefficients]
\label{p2-ass:implicit_time_regular_coefficients}
The coefficients $b$, $f$, and $\sigma$ are uniformly $1/2$-H\"older continuous in time. More precisely, there exists a constant $C>0$ such that
\begin{equation}
\label{p2-eq:time_regular_coefficients}
\begin{aligned}
&|b(t,x,y,z)-b(s,x,y,z)|^2
+
|f(t,x,y,z)-f(s,x,y,z)|^2
+
|\sigma(t)-\sigma(s)|^2
\leq C|t-s|,
\end{aligned}
\end{equation}
for all $s,t\in[0,T]$ and $x,y,z\in\mathbb R$.
\end{assumption}

\begin{assumption}[Regularity of the continuous solution]
\label{p2-ass:implicit_solution_regular}
The following properties hold:
\begin{enumerate}[label=(\roman*),itemsep=3pt]

\item\label{p2-item:exact_solution_existence}
For every initial state $x_0\in\mathbb R$, the continuous-time FBSDE \eqref{p2-eq:continuousFBSDE} has a square-integrable solution $(X,Y,Z)$ with common deterministic decoupling fields $u$ and $v$, such that $Y_t=u(t,X_t)$ and $Z_t=v(t,X_t)$ $\mathbb P$-almost surely for every $t\in[0,T]$.

\item\label{p2-item:exact_field_spatial_lipschitz}
The fields $u,v$ are uniformly Lipschitz in the spatial variable over $t\in[0,T]$.

\item\label{p2-item:exact_field_time_holder}
The fields $u,v$ are $1/2$-H\"older continuous in time with the following spatially weighted bound:
\begin{equation}
\label{p2-eq:implicit_exact_field_time_holder}
|u(t,x)-u(s,x)|^2
+
|v(t,x)-v(s,x)|^2
\leq C(1+|x|^2)|t-s|,
\end{equation}
for all $s,t\in[0,T]$ and $x\in\mathbb R$, where $C$ is independent of $s,t$, and $x$.


\item\label{p2-item:exact_solution_time_regularity}
For every admissible uniform partition, the solution satisfies
\begin{equation}\label{p2-eq:implicit_XYZ_time_regular}
\begin{aligned}
\max_{0\leq i<N}
\mathbb E\left[
\sup_{t\in[t_i,t_{i+1}]}
\bigl(
|X_t-X_{t_i}|^2+|Y_t-Y_{t_i}|^2
\bigr)
\right]
+
\sum_{i=0}^{N-1}\mathbb E
\int_{t_i}^{t_{i+1}}
|Z_t-Z_{t_i}|^2\,\mathrm dt
\leq C(1+|x_0|^2)h.
\end{aligned}
\end{equation}

\end{enumerate}
All the constants $C$ are independent of the partition and the initial state.
\end{assumption}

\begin{remark}\label{p2-rem:exact_solution_regularity}
We state the properties in Assumption 5 directly to keep the time-discretization analysis focused on the numerical scheme. In particular, this separates the regularity requirements needed for the time-discretization error from the convergence analysis of the Markovian iteration, which is the main focus of this work. Some of these properties can be derived from Assumptions 1--4 on a sufficiently short time interval. In particular, Theorems~8.2.1 and~8.2.3 of \cite{zhang2017backward} establish existence and uniqueness of a square-integrable solution to the continuous FBSDE, together with spatial Lipschitz continuity of its $Y$-decoupling field. The additional compatibility condition in equation~(8.2.1) of that reference is automatically satisfied, since $\sigma$ depends only on time in our setting.

Furthermore, under the standing coefficient assumptions, items~\ref{p2-item:exact_solution_existence}-- \ref{p2-item:exact_field_time_holder} of Assumption~\ref{p2-ass:implicit_solution_regular} imply item~\ref{p2-item:exact_solution_time_regularity}. Indeed, fixing the exact decoupling fields gives the effective forward drift
\begin{equation}
B(t,x)\coloneqq b(t,x,u(t,x),v(t,x)).
\end{equation}
The spatial Lipschitz and time regularity assumptions imply
\begin{equation}
|B(t,x)-B(t,x')|^2\leq C|x-x'|^2,
\qquad
|B(t,x)-B(s,x)|^2
\leq C(1+|x|^2)|t-s|.
\end{equation}

Moreover, the time regularity estimate at $x=0$ shows that $B(t,0)$ is bounded uniformly in $t\in[0,T]$. Together with spatial Lipschitz continuity, this yields the uniform linear-growth bound for $B(t,x)$. Since $\sigma$ is bounded, the standard moment and increment estimates for the forward SDE therefore give
\begin{equation}
\mathbb E\left[\sup_{0\leq t\leq T}|X_t|^2\right]
\leq C(1+|x_0|^2),
\qquad
\mathbb E\left[
\sup_{r\in[s,t]}|X_r-X_s|^2
\right]
\leq C(1+|x_0|^2)(t-s),
\end{equation}
for all $0\leq s\leq t\leq T$, where $C$ is independent of $s,t$, and $x_0$.

Using the spatial Lipschitz continuity and time regularity of $u$, we obtain, for $0\leq s\leq t\leq T$,
\begin{equation}
\begin{aligned}
\mathbb E\left[
\sup_{r\in[s,t]}|Y_r-Y_s|^2
\right]
&\leq
2\mathbb E\left[
\sup_{r\in[s,t]}
|u(r,X_r)-u(r,X_s)|^2
\right]
+
2\mathbb E\left[
\sup_{r\in[s,t]}
|u(r,X_s)-u(s,X_s)|^2
\right]
\\
&\leq
C\mathbb E\left[
\sup_{r\in[s,t]}|X_r-X_s|^2
\right]
+
C(1+\mathbb E|X_s|^2)(t-s)
\\
&\leq C(1+|x_0|^2)(t-s).
\end{aligned}
\end{equation}
Similarly, the spatial and time regularity of $v$ gives
\begin{equation}
\begin{aligned}
\mathbb E|Z_t-Z_s|^2
&\leq
2\mathbb E|v(t,X_t)-v(t,X_s)|^2
+
2\mathbb E|v(t,X_s)-v(s,X_s)|^2
\\
&\leq
C\mathbb E|X_t-X_s|^2
+
C(1+\mathbb E|X_s|^2)(t-s)
\\
&\leq C(1+|x_0|^2)(t-s).
\end{aligned}
\end{equation}
Here, the last inequalities use the forward moment and increment estimates established above.

Integrating the estimate for $Z$ over each time interval and summing over $i$ gives
\begin{equation}
\begin{aligned}
\sum_{i=0}^{N-1}
\mathbb E\int_{t_i}^{t_{i+1}}
|Z_t-Z_{t_i}|^2\,\mathrm dt
\leq
C(1+|x_0|^2)
\sum_{i=0}^{N-1}
\int_{t_i}^{t_{i+1}}(t-t_i)\,\mathrm dt
=
\frac{CT}{2}(1+|x_0|^2)h.
\end{aligned}
\end{equation}
Taking $s=t_i$ and $t=t_{i+1}$ in the increment estimates for $X$ and $Y$, and then taking the maximum over $i$, yields the remaining term in \eqref{p2-eq:implicit_XYZ_time_regular}. Thus, after absorbing the fixed factor $T/2$ into the constant, we obtain \eqref{p2-eq:implicit_XYZ_time_regular}, which we retain explicitly for convenience.
\end{remark}

Next, we shall apply one Markovian iteration to the exact continuous-time decoupling fields and compare its output with the discrete fixed point. To this end, we define the exact grid fields by
\begin{equation}
u_i^0(x)\coloneqq u(t_i,x),
\quad 0\leq i\leq N,
\qquad
v_i^0(x)\coloneqq v(t_i,x),
\quad 0\leq i<N.
\end{equation}
Here $u_N^0=g$. The pair $(u^0,v^0)$ denotes the exact grid fields, whereas $(u^{\pi,0},v^{\pi,0})$ denotes the initialization of the Markovian iteration, which is zero for $i<N$ and satisfies $u_N^{\pi,0}=g$.

Let $(\widetilde u^\pi,\widetilde v^\pi) \coloneqq F^\pi(u^0,v^0)$ be the approximation obtained after one iteration with $(u^0,v^0)$ as input fields. The associated processes satisfy
\begin{equation}
\label{p2-eq:implicit_one_iteration_process}
\left\{
\begin{aligned}
\widetilde X_0^\pi
&=x_0,
\\
\widetilde X_{i+1}^\pi
&=
\widetilde X_i^\pi
+
h b\bigl(
t_i,\widetilde X_i^\pi,
u_i^0(\widetilde X_i^\pi),
v_i^0(\widetilde X_i^\pi)
\bigr)
+
\sigma(t_i) \Delta W_i,
\\
\widetilde Y_N^\pi
&=g(\widetilde X_N^\pi),
\\
\widetilde Z_i^\pi
&=
h^{-1}\mathbb E_{t_i}
\left[\widetilde Y_{i+1}^\pi\Delta W_i\right],
\\
\widetilde Y_i^\pi
&=
\mathbb E_{t_i}\left[\widetilde Y_{i+1}^\pi\right]
+
h f\bigl(
t_i,\widetilde X_i^\pi,
\widetilde Y_i^\pi,\widetilde Z_i^\pi
\bigr).
\end{aligned}
\right.
\end{equation}

\begin{lemma}[Consistency of one implicit discrete iteration]
\label{p2-lem:implicit_one_iteration_consistency}
Let Assumptions~\ref{p2-ass:implicit_lipschitz}, \ref{p2-ass:constant_sigma}, \ref{p2-ass:implicit_time_regular_coefficients}, and \ref{p2-ass:implicit_solution_regular} hold. For sufficiently small $h_0>0$, there exists a constant $C>0$, independent of $\pi$ and $x_0$, such that for every $0<h\leq h_0$, the following estimates hold:
\begin{equation}\label{p2-eq:implicit_process_consistency}
\begin{aligned}
\max_{0\leq i<N}\mathbb E\left[
\sup_{t\in[t_i,t_{i+1}]}
\bigl(|X_t-\widetilde X_i^\pi|^2
+|Y_t-\widetilde Y_i^\pi|^2\bigr)\right]
+
\sum_{i=0}^{N-1}\mathbb E\int_{t_i}^{t_{i+1}}
|Z_t-\widetilde Z_i^\pi|^2\,\mathrm dt
\leq C(1+|x_0|^2)h.
\end{aligned}
\end{equation}
Moreover,
\begin{equation}\label{p2-eq:implicit_consistency_residual}
\begin{aligned}
\max_{0\leq i\leq N}
\mathbb E\bigl|(
\widetilde u_i^\pi-u_i^0)(\widetilde X_i^\pi)\bigr|^2
+
h\sum_{i=0}^{N-1}
\mathbb E\bigl|(
\widetilde v_i^\pi-v_i^0)(\widetilde X_i^\pi)\bigr|^2
\leq C(1+|x_0|^2)h.
\end{aligned}
\end{equation}
\end{lemma}

\begin{proof}

To prove \eqref{p2-eq:implicit_process_consistency}, fix the exact decoupling fields and use the effective drift $B(t,x)=b(t,x,u(t,x),v(t,x))$ introduced in Remark~\ref{p2-rem:exact_solution_regularity}. Its uniform spatial Lipschitz continuity and time regularity, together with the moment bound for $X$ and the time regularity of $\sigma$, give the standard Euler--Maruyama estimate
\begin{equation}
\max_{0\leq i\leq N}
\mathbb E|X_{t_i}-\widetilde X_i^\pi|^2
\leq C(1+|x_0|^2)h.
\end{equation}
Using $|a+b|^2\leq2|a|^2+2|b|^2$ and the increment bound in \eqref{p2-eq:implicit_XYZ_time_regular}, we obtain, for every $0\leq i<N$,
\begin{equation}
\begin{aligned}
\mathbb E\left[
\sup_{s\in[t_i,t_{i+1}]}
|X_s-\widetilde X_i^\pi|^2
\right]
&\leq
2\mathbb E\left[
\sup_{s\in[t_i,t_{i+1}]}
|X_s-X_{t_i}|^2
\right]
+
2\mathbb E|X_{t_i}-\widetilde X_i^\pi|^2
\leq C(1+|x_0|^2)h.
\end{aligned}
\end{equation}
Taking the maximum over $i$ proves the forward estimate required in \eqref{p2-eq:implicit_process_consistency}.

Next, for the backward processes, we follow the proof of Theorem~3.1 in \cite{bouchard2004discrete}, whose implicit scheme (3.3)--(3.4) coincides with our backward recursion. To facilitate comparison, their discrete values $(X_{t_i}^\pi,Y_{t_i}^\pi,Z_{t_i}^\pi)$ correspond to our $(\widetilde X_i^\pi,\widetilde Y_i^\pi,\widetilde Z_i^\pi)$. We denote the corresponding continuous martingale integrand, defined as in their equation (3.5), by $\mathcal Z_t^\pi$, and the conditional time average $\overline Z_{t_i}^\pi$ in their Lemma~3.2 by $\overline Z_i$.

We first verify the identities and regularity estimates needed in their proof. The forward estimate above implies $\mathbb E|X_s-\widetilde X_i^\pi|^2 \leq C(1+|x_0|^2)h$ for every $s\in[t_i,t_{i+1}]$, providing the counterpart of their inequality (3.7). Moreover, the definition of $\widetilde Z_i^\pi$ and conditional It\^o isometry give the identity in their Lemma~3.1.

To verify the regularity estimate in their Lemma~3.2, recall that the conditional time average minimizes the integrated squared error over square-integrable, $\mathcal F_{t_i}$-measurable random variables. Thus,
\begin{equation}
\label{p2-eq:consistency_exact_Z_average}
\overline Z_i
\coloneqq
\frac1h\mathbb E_{t_i}
\int_{t_i}^{t_{i+1}}Z_s\,\mathrm ds,
\qquad
\mathbb E\int_{t_i}^{t_{i+1}}
|Z_s-\overline Z_i|^2\,\mathrm ds
\leq
\mathbb E\int_{t_i}^{t_{i+1}}
|Z_s-Z_{t_i}|^2\,\mathrm ds.
\end{equation}
Consequently, \eqref{p2-eq:implicit_XYZ_time_regular} implies
\begin{equation}
\label{p2-eq:consistency_Y_Z_average_regularity}
\max_{0\leq i<N}
\sup_{s\in[t_i,t_{i+1}]}
\mathbb E|Y_s-Y_{t_i}|^2
+
\sum_{i=0}^{N-1}
\mathbb E\int_{t_i}^{t_{i+1}}
|Z_s-\overline Z_i|^2\,\mathrm ds
\leq C(1+|x_0|^2)h.
\end{equation}
Thus, the regularity estimate in their Lemma~3.2 holds under our assumptions.

The remaining difference is that our driver is uniformly $1/2$-H\"older continuous in time, whereas their condition (2.3) assumes Lipschitz continuity in time. In their inequality (3.6), this replaces the squared time-step size contribution $h^2$ by $h$, so its integral over one time step changes from order $h^3$ to order $h^2$. However, their estimates (3.7)--(3.8) already contribute terms of order $h$ inside the integral in (3.10), arising from the forward approximation error and the $Y$-increments. After integration, these give the $h^2$ term in $B_i$ following (3.12). Thus, our weaker time regularity of the driver $f$ preserves the order of the one-step estimate following their (3.13).

Using these estimates, we follow the proof of Theorem~3.1 of \cite{bouchard2004discrete} to obtain the same type of one-step estimate as the one following their equation (3.13), for sufficiently small $h$.

Since $g$ is globally Lipschitz, the forward error estimate gives
\begin{equation}
\mathbb E|Y_T-\widetilde Y_N^\pi|^2
=
\mathbb E|g(X_T)-g(\widetilde X_N^\pi)|^2
\leq C(1+|x_0|^2)h.
\end{equation}
Backward iteration and summation of the one-step estimate, following their equation (3.14) and the subsequent argument on page~183, respectively, yield
\begin{equation}
\label{p2-eq:consistency_backward_error_314}
\max_{0\leq i\leq N}
\mathbb E|Y_{t_i}-\widetilde Y_i^\pi|^2
\leq C(1+|x_0|^2)h,
\qquad
\sum_{i=0}^{N-1}
\mathbb E\int_{t_i}^{t_{i+1}}
|Z_s-\mathcal Z_s^\pi|^2\,\mathrm ds
\leq C(1+|x_0|^2)h.
\end{equation}
The factor $1+|x_0|^2$ follows from the terminal error, the forward error estimate, and the solution regularity bounds used in the argument.

To obtain the estimate for the discrete $Z$-values, we integrate and sum their inequality (3.9), interchanging expectation and integration by Tonelli's theorem. The second term in that inequality is independent of $s$, so integration over $[t_i,t_{i+1}]$ cancels the factor $h^{-1}$. This gives
\begin{equation}
\label{p2-eq:consistency_discrete_Z_error}
\begin{aligned}
\sum_{i=0}^{N-1}
\mathbb E\int_{t_i}^{t_{i+1}}
|Z_s-\widetilde Z_i^\pi|^2\,\mathrm ds
&\leq
2\sum_{i=0}^{N-1}
\int_{t_i}^{t_{i+1}}
\left[
\mathbb E|Z_s-\overline Z_i|^2
+
\frac1h
\mathbb E\int_{t_i}^{t_{i+1}}
|Z_r-\mathcal Z_r^\pi|^2\,\mathrm dr
\right]\,\mathrm ds
\\
&=
2\sum_{i=0}^{N-1}
\mathbb E\int_{t_i}^{t_{i+1}}
|Z_s-\overline Z_i|^2\,\mathrm ds
+
2\sum_{i=0}^{N-1}
\mathbb E\int_{t_i}^{t_{i+1}}
|Z_r-\mathcal Z_r^\pi|^2\,\mathrm dr
\\
&\leq C(1+|x_0|^2)h,
\end{aligned}
\end{equation}
where the last inequality follows from \eqref{p2-eq:consistency_Y_Z_average_regularity} and \eqref{p2-eq:consistency_backward_error_314}.

Finally, the increment bound \eqref{p2-eq:implicit_XYZ_time_regular} and \eqref{p2-eq:consistency_backward_error_314} give, for every $0\leq i<N$,
\begin{equation}
\begin{aligned}
\mathbb E\left[
\sup_{t\in[t_i,t_{i+1}]}
|Y_t-\widetilde Y_i^\pi|^2
\right]
&\leq
2\mathbb E\left[
\sup_{t\in[t_i,t_{i+1}]}
|Y_t-Y_{t_i}|^2
\right]
+
2\mathbb E|Y_{t_i}-\widetilde Y_i^\pi|^2
\leq C(1+|x_0|^2)h.
\end{aligned}
\end{equation}
Taking the maximum over $i$, and combining previous results proves \eqref{p2-eq:implicit_process_consistency}.


It remains to derive \eqref{p2-eq:implicit_consistency_residual}. For $0\leq i<N$, using $\widetilde Y_i^\pi=\widetilde u_i^\pi(\widetilde X_i^\pi)$, $Y_{t_i}=u_i^0(X_{t_i})$, and Assumption \ref{p2-ass:implicit_solution_regular}, we obtain
\begin{equation}
\begin{aligned}
\left|
\widetilde u_i^\pi(\widetilde X_i^\pi)
-
u_i^0(\widetilde X_i^\pi)
\right|^2
& \leq
2
|\widetilde Y_i^\pi-Y_{t_i}|^2
+
2
|u_i^0(X_{t_i})-u_i^0(\widetilde X_i^\pi)|^2
\leq
2
|\widetilde Y_i^\pi-Y_{t_i}|^2
+
C|X_{t_i}-\widetilde X_i^\pi|^2,
\end{aligned}
\end{equation}
At the terminal time, $\widetilde u_N^\pi=u_N^0=g$, so the corresponding field difference vanishes. Therefore, \eqref{p2-eq:implicit_process_consistency} implies
\begin{equation}
\max_{0\leq i\leq N}
\mathbb E
\left[
\left|
\widetilde u_i^\pi(\widetilde X_i^\pi)
-
u_i^0(\widetilde X_i^\pi)
\right|^2
\right]
\leq
C(1+|x_0|^2)h.
\end{equation}
Similarly, for the $v$-component we have
\begin{align}
h
\sum_{i=0}^{N-1}
\mathbb E
\left|
\widetilde v_i^\pi(\widetilde X_i^\pi)
-
v_i^0(\widetilde X_i^\pi)
\right|^2
\leq
2h
\sum_{i=0}^{N-1}
\mathbb E
\left|
\widetilde Z_i^\pi-Z_{t_i}
\right|^2
+
2h
\sum_{i=0}^{N-1}
\mathbb E
\left|
v_i^0(X_{t_i})
-
v_i^0(\widetilde X_i^\pi)
\right|^2.
\end{align}
We estimate the two terms on the right-hand side separately. For the first term, using \eqref{p2-eq:implicit_process_consistency} and \eqref{p2-eq:implicit_XYZ_time_regular} gives
\begin{align*}
h
\sum_{i=0}^{N-1}
\mathbb E
\left|
\widetilde Z_i^\pi-Z_{t_i}
\right|^2
&\leq
2
\sum_{i=0}^{N-1}
\mathbb E
\left[
\int_{t_i}^{t_{i+1}}
|Z_t-\widetilde Z_i^\pi|^2
\,\mathrm dt
\right]
+
2
\sum_{i=0}^{N-1}
\mathbb E
\left[
\int_{t_i}^{t_{i+1}}
|Z_t-Z_{t_i}|^2
\,\mathrm dt
\right]
\\
&\leq
C(1+|x_0|^2)h,
\end{align*}
and for the second term, the spatial Lipschitz continuity of $v$ and \eqref{p2-eq:implicit_process_consistency} give
\begin{align*}
h
\sum_{i=0}^{N-1}
\mathbb E
\left|
v_i^0(X_{t_i})
-
v_i^0(\widetilde X_i^\pi)
\right|^2
\leq
Ch
\sum_{i=0}^{N-1}
\mathbb E
|X_{t_i}-\widetilde X_i^\pi|^2
\leq
C(1+|x_0|^2)h,
\end{align*}
Combining these results gives the desired estimate \eqref{p2-eq:implicit_consistency_residual}.

\end{proof}

\begin{theorem}[Time-discretization error]
\label{p2-thm:implicit_time_discretization_error}
Suppose that the assumptions of Theorem~\ref{p2-thm:implicit_markovian_convergence} and Lemma~\ref{p2-lem:implicit_one_iteration_consistency} hold. Let $(X^\pi,Y^\pi,\widehat Z^\pi)$ be the discrete fixed-point solution. There exists a constant $C>0$, independent of the partition and $x_0$, such that
\begin{equation}\label{p2-eq:implicit_time_error}
\begin{aligned}
\max_{0\leq i<N}\mathbb E\left[
\sup_{t\in[t_i,t_{i+1}]}
\bigl(
|X_t-X_i^\pi|^2+|Y_t-Y_i^\pi|^2
\bigr)
\right]
+
\sum_{i=0}^{N-1}\mathbb E
\int_{t_i}^{t_{i+1}}
|Z_t-\widehat Z_i^\pi|^2\,\mathrm dt
\leq C(1+|x_0|^2)h.
\end{aligned}
\end{equation}
\end{theorem}

\begin{proof}
Let $(\widetilde u^\pi,\widetilde v^\pi)=F^\pi(u^0,v^0)$ be the approximation obtained after one iteration with the exact grid fields $(u^0,v^0)$ as inputs, and denote the corresponding processes by $(\widetilde X^\pi,\widetilde Y^\pi,\widetilde Z^\pi)$. We compare this approximation with the discrete fixed-point solution, using the same initial state and Brownian increments.

Define
\begin{equation}
\Delta X_i
\coloneqq X_i^\pi-\widetilde X_i^\pi,
\qquad
\Delta Y_i
\coloneqq Y_i^\pi-\widetilde Y_i^\pi,
\qquad
\Delta\widehat Z_i
\coloneqq \widehat Z_i^\pi-\widetilde Z_i^\pi.
\end{equation}
Since both forward processes start from $x_0$, we have $\Delta X_0=0$.

By Theorem~\ref{p2-thm:implicit_markovian_convergence}, the fixed-point fields satisfy $L(u^\pi)\leq\bar L$ and $L(v^\pi)\leq\bar V$. Consequently, we can show that 
\begin{equation}
\begin{aligned}
& \left|
u_i^\pi(X_i^\pi)-u_i^0(\widetilde X_i^\pi)
\right|^2
\leq
2\left|
u_i^\pi(X_i^\pi)-u_i^\pi(\widetilde X_i^\pi)
\right|^2
+
2\left|
(u_i^\pi-u_i^0)(\widetilde X_i^\pi)
\right|^2
\leq
2\bar L|\Delta X_i|^2
+
2\left|
(u_i^\pi-u_i^0)(\widetilde X_i^\pi)
\right|^2,    \\
&
\left|
v_i^\pi(X_i^\pi)-v_i^0(\widetilde X_i^\pi)
\right|^2
\leq
2\bar V|\Delta X_i|^2
+
2\left|
(v_i^\pi-v_i^0)(\widetilde X_i^\pi)
\right|^2.
\end{aligned}
\end{equation}
Since both discrete forward equations use the same coefficient $\sigma_i$ and the same Brownian increment, their diffusion terms cancel when the equations are subtracted. Applying the same forward difference estimate as in the proof of Theorem~\ref{p2-thm:implicit_markovian_convergence} therefore gives
\begin{equation}\label{p2-eq:implicit_time_forward_stability}
\begin{aligned}
\max_{0\leq i\leq N}\mathbb E|\Delta X_i|^2
&\leq
\beta_u
\max_{0\leq i\leq N}
\mathbb E\left|
(u_i^\pi-u_i^0)(\widetilde X_i^\pi)
\right|^2
+
\beta_v h\sum_{i=0}^{N-1}
\mathbb E\left|
(v_i^\pi-v_i^0)(\widetilde X_i^\pi)
\right|^2.
\end{aligned}
\end{equation}
The backward difference estimates apply to these two solutions as well, yielding
\begin{equation}
\label{p2-eq:implicit_time_backward_Y_Z}
\max_{0\leq i\leq N}\mathbb E|\Delta Y_i|^2
\leq
C_Y\max_{0\leq i\leq N}\mathbb E|\Delta X_i|^2
,\qquad
h\sum_{i=0}^{N-1}\mathbb E|\Delta\widehat Z_i|^2
\leq
C_Z\max_{0\leq i\leq N}\mathbb E|\Delta X_i|^2.
\end{equation}
We now estimate the two field-difference terms on the right-hand side of \eqref{p2-eq:implicit_time_forward_stability}. Since $Y_i^\pi=u_i^\pi(X_i^\pi)$ and $\widetilde Y_i^\pi = \widetilde u_i^\pi(\widetilde X_i^\pi)$, we have
\begin{equation}
\begin{aligned}
(u_i^\pi-u_i^0)(\widetilde X_i^\pi)
=
u_i^\pi(\widetilde X_i^\pi)-u_i^\pi(X_i^\pi)
+ \Delta Y_i
+ (\widetilde u_i^\pi-u_i^0)(\widetilde X_i^\pi).
\end{aligned}
\end{equation}
For any $\varepsilon>0$, Young's inequality and the Lipschitz bound on $u_i^\pi$ give
\begin{equation}
\begin{aligned}
\left|
(u_i^\pi-u_i^0)(\widetilde X_i^\pi)
\right|^2
&\leq
2(1+\varepsilon)
\left[
\bar L|\Delta X_i|^2+|\Delta Y_i|^2
\right]
+
\left(1+\frac1\varepsilon\right)
\left|
(\widetilde u_i^\pi-u_i^0)(\widetilde X_i^\pi)
\right|^2.
\end{aligned}
\end{equation}
Taking expectations and maxima, and using \eqref{p2-eq:implicit_time_backward_Y_Z}, we obtain
\begin{equation}\label{p2-eq:implicit_time_u_difference}
\begin{aligned}
\max_{0\leq i\leq N}
\mathbb E\left|
(u_i^\pi-u_i^0)(\widetilde X_i^\pi)
\right|^2
\leq
2(1+\varepsilon)(C_Y+\bar L)
\max_{0\leq i\leq N}\mathbb E|\Delta X_i|^2
+
\left(1+\frac1\varepsilon\right)
\max_{0\leq i\leq N}
\mathbb E\left|
(\widetilde u_i^\pi-u_i^0)(\widetilde X_i^\pi)
\right|^2.
\end{aligned}
\end{equation}
Similarly, using $\widehat Z_i^\pi=v_i^\pi(X_i^\pi)$ and $\widetilde Z_i^\pi = \widetilde v_i^\pi(\widetilde X_i^\pi)$, we obtain
\begin{equation}
\begin{aligned}
\left|
(v_i^\pi-v_i^0)(\widetilde X_i^\pi)
\right|^2
\leq
2(1+\varepsilon)
\left[
\bar V|\Delta X_i|^2+|\Delta\widehat Z_i|^2
\right]
+
\left(1+\frac1\varepsilon\right)
\left|
(\widetilde v_i^\pi-v_i^0)(\widetilde X_i^\pi)
\right|^2.
\end{aligned}
\end{equation}
Taking expectations, multiplying by $h$, and summing over $i=0,\ldots,N-1$, we can derive
\begin{equation}\label{p2-eq:implicit_time_v_difference}
\begin{aligned}
h\sum_{i=0}^{N-1}
\mathbb E\left|
(v_i^\pi-v_i^0)(\widetilde X_i^\pi)
\right|^2
\leq
2(1+\varepsilon)(C_Z+\bar VT)
\max_{0\leq i\leq N}\mathbb E|\Delta X_i|^2
+
\left(1+\frac1\varepsilon\right)
h\sum_{i=0}^{N-1}
\mathbb E\left|
(\widetilde v_i^\pi-v_i^0)(\widetilde X_i^\pi)
\right|^2.
\end{aligned}
\end{equation}
Substituting \eqref{p2-eq:implicit_time_u_difference} and \eqref{p2-eq:implicit_time_v_difference} into \eqref{p2-eq:implicit_time_forward_stability} gives
\begin{equation}
\label{p2-eq:implicit_time_absorption}
\begin{aligned}
\max_{0\leq i\leq N}\mathbb E|\Delta X_i|^2
&\leq
2(1+\varepsilon)
\left[
\beta_u(C_Y+\bar L)+\beta_v(C_Z+\bar VT)
\right]
\max_{0\leq i\leq N}\mathbb E|\Delta X_i|^2
\\
&\quad+
\left(1+\frac1\varepsilon\right)\beta_u
\max_{0\leq i\leq N}
\mathbb E\left|
(\widetilde u_i^\pi-u_i^0)(\widetilde X_i^\pi)
\right|^2
\\
&\quad+
\left(1+\frac1\varepsilon\right)\beta_v h
\sum_{i=0}^{N-1}
\mathbb E\left|
(\widetilde v_i^\pi-v_i^0)(\widetilde X_i^\pi)
\right|^2.
\end{aligned}
\end{equation}
By the definitions of $C_U$, $C_V$, and $\rho_\eta$, the coefficient satisfies
\begin{equation}
2\beta_u(C_Y+\bar L)+2\beta_v(C_Z+\bar VT)
=
C_U\beta_u+C_V\beta_v
\leq
(C_U+\eta C_V)
\max\left\{\beta_u,\frac{\beta_v}{\eta}\right\}
=\rho_\eta.
\end{equation}
Since $\rho_\eta<1$, we may choose $\varepsilon>0$, independently of the partition, such that $(1+\varepsilon)\rho_\eta<1$. Moving the corresponding term in \eqref{p2-eq:implicit_time_absorption} to the left-hand side yields
\begin{equation}
\begin{aligned}
\left[1-(1+\varepsilon)\rho_\eta\right]
\max_{0\leq i\leq N}\mathbb E|\Delta X_i|^2
\leq
\left(1+\frac1\varepsilon\right)
\Bigg[
\beta_u
\max_{0\leq i\leq N}
\mathbb E\left|
(\widetilde u_i^\pi-u_i^0)(\widetilde X_i^\pi)
\right|^2
+
\beta_v h\sum_{i=0}^{N-1}
\mathbb E\left|
(\widetilde v_i^\pi-v_i^0)(\widetilde X_i^\pi)
\right|^2
\Bigg].
\end{aligned}
\end{equation}
By the consistency estimate \eqref{p2-eq:implicit_consistency_residual} in Lemma~\ref{p2-lem:implicit_one_iteration_consistency}, the expression in square brackets on the right-hand side is bounded by $C(1+|x_0|^2)h$. Since $\varepsilon$, $\rho_\eta$, $\beta_u$, and $\beta_v$ are independent of the partition, we conclude that
\begin{equation}
\label{p2-eq:implicit_time_X_difference}
\max_{0\leq i\leq N}
\mathbb E|X_i^\pi-\widetilde X_i^\pi|^2
\leq C(1+|x_0|^2)h.
\end{equation}
Combining this estimate with \eqref{p2-eq:implicit_time_backward_Y_Z} gives
\begin{equation}\label{p2-eq:implicit_fixed_vs_one_iteration}
\begin{aligned}
\max_{0\leq i\leq N}
\mathbb E|X_i^\pi-\widetilde X_i^\pi|^2
+
\max_{0\leq i\leq N}
\mathbb E|Y_i^\pi-\widetilde Y_i^\pi|^2
+
h\sum_{i=0}^{N-1}
\mathbb E|\widehat Z_i^\pi-\widetilde Z_i^\pi|^2
\leq C(1+|x_0|^2)h.
\end{aligned}
\end{equation}
Finally, split the errors through $(\widetilde X^\pi,\widetilde Y^\pi,\widetilde Z^\pi)$ and apply $|a+b|^2\leq2|a|^2+2|b|^2$. The differences between the discrete approximations are constant in $t$ on each interval $[t_i,t_{i+1}]$. Combining the process consistency estimate \eqref{p2-eq:implicit_process_consistency} with \eqref{p2-eq:implicit_fixed_vs_one_iteration} therefore proves \eqref{p2-eq:implicit_time_error}.

\end{proof}

\begin{corollary}[Total discretization and iteration error]
\label{p2-cor:implicit_total_error}
Under the assumptions of Theorem~\ref{p2-thm:implicit_time_discretization_error}, let $(X^{\pi,M},Y^{\pi,M},\widehat Z^{\pi,M})$ be the approximation after $M\geq1$ iterations. There exists a constant $C>0$, independent of $h$, $M$, and $x_0$, such that
\begin{equation}
\label{p2-eq:implicit_total_error}
\begin{aligned}
&\max_{0\leq i<N}\mathbb E\left[
\sup_{t\in[t_i,t_{i+1}]}
\bigl(
|X_t-X_i^{\pi,M}|^2
+
|Y_t-Y_i^{\pi,M}|^2
\bigr)
\right]
\\
&\qquad+
\sum_{i=0}^{N-1}\mathbb E
\int_{t_i}^{t_{i+1}}
|Z_t-\widehat Z_i^{\pi,M}|^2\,\mathrm dt
\leq
C(1+|x_0|^2)
\bigl(h+\rho_\eta^{M-1}\bigr).
\end{aligned}
\end{equation}
\end{corollary}

\begin{proof}
For the initial pair $(0,x_0)$, the forward and backward difference estimates established in the proof of Theorem~\ref{p2-thm:implicit_markovian_convergence} give, for every $m\geq1$,
\begin{equation}\label{p2-eq:implicit_process_iteration_geometric}
\begin{aligned}
& \max_{0\leq i\leq N}
\mathbb E|X_i^{\pi,m+1}-X_i^{\pi,m}|^2
+
\max_{0\leq i\leq N}
\mathbb E|Y_i^{\pi,m+1}-Y_i^{\pi,m}|^2
+
h\sum_{i=0}^{N-1}
\mathbb E|
\widehat Z_i^{\pi,m+1}-\widehat Z_i^{\pi,m}|^2  \\
\leq &
C(1+|x_0|^2)\mathcal E_m 
\leq 
C(1+|x_0|^2)\rho_\eta^{m-1},
\end{aligned}
\end{equation}
where we have used the partition-independent bound on $\mathcal E_1$.

We first consider the $X$-component. Fix $M\geq1$ and let $K>M$. Using the telescoping identity, Minkowski's inequality, and \eqref{p2-eq:implicit_process_iteration_geometric}, we obtain
\begin{equation}
\begin{aligned}
\left(
\max_{0\leq i\leq N}
\mathbb E|X_i^{\pi,K}-X_i^{\pi,M}|^2
\right)^{1/2}
&=
\left(
\max_{0\leq i\leq N}
\mathbb E\left|
\sum_{m=M}^{K-1}
\left(X_i^{\pi,m+1}-X_i^{\pi,m}\right)
\right|^2
\right)^{1/2}
\\
&\leq
\sum_{m=M}^{K-1}
\left(
\max_{0\leq i\leq N}
\mathbb E|X_i^{\pi,m+1}-X_i^{\pi,m}|^2
\right)^{1/2}
\\
&\leq
C(1+|x_0|^2)^{1/2}
\sum_{m=M}^{K-1}\rho_\eta^{(m-1)/2}.
\end{aligned}
\end{equation}
Letting $K\to\infty$ and using the convergence of $X^{\pi,K}$ to $X^\pi$ in this discrete $L^2$-norm, we obtain
\begin{align*}
\left(
\max_{0\leq i\leq N}
\mathbb E|X_i^\pi-X_i^{\pi,M}|^2
\right)^{1/2}
\leq
C(1+|x_0|^2)^{1/2}
\sum_{m=M}^{\infty}\rho_\eta^{(m-1)/2}
=
C(1+|x_0|^2)^{1/2}
\frac{\rho_\eta^{(M-1)/2}}{1-\sqrt{\rho_\eta}}.
\end{align*}
Squaring both sides yields
\begin{equation}
\max_{0\leq i\leq N}
\mathbb E|X_i^{\pi,M}-X_i^\pi|^2
\leq
C(1+|x_0|^2)\rho_\eta^{M-1},
\end{equation}
where the factor $(1-\sqrt{\rho_\eta})^{-2}$ has been absorbed into $C$.

The same argument applies to the $Y$-component and to the $Z$-component, and for the latter we use the discrete $L^2$-norm $ \bigl( h\sum_{i=0}^{N-1} \mathbb E| \widehat Z_i^{\pi,K}-\widehat Z_i^{\pi,M}|^2 \bigr)^{1/2} $. Combining these estimates with Theorem~\ref{p2-thm:implicit_time_discretization_error} and using the inequality $|a+b|^2\leq2|a|^2+2|b|^2$, we obtain \eqref{p2-eq:implicit_total_error}.

\end{proof}

\section{Numerical experiments}\label{p2-sec4}

We investigate the numerical performance and convergence behavior of the Markovian iteration scheme. We first describe three numerical implementations and discuss how Monte Carlo sampling error may affect their performance, motivating our preference for the derivative-based and merged implementations. We then introduce the experimental setup and present the numerical examples.

\subsection{Three numerical implementations}

We consider three implementations, referred to as \emph{direct}, \emph{derivative-based}, and \emph{merged}. They share the same initialization and forward recursion in the scheme \eqref{p2-eq:implicit_full_scheme} but differ in how the backward step is implemented. Under suitable conditions, they yield the same discrete processes at the population level, meaning that expectations are evaluated exactly and minimization is performed over all square-integrable functions of the current state.

\medskip
\noindent\textbf{Direct implementation.}\enspace
The direct implementation uses Monte Carlo regression to approximate the conditional expectations in \eqref{p2-eq:implicit_full_scheme}. In particular, the $Z$-field is approximated by regressing the target $h^{-1}Y_{i+1}^{\pi,m}\Delta W_i$ onto functions of $X_i^{\pi,m}$, according to
\begin{equation}\label{p2-eq:directextension_z}
\widehat Z_i^{\pi,m}
=
\frac1h\mathbb E_{t_i}\left[
Y_{i+1}^{\pi,m}\Delta W_i
\right].
\end{equation}
Once the $Z$-field has been approximated, the $Y$-field is obtained by regression using the implicit backward equation, with the same parametrization of $Y_i^{\pi,m}$ used both inside and outside the driver.

\medskip
\noindent\textbf{Derivative-based implementation.}\enspace
The derivative-based implementation also approximates the $Z$-field by Monte Carlo regression, but uses the Gaussian integration-by-parts representation \eqref{p2-eq:v_gaussian_classcial}, 
\begin{equation}
\label{p2-eq:derivative_based_z}
\widehat Z_i^{\pi,m}
=
\mathbb E_{t_i}\left[
\sigma_i\,\partial_x u_{i+1}^{\pi,m}
\bigl(X_{i+1}^{\pi,m}\bigr)
\right].
\end{equation}
This classical representation holds under additional spatial regularity, for example if $u_{i+1}^{\pi,m}\in C^1(\mathbb R)$ with bounded derivative. The regression target is then $\sigma_i\,\partial_x u_{i+1}^{\pi,m}(X_{i+1}^{\pi,m})$, which replaces the Brownian-weighted target in \eqref{p2-eq:directextension_z}. Whenever the integration-by-parts identity applies, the two targets have the same conditional expectation. Since the $Y$-field is obtained using the same implicit regression formulation, the derivative-based and direct implementations are equivalent at the population level.

\medskip
\noindent\textbf{Merged implementation.}\enspace
Even when the required regularity holds, an accurate approximation of $u_{i+1}^{\pi,m}$ does not necessarily yield an accurate approximation of its derivative. We therefore also consider a merged implementation, which avoids differentiating the approximated $Y$-field and approximates both backward fields through a joint regression. Its population formulation is
\begin{equation}
\label{p2-eq:merged_formulation}
\begin{aligned}
(y_i^\star,z_i^\star)
&=
\underset{y_i,z_i\in L^2(\mathcal F_{t_i})}{\arg\min}
\mathbb E\left[
\left|
Y_{i+1}^{\pi,m}
-
\left(
y_i-hf(t_i,X_i^{\pi,m},y_i,z_i)
+z_i\Delta W_i
\right)
\right|^2
\right].
\end{aligned}
\end{equation}
Under the standing step-size condition, the unique minimizer satisfies $y_i^\star=Y_i^{\pi,m}$ and $z_i^\star=\widehat Z_i^{\pi,m}$, where $Y_i^{\pi,m}$ and $\widehat Z_i^{\pi,m}$ are the values obtained from the direct implementation. For a justification and application of this equivalence, see for example, \cite{gobet2005regression, hure2020deep}. We remark that this equivalence only holds when the driver $f$ depends on $Y_i^{\pi,m}$ in the backward scheme, instead of $Y_{i+1}^{\pi,m}$. In practice, $y_i$ and $z_i$ are parametrized as functions of $X_i^{\pi,m}$, and their coefficients are optimized jointly.

Thus, whenever the derivative representation is valid, the three implementations yield the same backward values at the population level. Since they also share the same initialization and forward recursion, they produce the same discrete processes $(X_i^{\pi,m},Y_i^{\pi,m},\widehat Z_i^{\pi,m})$ at every Markovian iteration.

\medskip
\noindent\textbf{Monte Carlo sampling error.}\enspace
Despite their equivalence at the population level, the three implementations may differ numerically because of restricted parametrizations, finite samples, and optimization errors. At each Markovian iteration, the input fields are fixed and the backward recursion uses an implicit Euler scheme. Numerical difficulties associated with implementing this backward step therefore also arise within the iteration. To illustrate Monte Carlo sampling error as one possible source of these differences, we compare the direct and derivative-based estimators of $v_i^{\pi,m}(x)$.

Fix a partition $\pi$, a time index $i$, an iteration $m$, and a starting state $x$, and hold field $u_{i+1}^{\pi,m}$ fixed. We treat this field and, where required, its derivative as known functions, so that only the sampling error in approximating the conditional expectations is considered. Let $N_{\mathrm{mc}}$ denote the Monte Carlo sample size, and let $X_{i+1}^{\pi,m,(j)}$, $j=1,\ldots,N_{\mathrm{mc}}$, be independent one-step transitions from $x$, driven by independent increments $\Delta W_i^{(j)}\sim\mathcal N(0,h)$. Setting $Y_{i+1}^{\pi,m,(j)} =u_{i+1}^{\pi,m}(X_{i+1}^{\pi,m,(j)})$, we obtain the following Monte Carlo estimators:
\begin{equation}
\begin{aligned}
\widehat Z_{i,N_{\mathrm{mc}}}^{\mathrm{dir}}
&=
\frac{1}{N_{\mathrm{mc}}}
\sum_{j=1}^{N_{\mathrm{mc}}}
\frac{Y_{i+1}^{\pi,m,(j)}\Delta W_i^{(j)}}{h}
,\qquad
\widehat Z_{i,N_{\mathrm{mc}}}^{\mathrm{der}}
&=
\frac{\sigma_i}{N_{\mathrm{mc}}}
\sum_{j=1}^{N_{\mathrm{mc}}}
\partial_xu_{i+1}^{\pi,m}
\left(X_{i+1}^{\pi,m,(j)}\right).
\end{aligned}
\end{equation}
For the direct estimator, we have
\begin{equation}
\mathbb E\left|
\widehat Z_{i,N_{\mathrm{mc}}}^{\mathrm{dir}}
-v_i^{\pi,m}(x)
\right|^2
=
\frac{1}{N_{\mathrm{mc}}}
\operatorname{Var}\left(
\frac{Y_{i+1}^{\pi,m,(1)}\Delta W_i^{(1)}}{h}
\right)
=
\mathcal O \!\left(\frac{1}{N_{\mathrm{mc}}h}\right).
\end{equation}
The variance identity follows from independence, while the order estimate follows from the Gaussian moments and the uniform Lipschitz and growth bounds established earlier under our standing assumptions. For related discussions of variance amplification and statistical error estimates in more general settings, see \cite{alanko2013reducing,Lionnet2015,gobet2008numerical}. The order $(N_{\mathrm{mc}}h)^{-1}$ is attained when $u_{i+1}^{\pi,m}$ is a nonzero constant. Thus, for a fixed sample size, the sampling error may increase as the time step decreases. In the full numerical scheme, these errors may also propagate across time steps and successive Markovian iterations.

For the derivative-based estimator, whenever the classical integration-by-parts representation applies, the uniform Lipschitz bound gives $|\partial_xu_{i+1}^{\pi,m}(y)|^2 \leq \bar L$ for every $y\in\mathbb R$. Independence and the boundedness of $\sigma$ therefore yield
\begin{equation}
\mathbb E\left|
\widehat Z_{i,N_{\mathrm{mc}}}^{\mathrm{der}}
-v_i^{\pi,m}(x)
\right|^2
=
\frac{\sigma_i^2}{N_{\mathrm{mc}}}
\operatorname{Var}\left(
\partial_xu_{i+1}^{\pi,m}
\left(X_{i+1}^{\pi,m,(1)}\right)
\right)
\leq
\frac{\overline\sigma^2\bar L}{N_{\mathrm{mc}}}.
\end{equation}
Thus, the mean-square sampling error is $\mathcal O (N_{\mathrm{mc}}^{-1})$, uniformly in the time-step size, time index, iteration index, and starting state $x$. This estimate avoids the explicit $h^{-1}$ amplification associated with the direct estimator, although practical performance also depends on the accuracy of the derivative approximation.

The merged implementation avoids a separate regression with the Brownian-weighted target and requires no derivative approximation. In particular, avoiding this target is a feature shared by the derivative-based implementation, which may help explain their comparable numerical performance shown in the subsequent subsections. For statistical error analyses of methods based on joint regression objectives of the form \eqref{p2-eq:merged_formulation}, under different parametrizations and assumptions, see \cite{gobet2005regression,OUYANG2025319}.

This local comparison of Monte Carlo sampling errors offers a possible explanation for the numerical behavior observed in subsequent subsections. A complete analysis accounting for parametrization, optimization, and error propagation across time steps and iterations is beyond the scope of this work.

\subsection{Experimental setup}

We now describe the parametrization, reference solution, and error measures used in the numerical experiments. For all three implementations, we follow \cite{bender2008time} in parametrizing both decoupling fields using quadratic polynomial basis functions. For simplicity, we retain the notation $u_i^{\pi,m}$ and $v_i^{\pi,m}$ for their numerical approximations throughout the numerical experiments. More precisely, for $i=0,\ldots,N-1$, on the computational domain,
\begin{equation}
\label{p2-u_basis_approx}
u_i^{\pi,m}(x)
=
\alpha_0^{i,m}
+
\sum_{j=1}^{d_1}
\alpha_j^{i,m}x_j
+
\sum_{j=1}^{d_1}
\alpha_{d_1+j}^{i,m}x_j^2
+
\sum_{1\leq j<k\leq d_1}
\alpha_{j,k}^{i,m}x_jx_k,
\end{equation}
where $x_j$ denotes the $j$-th component of $x\in\mathbb R^{d_1}$. The coefficients $\alpha_0^{i,m}$, $\alpha_j^{i,m}$, $\alpha_{d_1+j}^{i,m}$, and $\alpha_{j,k}^{i,m}$ are determined by solving the corresponding regression problem. The same class of basis functions is used for $v_i^{\pi,m}$, with separate coefficients for each component, collectively denoted by $\beta^{i,m}$. The terminal field is prescribed by $u_N^{\pi,m}=g$.

The inputs to the basis functions are clipped to a bounded computational domain, and the derivative-based implementation uses the derivatives of the resulting clipped functions. Although our convergence analysis is restricted to one dimension for simplicity, the numerical implementations extend to multidimensional processes by applying the parametrization componentwise to the backward processes $Y$ and $Z$. The resulting regression problems are solved using SciPy's implementation of the BFGS algorithm in Python.

For each test problem with known continuous decoupling fields, we construct a reference solution by first decoupling the FBSDE and then applying the Euler--Maruyama method on a fine uniform grid with $N'=2\times10^4$ time steps. The reference values of $Y$ and $Z$ are obtained by evaluating these fields along the reference forward process. We choose $N$ so that $N'/N$ is an integer and construct each coarse Brownian increment by summing the corresponding fine increments, ensuring that the coarse and reference approximations use the same Brownian paths. The fine-grid reference solution is sampled at the coarse-grid time points and serves as a surrogate for the continuous-time solution.

Let $N_{\mathrm{mc}}$ denote the number of Monte Carlo sample paths, and let the superscript $(j)$ identify the $j$-th sample path. We report the following error measures after $M$ Markovian iterations, so that $m=M$ below.
\begin{equation}
\label{p2-def:error_terms}
\begin{aligned}
\mathrm{Err}(X)
&=
\sup_{0\leq i\leq N}
\frac{1}{N_{\mathrm{mc}}}
\sum_{j=1}^{N_{\mathrm{mc}}}
|X_i^{\pi,m,(j)}-X_i^{(j)}|^2,
&
\mathrm{Err}(Y)
&=
\sup_{0\leq i\leq N}
\frac{1}{N_{\mathrm{mc}}}
\sum_{j=1}^{N_{\mathrm{mc}}}
|Y_i^{\pi,m,(j)}-Y_i^{(j)}|^2,
\\
\mathrm{Err}(Z)
&=
\frac{T}{N N_{\mathrm{mc}}}
\sum_{i=0}^{N-1}
\sum_{j=1}^{N_{\mathrm{mc}}}
|\widehat Z_i^{\pi,m,(j)}-Z_i^{(j)}|^2,
&
\mathrm{Total\ Err}
&=
\mathrm{Err}(X)
+
\mathrm{Err}(Y)
+
\mathrm{Err}(Z).
\end{aligned}
\end{equation}
Here $X_i^{(j)}$, $Y_i^{(j)}$, and $Z_i^{(j)}$ denote the corresponding reference values at time $t_i$, and $|\cdot|$ denotes the Euclidean norm for vectors and the Frobenius norm for matrices. Since $h=T/N$, $\mathrm{Err}(Z)$ approximates the time-integrated squared $L^2$ error used in the theoretical analysis through Monte Carlo sampling and left-endpoint quadrature. Throughout the experiments, we choose $N_{\mathrm{mc}}=15000$ unless otherwise specified.

\subsection{Example 1: A one-dimensional FBSDE}\label{p2-example1}

We begin with a simple one-dimensional FBSDE. The coefficients and terminal condition satisfy all the regularity assumptions used in our analysis: the coefficients are globally Lipschitz in $(x,y,z)$ and uniformly Lipschitz in time, while the terminal function is smooth with bounded derivatives. For $t\in[0,T]$, the FBSDE is given by
 
\begin{equation}\label{p2-eq:exampl1}
\left\{
\begin{aligned}
X_t
&=
x_0+\int_0^t
\bigl(\kappa_yY_s+\kappa_zZ_s\bigr)\,\mathrm ds
+\sigma W_t,
\\
Y_t
&=
\sin(T+X_T)
+\int_t^T f(s,X_s,Y_s,Z_s)\,\mathrm ds
-\int_t^T Z_s\,\mathrm dW_s,
\end{aligned}
\right.
\end{equation}
where the driver is
\begin{equation}\label{p2-eq:example1_driver}
\begin{aligned}
f(t,x,y,z)
&=
-\cos(t+x)
+\frac{\sigma^2}{2}\sin(t+x)
-\kappa_y\sin(t+x)\cos(t+x)
-\kappa_z\sigma\cos^2(t+x)
\\
&\quad
+\gamma_y \bigl(y-\sin(t+x)\bigr)
+\gamma_z \bigl(z-\sigma\cos(t+x)\bigr).
\end{aligned}
\end{equation}
The exact backward solution is given by
\begin{equation}\label{p2-eq:example1_solution}
Y_t=\sin(t+X_t),
\qquad
Z_t=\sigma\cos(t+X_t),
\qquad t\in[0,T].
\end{equation}
We set the parameters $\kappa_y=\kappa_z=\gamma_y=\gamma_z=0.1$, $\sigma=0.5$ and $x_0=1.5$ for this example. For this parameter choice, the invariant-class and contraction conditions of the convergence analysis are satisfied.

We carry out two numerical experiments for this FBSDE. In the first, we fix $T=0.25$ and compare the derivative-based, merged, and direct implementations as we vary the number of time steps $N$ and the number of Markovian iterations $M$. The results are shown in Figure~\ref{p2-fig_ex1_1}.

The left column of Figure~\ref{p2-fig_ex1_1} shows the errors for $N=2,4,8,16,32$, with $M=5$ fixed. Figures~\ref{p2-fig_ex1_1_a} and~\ref{p2-fig_ex1_1_c}, corresponding to the derivative-based and merged implementations, respectively, show very similar convergence patterns: all error components decrease as $N$ increases. In particular, the total error and $\mathrm{Err}(Z)$ decrease from approximately $3\times 10^{-4}$ to $10^{-5}$ by $N=32$. These values suggest a convergence rate of at least first order in $h$ over the tested range, consistent with our theoretical result stated in Corollary \ref{p2-cor:implicit_total_error}. For the direct implementation in Figure~\ref{p2-fig_ex1_1_e}, however, $\mathrm{Err}(Z)$ first decreases and then increases from approximately $10^{-3}$ to $10^{-2}$ as $N$ increases. Since $\mathrm{Err}(Z)$ is much larger than $\mathrm{Err}(X)$ and $\mathrm{Err}(Y)$, the total error follows the same pattern. For the direct implementation, this behavior of $\mathrm{Err}(Z)$, and hence of the total error, may be caused by increasing Monte Carlo sampling error at smaller step sizes, as discussed above. The initial decrease may occur because the reduction in time-discretization error outweighs the increase in sampling error.

The right column of Figure~\ref{p2-fig_ex1_1}  shows the errors for $M=1,2,3,4,5$, with $N=32$ fixed. The derivative-based and merged implementations in Figures~\ref{p2-fig_ex1_1_b} and~\ref{p2-fig_ex1_1_d} again behave similarly: the errors decrease rapidly over the first iterations and then change little, which is consistent with the geometric convergence in iteration steps as stated in Corollary \ref{p2-cor:implicit_total_error}. The derivative-based implementation shows small fluctuations, which may be related to errors in approximating the derivative of the $Y$-field. For the direct implementation in Figure~\ref{p2-fig_ex1_1_f}, $\mathrm{Err}(Z)$ remains close to $5\times10^{-3}$ throughout the iterations. Although the overall trend is similar to that of the other two implementations, $\mathrm{Err}(X)$, $\mathrm{Err}(Z)$, and the total error remain significantly larger. In particular, $\mathrm{Err}(Z)$ is approximately three orders of magnitude larger than in the derivative-based and merged implementations.

In the second experiment, we examine the effect of the time horizon by taking $T=0.05$, 0.15, 0.20, 0.25. For each $T$, we use $N=2,4,8,16,32$ and fix $M=5$. Figure~\ref{p2-fig_ex1_2} shows the total error in the left column and $\mathrm{Err}(Z)$ in the right column.

For the derivative-based and merged implementations, Figures~\ref{p2-fig_ex1_2_a}--\ref{p2-fig_ex1_2_d} show that both errors decrease as $N$ increases, as in the first experiment. In particular, all error curves shift downward as $T$ decreases. This trend is qualitatively consistent with the fact that shorter time horizons lead to more favorable contraction conditions as discussed in Remarks \ref{p2-rem:constant_sigma_invariant_feasibility} and  \ref{p2-rem:implicit_weak_coupling}.

For the direct implementation, Figures~\ref{p2-fig_ex1_2_e} and~\ref{p2-fig_ex1_2_f} again show that increasing $N$ does not consistently reduce the errors. For $T=0.20$ and $T=0.25$, both errors first decrease and then increase, whereas for the shorter horizons they increase throughout the tested range. These patterns may reflect the changing balance between time-discretization and Monte Carlo sampling errors discussed above. In particular, for the longer horizons, time-discretization error may dominate on coarser grids, allowing the errors to decrease before sampling error becomes more significant.

\begin{figure}[htp]
    \centering
    \begin{subfigure}{0.45\textwidth}
        \centering
        \includegraphics[width=\textwidth]{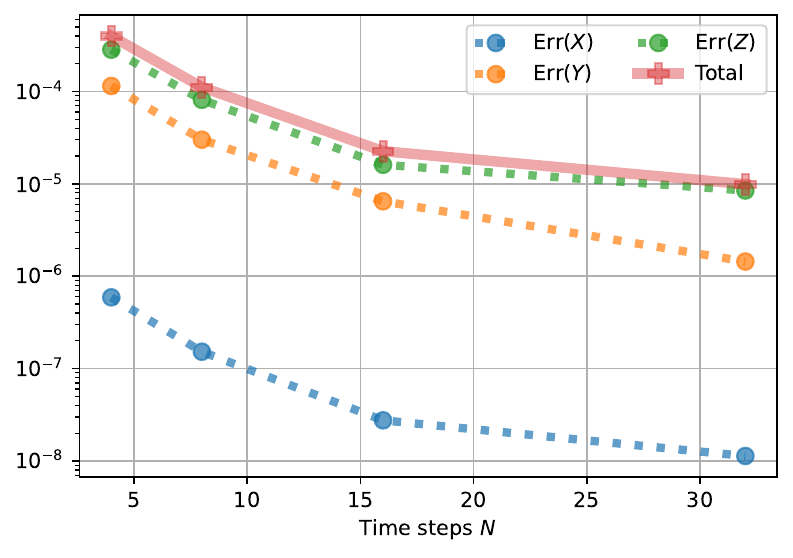}
        \caption{\textit{Deriv.}: Errors vs.\ time steps}
        \label{p2-fig_ex1_1_a}
    \end{subfigure}
    \hspace{0.2cm}
    \begin{subfigure}{0.45\textwidth}
        \centering
        \includegraphics[width=\textwidth]{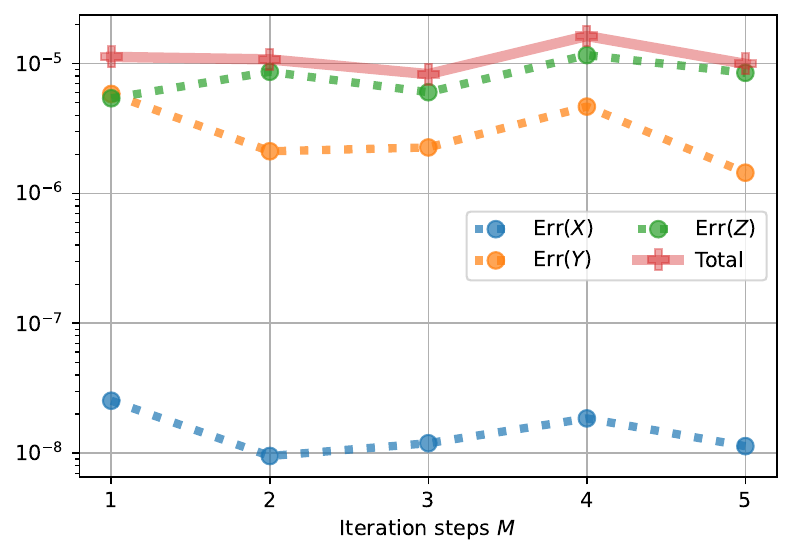}
        \caption{\textit{Deriv.}: Errors vs.\ iteration steps}
        \label{p2-fig_ex1_1_b}
    \end{subfigure}

    \vskip\baselineskip
    \begin{subfigure}{0.45\textwidth}
        \centering
        \includegraphics[width=\textwidth]{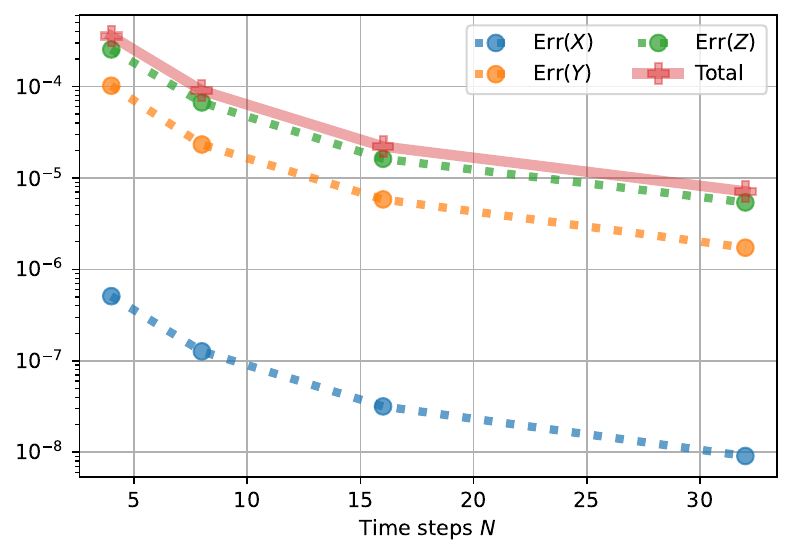}
        \caption{\textit{Merged}: Errors vs.\ time steps}
        \label{p2-fig_ex1_1_c}
    \end{subfigure}
    \hspace{0.2cm}
    \begin{subfigure}{0.45\textwidth}
        \centering
        \includegraphics[width=\textwidth]{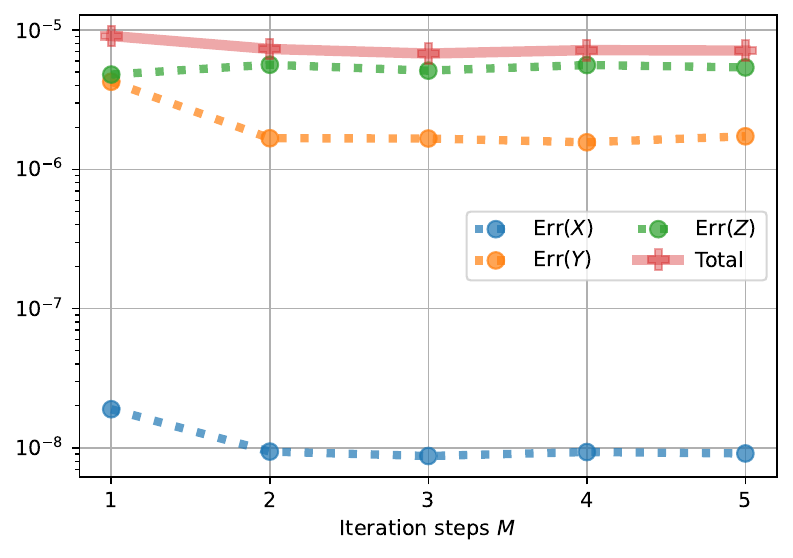}
        \caption{\textit{Merged}: Errors vs.\ iteration steps}
        \label{p2-fig_ex1_1_d}
    \end{subfigure}

    \vskip\baselineskip
    \begin{subfigure}{0.45\textwidth}
        \centering
        \includegraphics[width=\textwidth]{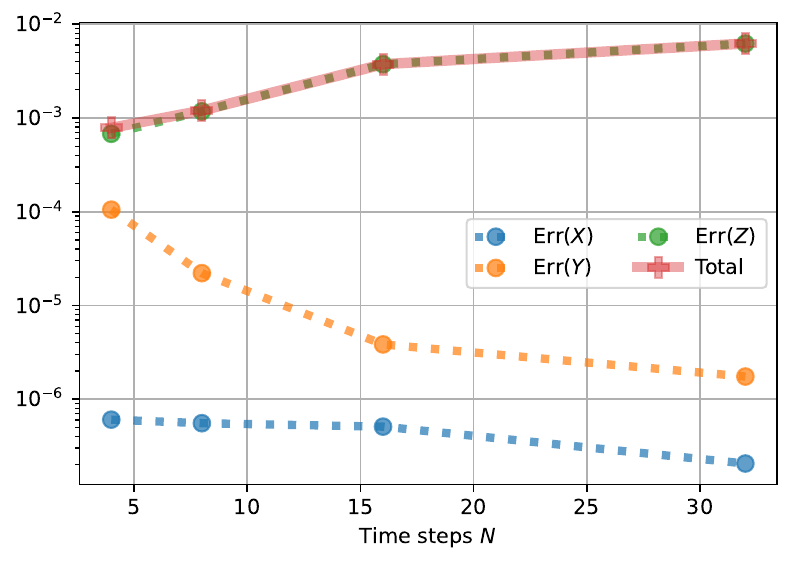}
        \caption{\textit{Direct}: Errors vs.\ time steps}
        \label{p2-fig_ex1_1_e}
    \end{subfigure}
    \hspace{0.2cm}
    \begin{subfigure}{0.45\textwidth}
        \centering
        \includegraphics[width=\textwidth]{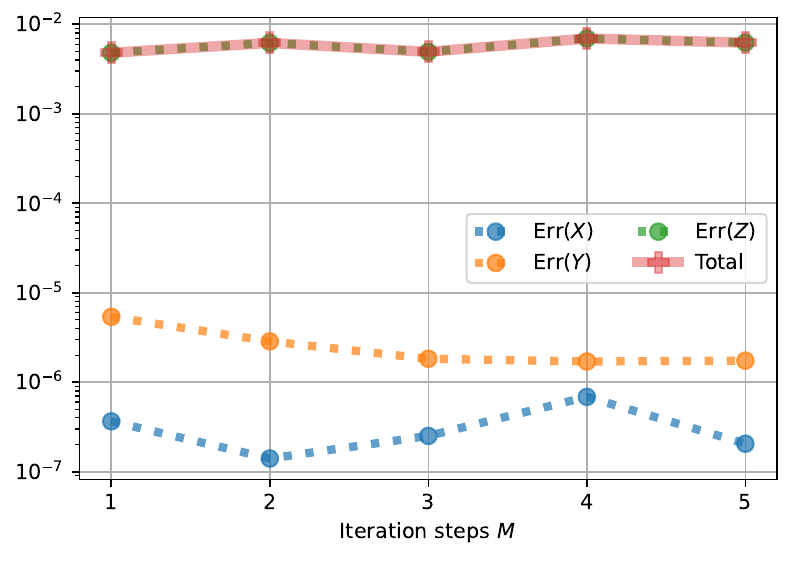}
        \caption{\textit{Direct}: Errors vs.\ iteration steps}
        \label{p2-fig_ex1_1_f}
    \end{subfigure}

    \caption{Numerical results for Example~\hyperref[p2-example1]{1} with $T=0.25$, varying the number of time steps $N$ (left) and the number of Markovian iterations $M$ (right). ``Deriv.'' denotes the derivative-based implementation.}
    \label{p2-fig_ex1_1}
\end{figure}

\begin{figure}[htp]
    \centering
    \begin{subfigure}{0.45\textwidth}
        \centering
        \includegraphics[width=\textwidth]{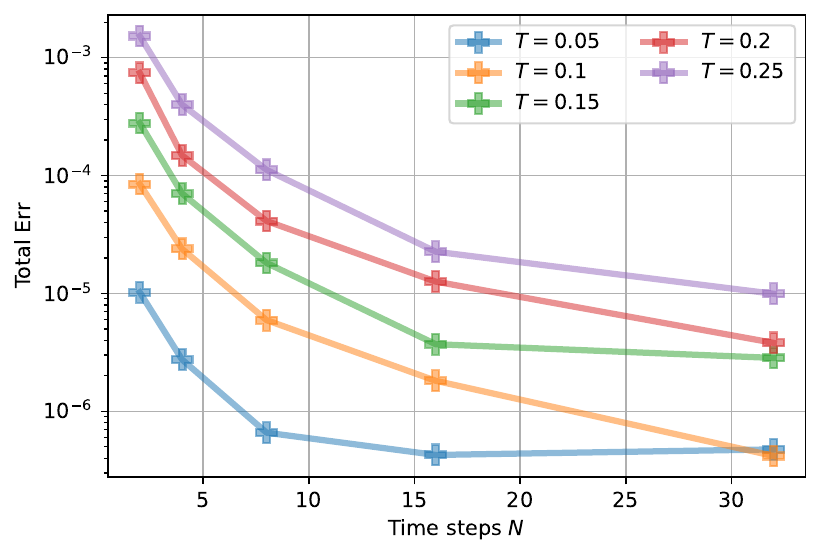}
        \caption{\textit{Deriv.}: $\mathrm{Total\ Err}$ vs.\ time steps}
        \label{p2-fig_ex1_2_a}
    \end{subfigure}
    \hspace{0.2cm}
    \begin{subfigure}{0.45\textwidth}
        \centering
        \includegraphics[width=\textwidth]{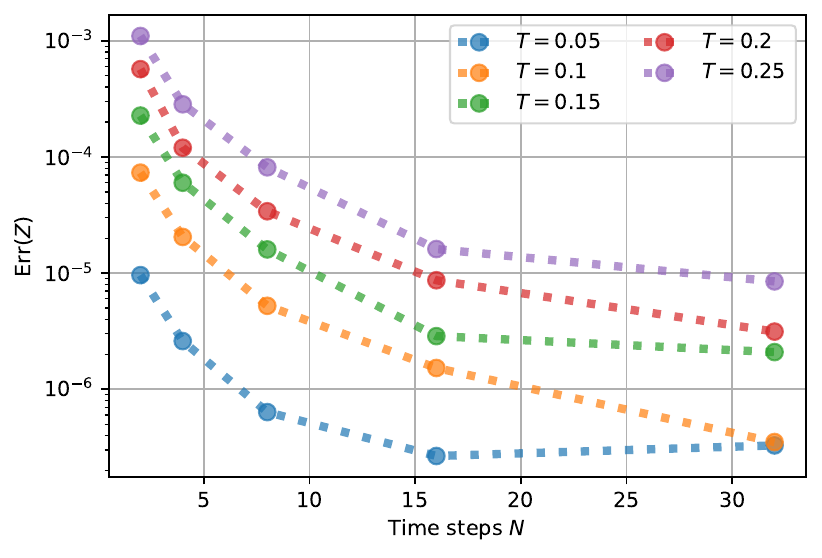}
        \caption{\textit{Deriv.}: $\mathrm{Err}(Z)$ vs.\ time steps}
        \label{p2-fig_ex1_2_b}
    \end{subfigure}

    \vskip\baselineskip
    \begin{subfigure}{0.45\textwidth}
        \centering
        \includegraphics[width=\textwidth]{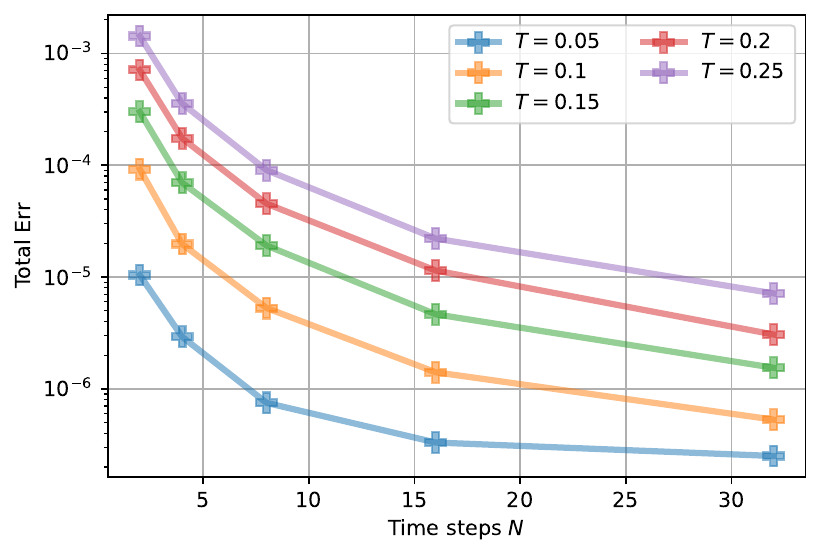}
        \caption{\textit{Merged}: $\mathrm{Total\ Err}$ vs.\ time steps}
        \label{p2-fig_ex1_2_c}
    \end{subfigure}
    \hspace{0.2cm}
    \begin{subfigure}{0.45\textwidth}
        \centering
        \includegraphics[width=\textwidth]{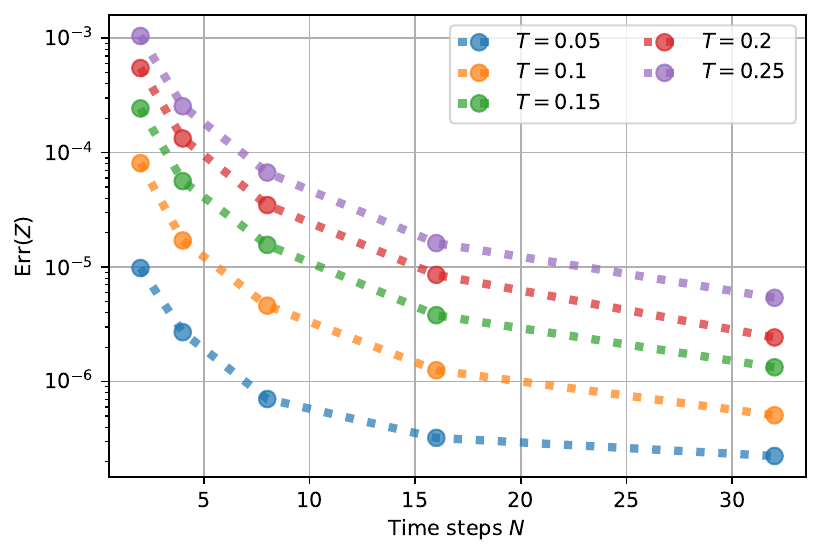}
        \caption{\textit{Merged}: $\mathrm{Err}(Z)$ vs.\ time steps}
        \label{p2-fig_ex1_2_d}
    \end{subfigure}

    \vskip\baselineskip
    \begin{subfigure}{0.45\textwidth}
        \centering
        \includegraphics[width=\textwidth]{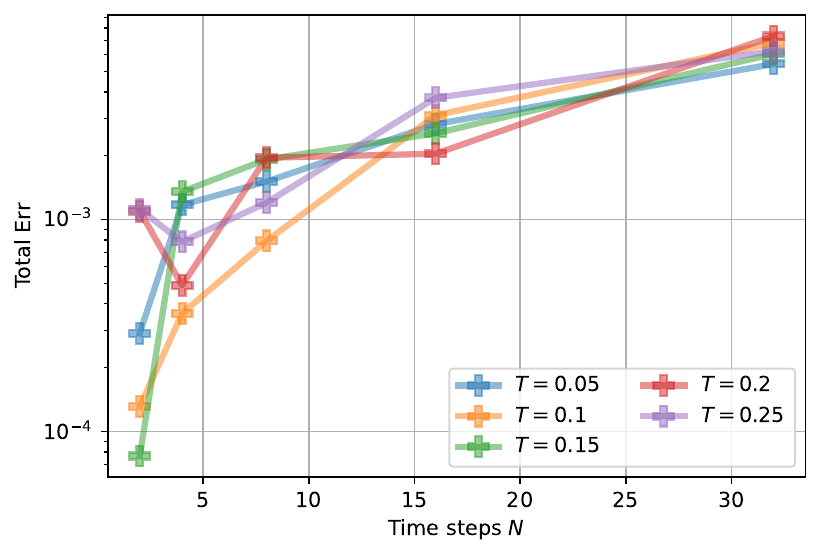}
        \caption{\textit{Direct}: $\mathrm{Total\ Err}$ vs.\ time steps}
        \label{p2-fig_ex1_2_e}
    \end{subfigure}
    \hspace{0.2cm}
    \begin{subfigure}{0.45\textwidth}
        \centering
        \includegraphics[width=\textwidth]{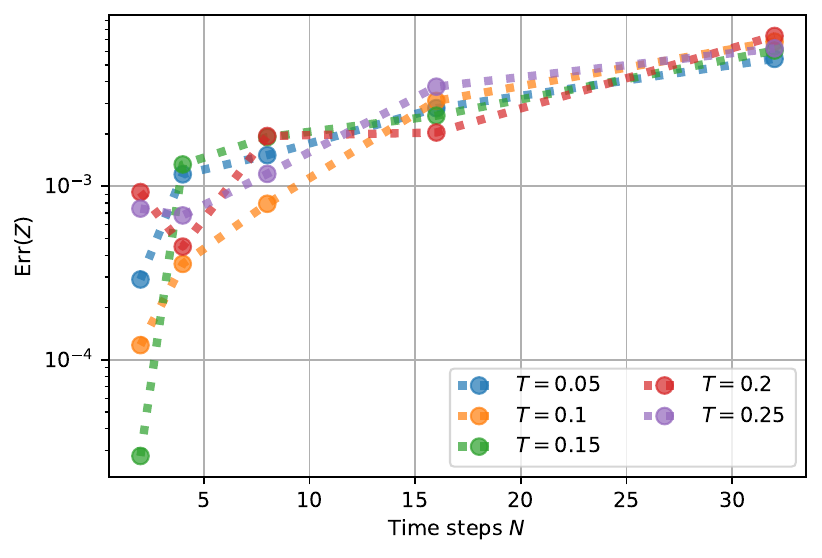}
        \caption{\textit{Direct}: $\mathrm{Err}(Z)$ vs.\ time steps}
        \label{p2-fig_ex1_2_f}
    \end{subfigure}

    \caption{Numerical results for Example~\hyperref[p2-example1]{1} with different time horizons $T$ after $M=5$ iterations. The left and right columns show $\mathrm{Total\ Err}$ and $\mathrm{Err}(Z)$, respectively, as functions of the number of time steps $N$. ``Deriv.'' denotes the derivative-based implementation.}
    \label{p2-fig_ex1_2}
\end{figure}

\subsection{Example 2: A four-dimensional FBSDE}\label{p2-example2}

We next consider a four-dimensional modification of the coupled FBSDE studied in \cite{negyesi2026generalized,negyesi2025numerical}. We replace $y$ in the forward diffusion by the exact decoupling field $u(t,x)$, so that the diffusion depends only on $(t,x)$, while both the drift and driver retain their dependence on $Y$ and $Z$. This example explores the performance of the proposed implementations beyond our one-dimensional convergence analysis with diffusion depending only on time. The coefficients are given by
\begin{equation}\label{p2-ex2:fbsde}
\begin{aligned}
b(t,x,y,z)
& =
\kappa_y y\mathbf{1}_{d_1}
+\kappa_z z^\top,
\quad
g(x)=\sum_{i=1}^{d_1}\sin(x_i),
\quad
\sigma(t,x)
=
e^{-(T-t)}
\left(\sum_{i=1}^{d_1}\sin(x_i)\right)I_{d_1},
\\
f(t,x,y,z)
&=
-y
+\frac12 e^{-3(T-t)}
\left(\sum_{i=1}^{d_1}\sin(x_i)\right)^3
\\
&\quad
-\kappa_y\sum_{i=1}^{d_3}z_i
-\kappa_z e^{-3(T-t)}
\left(\sum_{i=1}^{d_1}\sin(x_i)\right)
\left(\sum_{i=1}^{d_1}\cos^2(x_i)\right).
\end{aligned}
\end{equation}
Here, $\mathbf{1}_{d_1}$ denotes the vector of ones in $\mathbb R^{d_1}$, and $I_{d_1}$ is the $d_1\times d_1$ identity matrix. We take $d_2=1$ and $d_1=d_3=4$. The corresponding exact decoupling fields are
\begin{equation}\label{p2-ex2:fbsde_sol}
\begin{aligned}
u(t,x)
=
e^{-(T-t)}
\sum_{j=1}^{d_1}\sin(x_j)
,\qquad
v^j(t,x)
=
e^{-2(T-t)}
\left(\sum_{k=1}^{d_1}\sin(x_k)\right)\cos(x_j),
\qquad j=1,\ldots,d_3,
\end{aligned}
\end{equation}
The parameters $\kappa_y$ and $\kappa_z$ control the coupling strengths. We choose $\kappa_y=\kappa_z=0.1$, $T=0.25$, and $x_0= 0.25 \pi \mathbf{1}_{d_1}$.

Figure~\ref{p2-fig_ex2} reports the numerical results for the derivative-based and merged implementations with the same choice for $N$ and $M$ as in Figure~\ref{p2-fig_ex1_1}. The direct implementation shows similar error growth to that observed in Example~\hyperref[p2-example1]{1}, so we omit its results here. The left column shows that all error components decrease from approximately $10^{-3}$ to $10^{-5}$ as $N$ increases from $2$ to $32$. The two implementations give very similar errors, with the observed decrease suggesting at least first-order convergence in $h$ for the adopted error measures over the tested range.

The right column shows a rapid initial decrease in all error components, followed by a slower decrease after approximately two iterations. For the derivative-based implementation (Figure~\ref{p2-fig_ex2_b}), the errors fluctuate slightly after $M=2$, whereas the merged implementation (Figure~\ref{p2-fig_ex2_d}) shows smoother curves with a small but continued decrease.

The results in this example are similar to those obtained in the one-dimensional Example~\hyperref[p2-example1]{1}: the derivative-based and merged implementations achieve comparable accuracy and show similar convergence behavior. In particular, the derivative-based implementation performs well despite its potential sensitivity to derivative approximation errors, since accurate approximation of $u_i^{\pi,m}$ alone does not guarantee accurate approximation of its spatial gradient. These observations illustrate the numerical effectiveness of both implementations in this broader setting and motivate further analysis of derivative approximation errors, as well as extensions of the convergence analysis to FBSDEs with more general coupled diffusion coefficients.

\begin{figure}[htp]
    \centering
    \begin{subfigure}{0.45\textwidth}
        \centering
        \includegraphics[width=\textwidth]{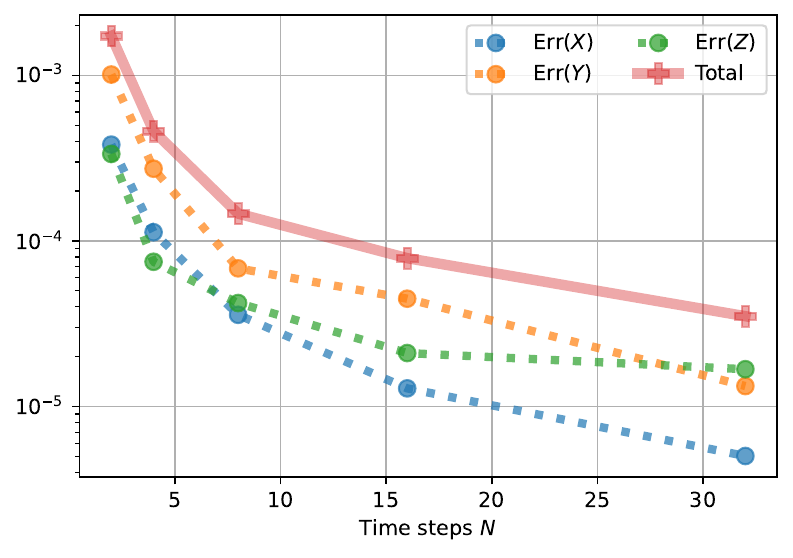}
        \caption{\textit{Deriv.}: Errors vs.\ time steps}
        \label{p2-fig_ex2_a}
    \end{subfigure}
    \hspace{0.2cm}
    \begin{subfigure}{0.45\textwidth}
        \centering
        \includegraphics[width=\textwidth]{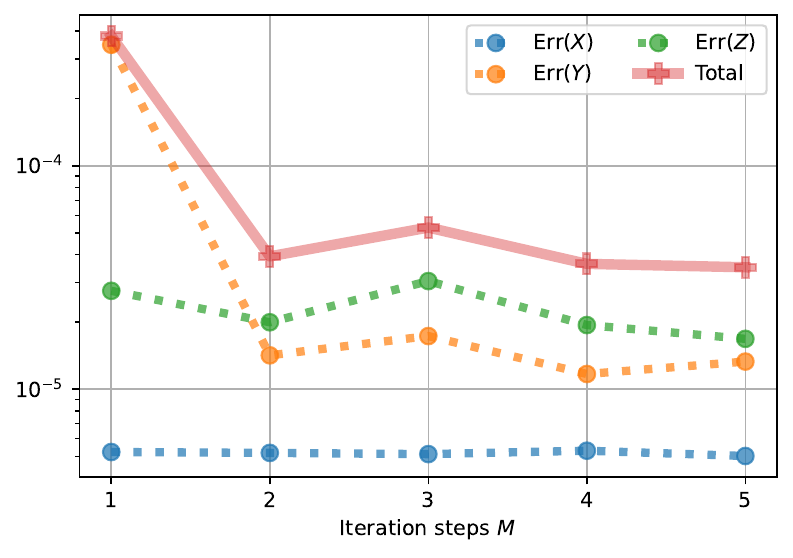}
        \caption{\textit{Deriv.}: Errors vs.\ iteration steps}
        \label{p2-fig_ex2_b}
    \end{subfigure}

    \vskip\baselineskip
    \begin{subfigure}{0.45\textwidth}
        \centering
        \includegraphics[width=\textwidth]{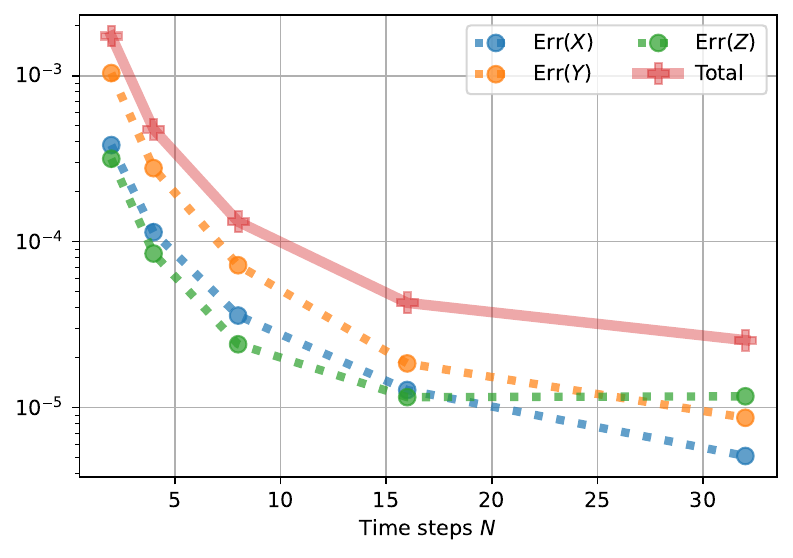}
        \caption{\textit{Merged}: Errors vs.\ time steps}
        \label{p2-fig_ex2_c}
    \end{subfigure}
    \hspace{0.2cm}
    \begin{subfigure}{0.45\textwidth}
        \centering
        \includegraphics[width=\textwidth]{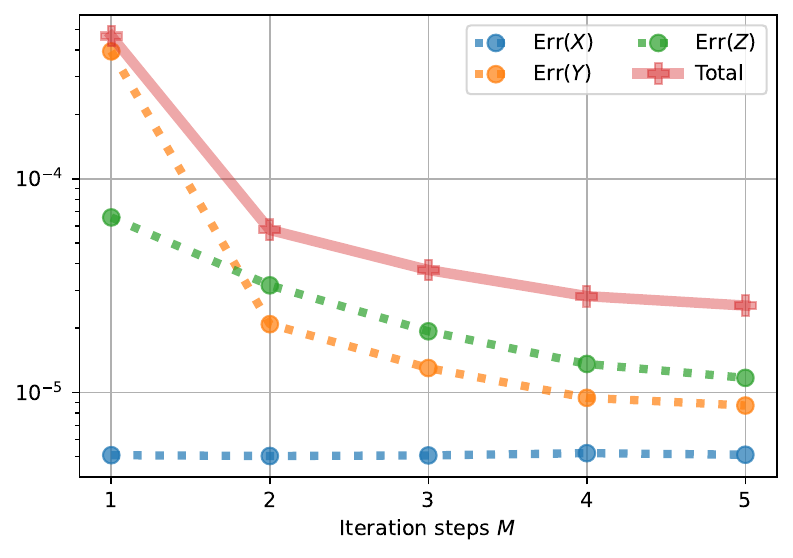}
        \caption{\textit{Merged}: Errors vs.\ iteration steps}
        \label{p2-fig_ex2_d}
    \end{subfigure}

    \caption{Numerical results for Example~\hyperref[p2-example2]{2} with $T=0.25$, varying the number of time steps $N$ (left) and the number of Markovian iterations $M$ (right). ``Deriv.'' denotes the derivative-based implementation.}
    \label{p2-fig_ex2}
\end{figure}
\clearpage
\section{Conclusion}\label{p2-sec:conclusion}

We have established convergence of the Markovian iteration for FBSDEs with a \(Z\)-dependent forward drift and deterministic, time-dependent diffusion. The main ingredient is a Gaussian smoothing representation of the discrete \(Z\)-field, which yields uniform spatial Lipschitz bounds for the discrete decoupling fields without requiring classical second derivatives. Under suitable contraction conditions, this leads to geometric convergence of the Markovian iteration and well-posedness of the discrete FBSDE. Combined with the time-discretization analysis, we obtain an overall error bound accounting for both the iteration and discretization errors.

The numerical experiments support the theoretical convergence results and illustrate the practical performance of the derivative-based and merged implementations. Extending the analysis to multidimensional systems and more general, in particular state-dependent, diffusion coefficients is a natural direction for future work.

\appendix

\section{Proofs of auxiliary results}\label{p2-app:auxiliary_proofs}

\subsection{Implicit backward stability}\label{p2-app:backward_stability}

\begin{proof}[Proof of Lemma~\ref{p2-lem:implicit_backward_stability}]
Let us denote $ \Delta\overline Y_i \coloneqq \mathbb E_{t_i}\left[\Delta Y_{i+1}\right] $. Since $\mathbb E_{t_i}[\Delta W_i]=0$, the definition of $\widehat Z_i^{(\ell)}$ gives
\begin{equation}
\Delta\widehat Z_i
=
\frac1h\mathbb E_{t_i}
\left[\Delta Y_{i+1}\Delta W_i\right]
=
\frac1h\mathbb E_{t_i}
\left[
(\Delta Y_{i+1}-\Delta\overline Y_i)\Delta W_i
\right].
\end{equation}
By the conditional Cauchy--Schwarz inequality and $\mathbb E_{t_i}[|\Delta W_i|^2]=h$,
\begin{equation}
\begin{aligned}
h^2|\Delta\widehat Z_i|^2
\leq
\mathbb E_{t_i}
\left[|\Delta Y_{i+1}-\Delta\overline Y_i|^2\right]
\mathbb E_{t_i}\left[|\Delta W_i|^2\right]
=
h\left(
\mathbb E_{t_i}\left[|\Delta Y_{i+1}|^2\right]
-
|\Delta\overline Y_i|^2
\right).
\end{aligned}
\end{equation}
Hence,
\begin{equation}
\label{p2-eq:conditional_projection_difference}
|\Delta\overline Y_i|^2
+
h|\Delta\widehat Z_i|^2
\leq
\mathbb E_{t_i}\left[|\Delta Y_{i+1}|^2\right].
\end{equation}
Subtracting the two implicit equations in \eqref{p2-eq:two_implicit_backward} gives
\begin{equation}
\Delta\overline Y_i=\Delta Y_i-h\Delta f_i,
\end{equation}
where $ \Delta f_i \coloneqq f(t_i,X_i^{(1)},Y_i^{(1)},\widehat Z_i^{(1)}) - f(t_i,X_i^{(2)},Y_i^{(2)},\widehat Z_i^{(2)}) $.

Substituting into \eqref{p2-eq:conditional_projection_difference} and rearranging, we obtain
\begin{align}
|\Delta Y_i|^2+h|\Delta\widehat Z_i|^2
&\leq
\mathbb E_{t_i}\left[|\Delta Y_{i+1}|^2\right]
+
2h\Delta Y_i\Delta f_i
-
h^2|\Delta f_i|^2
\nonumber\\
&\leq
\mathbb E_{t_i}\left[|\Delta Y_{i+1}|^2\right]
+
2h\Delta Y_i\Delta f_i.
\label{p2-eq:implicit_projection_expansion}
\end{align}
To estimate the product $\Delta Y_i\Delta f_i$, we decompose the driver difference as
\begin{equation}
\begin{aligned}
\Delta f_i
={}&
f(t_i,X_i^{(1)},Y_i^{(1)},\widehat Z_i^{(1)})
-
f(t_i,X_i^{(1)},Y_i^{(2)},\widehat Z_i^{(1)})
\\
&+
f(t_i,X_i^{(1)},Y_i^{(2)},\widehat Z_i^{(1)})
-
f(t_i,X_i^{(2)},Y_i^{(2)},\widehat Z_i^{(1)})
\\
&+
f(t_i,X_i^{(2)},Y_i^{(2)},\widehat Z_i^{(1)})
-
f(t_i,X_i^{(2)},Y_i^{(2)},\widehat Z_i^{(2)}).
\end{aligned}
\end{equation}
For the difference in the $y$-variable, the monotonicity condition \eqref{p2-eq:f_monotonicity} in Assumption \ref{p2-ass:implicit_lipschitz} yields
\begin{equation}
\begin{aligned}
2\Delta Y_i\Big[
&f(t_i,X_i^{(1)},Y_i^{(1)},\widehat Z_i^{(1)})
-
f(t_i,X_i^{(1)},Y_i^{(2)},\widehat Z_i^{(1)})
\Big]
\leq 2k_f|\Delta Y_i|^2.
\end{aligned}
\end{equation}
For the differences in the $x$- and $z$-variables, the Lipschitz bounds and Young's inequality with parameters $\lambda_x$ and $\lambda_z$ give
\begin{equation}
\begin{aligned}
& 2|\Delta Y_i|
\bigl|
f(t_i,X_i^{(1)},Y_i^{(2)},\widehat Z_i^{(1)})
-
f(t_i,X_i^{(2)},Y_i^{(2)},\widehat Z_i^{(1)})
\bigr|
\leq
\lambda_x|\Delta Y_i|^2
+
\frac{f_x}{\lambda_x}|\Delta X_i|^2,
\\
& 2|\Delta Y_i|
\bigl|
f(t_i,X_i^{(2)},Y_i^{(2)},\widehat Z_i^{(1)})
-
f(t_i,X_i^{(2)},Y_i^{(2)},\widehat Z_i^{(2)})
\bigr| 
\leq
\lambda_z|\Delta Y_i|^2
+
\frac{f_z}{\lambda_z}|\Delta\widehat Z_i|^2.
\end{aligned}
\end{equation}
Combining these estimates gives
\begin{equation}
\label{p2-eq:implicit_f_product_explicit}
2\Delta Y_i\Delta f_i
\leq
\theta_0|\Delta Y_i|^2
+
\frac{f_x}{\lambda_x}|\Delta X_i|^2
+
\frac{f_z}{\lambda_z}|\Delta\widehat Z_i|^2.
\end{equation}
Substituting \eqref{p2-eq:implicit_f_product_explicit} into \eqref{p2-eq:implicit_projection_expansion} yields
\begin{equation}
(1-\theta_0h)|\Delta Y_i|^2
+
\left(1-\frac{f_z}{\lambda_z}\right)
h|\Delta\widehat Z_i|^2
\leq
\mathbb E_{t_i}\left[|\Delta Y_{i+1}|^2\right]
+
\frac{f_x}{\lambda_x}h|\Delta X_i|^2.
\end{equation}
Since $1-\theta_0^+h_0>0$, we have $1-\theta_0h\geq1-\theta_0^+h_0>0$ for every $0<h\leq h_0$. Dividing by $1-\theta_0h$ gives
\begin{equation}
\label{p2-eq:implicit_backward_after_dividing}
\begin{aligned}
|\Delta Y_i|^2
+
\frac{1-f_z/\lambda_z}{1-\theta_0h}
h|\Delta\widehat Z_i|^2
\leq
\frac{1}{1-\theta_0h}
\left(
\mathbb E_{t_i}\left[|\Delta Y_{i+1}|^2\right]
+
\frac{f_x}{\lambda_x}h|\Delta X_i|^2
\right).
\end{aligned}
\end{equation}
and we can check that the coefficients satisfy
\begin{equation}
\begin{aligned}
\frac{1}{1-\theta_0h}
&=
1+\frac{\theta_0h}{1-\theta_0h}
\leq
1+\frac{\theta_0^+h}{1-\theta_0^+h_0}
=
1+\theta_2h,
\\
\frac{f_x}{\lambda_x(1-\theta_0h)}
&\leq
\frac{f_x}{\lambda_x(1-\theta_0^+h_0)}
=
\theta_3,
\\
\frac{1-f_z/\lambda_z}{1-\theta_0h}
&\geq
\frac{1-f_z/\lambda_z}{1+|\theta_0|h_0}
=
\theta_1>0.
\end{aligned}
\end{equation}
Applying these bounds to \eqref{p2-eq:implicit_backward_after_dividing} proves \eqref{p2-eq:implicit_backward_one_step}.
\end{proof}

\subsection{Properties of the Gaussian convolution semigroup}\label{p2-app:gaussian_semigroup}

\begin{proof}[Proof of Proposition~\ref{p2-prop:gaussian_semigroup}]

First, $\mathcal G_0^\varsigma$ is the identity operator. If $\phi$ is globally Lipschitz, then
\begin{equation}
|\phi(y)|\leq |\phi(0)|+\operatorname{Lip}(\phi)|y|.
\end{equation}
Since Gaussian random variables have finite first absolute moments, this linear-growth bound implies $\mathbb E|\phi(a+\varsigma\sqrt h\,G)|<\infty$. Thus the expectation defining $\mathcal G_h^\varsigma \phi$ is well defined and finite. The same conclusion holds immediately when $\phi$ is bounded and measurable.

Moreover, if $\phi$ is globally Lipschitz, then so is $\mathcal G_h^\varsigma \phi$, since
\begin{equation}
|(\mathcal G_h^\varsigma \phi)(a)-(\mathcal G_h^\varsigma \phi)(b)|
\leq
\operatorname{Lip}(\phi)|a-b|.
\end{equation}
Thus $\mathcal G_h^\varsigma$ maps globally Lipschitz functions into themselves. If $\phi$ is bounded and measurable, then $\mathcal G_h^\varsigma \phi$ is measurable and bounded, with $ |(\mathcal G_h^\varsigma \phi)(a)| \leq \sup_{y\in\mathbb R}|\phi(y)|. $ Hence, this class of functions is also preserved.

For $h,s>0$, let $G_1,G_2$ be independent standard normal random variables. For either class of input functions, absolute integrability and Fubini's theorem give
\begin{equation}
(\mathcal G_h^\varsigma \mathcal G_s^\varsigma \phi)(a)
=
\mathbb E\left[
\phi\left(
a+\varsigma\sqrt h\,G_1+\varsigma\sqrt s\,G_2
\right)
\right].
\end{equation}
Since $ \sqrt h\,G_1+\sqrt s\,G_2 \sim\mathcal N(0,h+s), $ it follows that $ (\mathcal G_h^\varsigma \mathcal G_s^\varsigma \phi)(a) = (\mathcal G_{h+s}^\varsigma \phi)(a). $ Together with $\mathcal G_0^\varsigma=I$, this proves the semigroup property.

For the remaining properties, let $p_h^\varsigma$ denote the centered Gaussian density with variance $\varsigma^2h$, namely,
\begin{equation}
p_h^\varsigma(z)
\coloneqq
\frac{1}{\sqrt{2\pi\varsigma^2h}}
\exp\left(-\frac{z^2}{2\varsigma^2h}\right).
\end{equation}
By the definition of $\mathcal G_h^\varsigma$, a change of variables, and the symmetry of $p_h^\varsigma$,
\begin{equation}
\label{p2-eq:gaussian_convolution_density}
(\mathcal G_h^\varsigma \phi)(a)
=
\int_{\mathbb R}\phi(y)p_h^\varsigma(y-a)\,\mathrm dy
=
(\phi*p_h^\varsigma)(a).
\end{equation}
Thus $\mathcal G_h^\varsigma \phi$ is the convolution of $\phi$ with a Gaussian density. Several of the properties below are standard consequences of this structure; nevertheless, we provide the details for completeness.

For (\ref{p2-item:gaussian_smoothing}), first suppose that $\phi$ is globally Lipschitz, so that it has at most linear growth. For every integer $m\geq1$, the derivative $\partial_a^m p_h^\varsigma(y-a)$ is a polynomial in $y-a$ multiplied by a Gaussian density. Hence, locally uniformly in $a$,
\begin{equation}
|\phi(y)|
\left|\partial_a^m p_h^\varsigma(y-a)\right|
\end{equation}
is dominated by an integrable function of $y$. The same argument applies when $\phi$ is bounded and measurable, since then $|\phi(y)|$ is bounded by a constant, and no continuity of $\phi$ is required.

The dominated convergence theorem therefore justifies differentiation under the integral sign in both cases, yielding
\begin{equation}
\partial_a^m(\mathcal G_h^\varsigma \phi)(a)
=
\int_{\mathbb R}
\phi(y)\partial_a^m p_h^\varsigma(y-a)\,\mathrm dy,
\qquad m\geq1.
\end{equation}
The same domination, together with the continuity of $\partial_a^m p_h^\varsigma(y-a)$ in $a$, shows that each derivative is continuous. Therefore, $ \mathcal G_h^\varsigma \phi\in C^\infty(\mathbb R) $.

For $\phi\in L^\infty(\mathbb R)$, choose a measurable representative. Since the Gaussian law has a density with respect to Lebesgue measure, changing this representative on a Lebesgue-null set does not change the integral in \eqref{p2-eq:gaussian_convolution_density} for any $a\in\mathbb R$. In particular, we may choose a bounded measurable representative. The differentiation-under-the-integral argument above then shows that $\mathcal G_h^\varsigma \phi\in C^\infty(\mathbb R)$. Moreover, for every $a\in\mathbb R$,
\begin{equation}
|(\mathcal G_h^\varsigma \phi)(a)|
\leq
\int_{\mathbb R}
|\phi(y)|p_h^\varsigma(y-a)\,\mathrm dy
\leq
\|\phi\|_{L^\infty}
\int_{\mathbb R}p_h^\varsigma(y-a)\,\mathrm dy
=
\|\phi\|_{L^\infty}.
\end{equation}
Taking the essential supremum gives $ \|\mathcal G_h^\varsigma \phi\|_{L^\infty} \leq \|\phi\|_{L^\infty}. $

For (\ref{p2-item:gaussian_differentiation}), the differentiation formula established above applies both to globally Lipschitz functions and to bounded measurable functions. Using
\begin{equation}
\partial_a p_h^\varsigma(y-a)
=
\frac{y-a}{\varsigma^2h}p_h^\varsigma(y-a),
\end{equation}
we obtain
\begin{align*}
\partial_a(\mathcal G_h^\varsigma \phi)(a)
=
\int_{\mathbb R}
\phi(y)\partial_a p_h^\varsigma(y-a)\,\mathrm dy
=
\int_{\mathbb R}
\phi(y)\frac{y-a}{\varsigma^2h}
p_h^\varsigma(y-a)\,\mathrm dy
=
\frac{1}{\varsigma\sqrt h}
\mathbb E\left[
\phi(a+\varsigma\sqrt h\,G)G
\right].
\end{align*}
The expectation is absolutely integrable in both cases, by boundedness or linear growth of $\phi$ and the finite Gaussian moments. This proves \eqref{p2-eq:gaussian_first_derivative}. Since $\Delta W=\sqrt h\,G\sim\mathcal N(0,h)$, it follows immediately that
\begin{equation}
\partial_a(\mathcal G_h^\varsigma \phi)(a)
=
\frac{1}{\varsigma h}
\mathbb E\left[
\phi(a+\varsigma\Delta W)\Delta W
\right].
\end{equation}

Finally, for (\ref{p2-item:gaussian_weak_commutation}), suppose that $\phi$ is globally Lipschitz. Then it is locally absolutely continuous and differentiable almost everywhere. Its almost-everywhere derivative coincides with its weak derivative, denoted by $\phi'$, and satisfies
\begin{equation}
|\phi'(y)|
\leq \operatorname{Lip}(\phi)
\qquad\text{for a.e. }y\in\mathbb R.
\end{equation}
Hence
\begin{equation}
\phi'\in L^\infty(\mathbb R),
\qquad
\|\phi'\|_{L^\infty}\leq\operatorname{Lip}(\phi),
\end{equation}
so that $\mathcal G_h^\varsigma \phi'$ is well defined and independent of the chosen measurable representative.

By the absolute continuity of $\phi$, integration by parts on bounded intervals, followed by letting the endpoints tend to infinity, gives
\begin{align*}
(\mathcal G_h^\varsigma \phi')(a)
=
\int_{\mathbb R}
\phi'(y)p_h^\varsigma(y-a)\,\mathrm dy
=
-\int_{\mathbb R}
\phi(y)\partial_y p_h^\varsigma(y-a)\,\mathrm dy.
\end{align*}
The boundary term vanishes because $\phi$ has at most linear growth, whereas the Gaussian density decays exponentially. Moreover, the integrals above are absolutely convergent.

Since $ \partial_a p_h^\varsigma(y-a) = -\partial_y p_h^\varsigma(y-a) $, we obtain
\begin{align*}
(\mathcal G_h^\varsigma \phi')(a)
=
\int_{\mathbb R}
\phi(y)\partial_a p_h^\varsigma(y-a)\,\mathrm dy
=
\partial_a
\int_{\mathbb R}
\phi(y)p_h^\varsigma(y-a)\,\mathrm dy
=
\partial_a(\mathcal G_h^\varsigma \phi)(a),
\end{align*}
where the second equality follows from the differentiation-under-the-integral argument above. This proves \eqref{p2-eq:gaussian_weak_derivative}.

Since $\phi'\in L^\infty(\mathbb R)$, applying (\ref{p2-item:gaussian_smoothing}) and (\ref{p2-item:gaussian_differentiation}) to a bounded measurable representative of $\phi'$ gives $\mathcal G_h^\varsigma \phi'\in C^\infty(\mathbb R)$ and
\begin{equation}
\partial_{aa}(\mathcal G_h^\varsigma \phi)(a)
=
\partial_a(\mathcal G_h^\varsigma \phi')(a)
=
\frac{1}{\varsigma\sqrt h}
\mathbb E\left[
\phi'(a+\varsigma\sqrt h\,G)G
\right],
\qquad a\in\mathbb R.
\end{equation}
This proves \eqref{p2-eq:gaussian_second_derivative_weak} without requiring a classical second derivative of $\phi$.

\end{proof}

\paragraph{Acknowledgment}
The authors would like to thank Balint Negyesi for the many fruitful discussions. The authors also thank the anonymous reviewers for their constructive comments and suggestions. The first author gratefully acknowledges the PhD scholarship awarded by the China Scholarship Council.

\bibliographystyle{alpha}
\bibliography{ref_update}

\end{document}